\documentclass{amsart}

\usepackage{arxiv_math}

\title{Real Heegaard Floer Homology}

\author[Gary Guth]{Gary Guth}
\address{Department of Mathematics\\Stanford University\\
		Building 380\\
		Stanford, California 94305}
\email{gmguth@stanford.edu}

\author[Ciprian Manolescu]{Ciprian Manolescu}
\address{Department of Mathematics\\Stanford University\\
		Building 380\\
		Stanford, California 94305}
\email{cm5@stanford.edu}

\begin{document}
\begin{abstract}
	We define an invariant of three-manifolds with an involution with non-empty fixed point set of codimension $2$; in particular, this applies to double branched covers over knots. Our construction gives the Heegaard Floer analogue of  Li's real monopole Floer homology. It is a special case of a real version of Lagrangian Floer homology, which may be of independent interest to symplectic geometers. The Euler characteristic of the real Heegaard Floer homology is the analogue of Miyazawa's invariant, and can be computed combinatorially for all knots.
\end{abstract}
\maketitle

\section{Introduction}
In recent years, there has been a burst of interest in gauge theoretic invariants of 3- and 4-manifolds equipped with an orientation preserving involution; such manifold-involution pairs have come to be called \emph{real manifolds}. This direction goes back to work of Tian and Wang who introduced and studied \emph{real Seiberg-Witten invariants} for Hermitian almost complex 4-manifolds \cite{tian-wang}. Work of Nakamura considered $\Pin^-(2)$-equivariant homotopy refinements of the Seiberg-Witten equations and gave applications concerning the intersection form with local coefficients of 4-manifolds and established some nonvanishing results \cite{nakamura1,nakamura2}. More recently, Kato used this theory to detect non-smoothable actions of $\Z/2\times \Z/2$ on spin 4-manifolds \cite{kato}. Konno, Miyazawa, and Taniguchi developed a real Seiberg-Witten Floer cohomology, as well as a Floer K-theory, for links via their branched covers \cite{KMT:Ktheory,KMT:homology}. Li has been developing a real monopole Floer homology theory for real 3- and 4-manifolds, as well as for webs and foams, and has produced a spectral sequence from Khovanov homology to his theory \cite{li:HMR,li:skein,li:spectral}. Baraglia and Hekmati \cite{baraglia-hekmati} have also studied equivariant Seiberg-Witten theory of rational homology spheres with an involution. Striking recent applications include Kang, Park, and Taniguchi's use of real Fr{\o}yshov invariants to prove that the $(4n+2,1)$-cables of the figure eight knot are not slice \cite{KPT:cables}, and Miyazawa's proof of the existence of an infinite family of exotic $\RP^2$s in $S^4$ \cite{miyazawa}. The latter led to a proof of the existence of exotic involutions on $\mathbb{CP}^2$ \cite{hugheskimmiller_branched}, and further generalizations were explored by Baraglia \cite{baraglia_miyazawa}. 

Seiberg-Witten theory has a more computable symplectic counterpart, Heegaard Floer homology \cite{os_holodisks, os_properties_apps, os_holotri}. In light of the interesting results coming from real Seiberg-Witten theory, it seems natural to formulate a Heegaard Floer theoretic analogue. In the Seiberg-Witten setting, a choice of an anti-linear involution on the spinor bundle which covers the involution of the underlying manifold induces a real structure on the associated Seiberg-Witten configuration space. In the Heegaard Floer setting, the analogue is to work with \emph{real Heegaard diagrams}, in which the alpha and beta curves are switched by the involution; these have been studied by Nagase \cite{nagase}. In this situation, we have an involution on the symmetric product of the Heegaard surface which switches the alpha and beta tori. Roughly, the goal is then to   count \emph{invariant} pseudo-holomorphic strips. 

Our theory is a special case of what one might call \emph{real Lagrangian Floer homology}. While real Gromov-Witten theory is well-developed (see, for instance, \cite{welschinger}, \cite{solomon}, \cite{cho}, \cite{tehrani}, \cite{Georgieva}, \cite{GeorgievaZinger}), the focus in the literature has primarily been on counting real $J$-holomorphic spheres (or higher genus curves), some of which correspond to open curves with boundary on the fixed point set $M^R$. Our approach is slightly different. Given a symplectic manifold $(M, \omega)$ equipped with an anti-symplectic involution, $R$, and a spherically monotone Lagrangian submanifold, $L$, we consider the usual Lagrangian Floer homology $\HF(L, M^R)$. Strips with one boundary on $L$ and the other on $M^R$ are in correspondence with $R$-symmetric strips with one boundary on $L$ and with the other boundary on its reflection $R(L)$. 
 
In the present article, we consider closed, oriented 3-manifolds equipped with an involution, such that the fixed point locus is non-empty and of codimension 2. We represent such a manifold by a real pointed Heegaard diagram $(\cH, w)$. To this data, we associate complexes:
\begin{align*}
    \CFR^-(\cH, w), \quad \CFR^\infty(\cH, w), \quad \CFR^+(\cH, w), \quad \widehat{\CFR}(\cH, w).\end{align*}
over the polynomial ring $\F[U]$ where $\F$ is the field of two elements. We denote the respective homologies by $\HFR^\circ(Y, \tau, w)$, for $\circ \in \{-, \infty, +, \hat{\phantom{o}} \}$. These can be equipped with an action of $\Lambda^*\cH_1(Y)$, where $\cH_1(Y,\tau)=(H_1(Y)^{-\tau_*}/\tors)/(H_1(Y)^{-\tau_*}/\tors)^{\tau_*}$.

\begin{theorem}\label{thm:invariant}
    The isomorphism classes of the $\F[U]\otimes \Lambda^*(\cH_1(Y, \tau))$-modules $\HFR^\circ(Y, \tau, w)$ are topological invariants of the underlying pointed real 3-manifold $(Y, \tau, w)$. 
\end{theorem}

As in the monopole setting, these groups split over real $\SpinC$-structures on $(Y, \tau)$. There is a real Maslov index, which gives rise to relative gradings. The variable $U$ has degree $-1$ (and corresponds to $v_i$ from \cite{li:HMR}).

For a knot  $K \subset S^3$, one can consider the double branched cover $\Sigma_2(K)$ with the involution $\tau_K$ given by the deck transformation. For double branched covers over knots in $S^3$, every $\SpinC$-structure carries a unique real $\SpinC$-structure, and their total number is the determinant of the knot, $\det(K)$. 
In the case that $|\Delta_{K}(-1)| = 1$, the manifold $\Sigma_2(K)$ has a unique $\SpinC$-structure, and Miyazawa defined an invariant $|\deg(K)|$ which is the absolute value of the signed count of points in the framed moduli space of solutions of the real Seiberg-Witten equations; see \cite{miyazawa} and \cite[Theorem 1.3]{li:spectral}. Miyazawa also associated invariants to surfaces inside $4$-manifolds, and showed that the invariant of the roll spun knot of $K$ (which is a $2$-knot in $S^4$) equals $\absdeg$. Using this, he showed that by taking the connected sum of a standard $\RP^2$ in $S^4$ with several roll spins of the pretzel knot $P(-2,3,7)$, one gets an infinite family of exotic $\RP^2$s in $S^4$. The Miyazawa invariant is applicable to surfaces in $S^4$ of small genus, and there is hope that it could be used to detect an exotic $2$-knot. 

In the case that $|\Delta_{K}(-1)| = 1$, Miyazawa's degree invariant ought to correspond to the Euler characteristic of $\widehat{\HFR}(\Sigma_2(K), \tau_K)$. More generally, for any knot $K$ and any $\s \in \SpinC(\Sigma_2(K))$, we  can consider the Euler characteristic
$$\chi_\s(K):=\chi(\HFRh(\Sigma_2(K),\tau_K,\s)).$$ 

 We have not developed absolute $\Z/2$-gradings in our theory, so a priori $\chi_\s(K)$ is well-defined only up to sign. However, in our setting we can talk about a global change in signs, for all $\SpinC$ structures at the same time. This gives more information, which we do not know how to see from Seiberg-Witten theory. Precisely, we make the global change so that $$\chitot(K) := \sum_\s \chi_\s(K)$$
is positive. This fixes the signs of all  $\chi_\s(K) \in \Z$. (Further, one can show that the invariants $\chi_\s(K)$ are always odd integers.)

Work of Kang, Park, and Taniguchi computed the Miyazawa invariant for Montesinos knots \cite[Corollary 1.6]{KPT:cables}, via the techniques of \cite{dai2023latticehomologyseibergwittenfloerspectra}. The Heegaard Floer analogue is much more straightforward to compute. 

\begin{theorem}
    The invariants $\chi_\s(K)$ (the Heegaard Floer analogues of Miyazawa's invariant) can be computed algorithmically for any knot $K \subset S^3$. 
\end{theorem}

Typically, the Euler characteristic of a Floer homology is a more classical invariant (e.g., the Casson invariant for instanton homology, Turaev torsion for monopole or Heegaard Floer homology, the Alexander polynomial for knot Floer homology). The Miyazawa-type invariants appear to be new, and have unusual properties. For all knots with up to $8$ crossings, we have $\chi_\s(K) \in \{\pm 1\}$ for all $\s$. However, this stops being true for larger knots: there exist $10$-crossing knots which have $\chi_s(K) = 3$ for some $\s$. For the pretzel knot $P(-2,3,7)=12n_{242}$, which has determinant $1$ (so a single $\s$), we have $\chitot(K) = 3$, in agreement with the calculation of $\absdeg$ from Seiberg-Witten theory \cite{miyazawa}. Further, even for small knots such as the trefoil, summing over $\SpinC$ structures yields a value of $\chitot(K)$ that is different from $\det(K) = \chi(\widehat{\HF}(\Sigma_2(K))$. 

Apart from the Euler characteristic, the groups $\widehat{\HFR}(\Sigma_2(K), \tau_K)$ themselves can be computed effectively in many cases, by counting holomorphic representatives of invariant domains. As an example, in \Cref{sec:exs} we will do this explicitly for the knot $9_{46}$. 

Nevertheless, the current paper only sets up the definition and the basic properties of real Heegaard Floer theory; there is much more work to be done. We expect the invariants $\HFR^\circ$ to be natural, and functorial under four-dimensional cobordisms equipped with involutions. When applied to double branched covers, they should produce invariants of surfaces in four-manifolds that are more easily computable than via Seiberg-Witten theory. Moreover, we should get absolute gradings,  and Fr{\o}yshov invariants similar to those in \cite{KMT:homology} and \cite{li:skein}. In some cases, the theory could be extended to $\Z$ instead of $\F$ coefficients. We also expect connections to Khovanov homology and to the usual Heegaard Floer groups $\HF^\circ(\Sigma_2(K))$.

An orientation-preserving involution on a connected $3$-manifold is either the identity or has a (possibly empty) link $C$ as its fixed point set. In this paper we only consider the case where $C$ is nonempty; then, we can place the basepoint on a component of $C$. In the case of free actions, basepoints must be placed in pairs, and the invariant appears to depend not only on the pair $(Y, \tau)$, but also on the link determined by the basepoints.

\subsection*{Outline.} In  \Cref{sec:real_lf}, we study the general case of real Lagrangian Floer homology; we establish transversality for generic symmetric almost complex structures, compute the real Maslov index in terms of the usual Maslov index, and consider gradings for the theory. In \Cref{sec: real HD}, we discuss real Heegaard diagrams, real Heegaard moves, as well as real $\SpinC$-structures. In \Cref{sec:def}, we study bubbling and define our invariants. In \Cref{sec:invariance} we prove invariance of the theory. In \Cref{sec:exs} we compute some examples. Finally, in \Cref{sec:euler}, we consider the Euler characteristic of the hat theory for knots in $S^3$. 

\bigskip
\noindent {\bf Acknowledgements.} We thank Mohammed Abouzaid, Ciprian Bonciocat, Kristen Hendricks, Eleny Ionel, Sungkyung Kang, Judson Kuhrman, Jiakai Li, Robert Lipshitz, Anubhav Mukherjee, Eha Srivastava, Mohan Swaminathan, Masaki Taniguchi, and  Yonghan Xiao for helpful conversations. We thank the referees for comments on a previous version of the paper.

The authors were supported by the Simons Collaboration Grant on New Structures in Low-Dimensional Topology. CM was also supported by a Simons Investigator Award.

\section{Real Lagrangian Floer homology}\label{sec:real_lf}

\subsection{The set-up}
\label{sec:setup}
Let $(M, \omega)$ be a symplectic manifold of dimension $2n$. Suppose we are given an anti-symplectic involution $R$, i.e. a smooth map $R: M \to M$ such that $R^2 = \id$ and $R_*(\omega) = -\omega$. We let $M^R$ denote the fixed point set of $R$, which is necessarily a Lagrangian submanifold of $M$. Indeed, let $(\C^n, \omega_0)$ be the standard symplectic structure on $\C^n$. A real version of the Darboux theorem \cite[Lemma 2.3]{cohomology_top_groups} implies that any point $p \in M^R$ has a neighborhood $U$ preserved by $R$,  with a chart $\varphi: (U, \omega) \to (\C^n, \omega_0)$ such that $\varphi^* \omega_0 = \omega$ and $ \varphi(R(z)) = \overline{\varphi(x)}$. Therefore, $\varphi^{-1}(\R^n) = M^R \cap U$, from which it follows that $M^R$ is a Lagrangian submanifold. 

Let $L \subset M$ be a compact Lagrangian that is {\em monotone on disks}; i.e., there exists $\lambda > 0$ such that
\begin{equation}
    \label{eq:monotonedisks}
 [\omega]|_{\pi_2(M, L)} = \lambda \cdot \mu|_{\pi_2(M, L)},
 \end{equation}
where $\mu$ is the Maslov class. Note that this implies that $M$ itself is {\em spherically monotone}:
 \begin{equation}
    \label{eq:monotonespheres}[\omega]|_{\pi_2(M)} = \frac{\lambda}{2} \cdot  c_1|_{\pi_2(M)}.
\end{equation}
In turn, this implies that the Lagrangian $M^R$ is monotone on disks, because every disk with boundary on $M^R$ can be combined with its reflection to give a sphere in $M$, and we can apply the relation \eqref{eq:monotonespheres}.

The {\em minimal Maslov number} of a Lagrangian $L$ is defined as the positive integer $N_L$ such that $\mu(\pi_2(M, L)) = N_L \Z$. If $L$ is orientable, then $N_L$ is even (in particular, at least $2$). In the non-orientable case, the minimal Maslov number could be $1$.

We assume that $M$ is either compact or convex at infinity, so that Gromov compactness applies. We will also assume that the counts of index $2$ disk bubbles (through a given point) on $L$ and $M^R$ are zero, or that they cancel out. (When $L$ or $M^R$ has minimal Maslov number $1$, we also need to assume that there are no contributions from index $1$ disk bubbles. This makes the index $2$ disk bubble counts well-defined.)

Under these hypotheses, there is a well-defined Lagrangian Floer complex
$$\CFR_*(L) := \CF_*(L, M^R),$$
which we call the {\em real Lagrangian Floer complex} associated to $L$. We will be interested in interpreting it in terms of certain symmetric $J$-holomorphic strips between $L$ and its reflection $R(L)$, instead of ordinary $J$-holomorphic strips between $L$ and $M^R$. 

A generator of $\CF_*(L, M^R)$ is an intersection point between $L$ and $M_R$. This  can instead be viewed as an element of $L \cap R(L)$ which is fixed by the action of $R$:
$$ L \cap M^R = (L \cap R(L))^R.$$
Thus, the generators of $\CF_*(L, M^R)$ are precisely the generators of $\CF_*(L, R(L))$ that are fixed by the involution. 

 Let $\D = \R \times [0,1]$ be the complex strip with coordinates $(s,t)$ and let $\rho: \D \ra \D$ be the involution 
\begin{align*}
    \rho(s, t) = \rho(s, 1-t).
\end{align*}
Under the identification of $\D$ with the unit disk in $\C$, the map $\rho$ corresponds to complex conjugation. Let $\cJ$ be the space of time-dependent almost complex structures on $M$ which are compatible with $\omega$ and let $\cJ_R$ be the subset of $\cJ$ which satisfies the additional symmetry condition:
\begin{equation}
\label{eq:JR}
    J_t \circ R_* = - R_* \circ J_{1-t}.
\end{equation}
This space is nonempty and contractible; see \cite[Proposition 1.1]{welschinger} for a proof in the time-independent case, which can be easily adapted to our setting. 

For $\x$ and $\y$ in $L \cap M^R$, define $\Pi_2(\x, \y)$ to be the set of homotopy classes of maps 
\begin{align*}
    \Pi_2(\x, \y) = 
    \left\{ u: \R \times [0,1] \ra M \bigg | 
    \substack{u(s,0) \in L,\, u(s,1) \in M^R\\
    \lim_{s\ra -\infty} u(s,t) = \x \\
    \lim_{s\ra\infty} u(s,t) = \y 
    } \right\}.
\end{align*}
Given $\psi \in \Pi_2(\x, \y)$ and $J \in \cJ$, let $\cM(\psi)$ be the space of strips $u$ in the class $\psi$ which satisfy Floer's equation: 
\begin{equation}
\label{eq:floer}
    \frac{\del u}{\del s} + J_t(u(s, t))  \frac{\del u}{\del t} = 0.
\end{equation}
This space has an action by $\R$ translation on the domain. We let $\widehat{\cM}(\psi) = \cM(\psi)/\R$ be the quotient. 

Given $\x,\y \in (L \cap R(L))^R$, define $\pi_2^R(\x,\y)$ to be the set of homotopy classes of maps 
\begin{align*}
    \pi_2^R(\x,\y) = 
    \left\{ u: \R \times [0,1] \ra M \bigg | 
    \substack{u(s,0) \in L,\, u(s,1) \in R(L)\\
    \lim_{s\ra -\infty} u(s,t) = \x \\
    \lim_{s\ra\infty} u(s,t) = \y \\
    u = R \circ u \circ \rho
    } \right\}.
\end{align*}
Elements of $\pi_2^R(\x,\y)$ will be called \emph{real invariant strips}. Given $\phi\in \pi_2^R(\x,\y)$ and $J \in \cJ^R$, define $\cM_R(\phi)$ to be the moduli space of real invariant strips $u$ which satisfy \Cref{eq:floer}. We let $\widehat{\cM}_R(\phi)$ be its quotient by $\R$.

Given $\x,\y \in (L\cap R(L))^R$ and $J \in \cJ_R$, there is a map
\begin{align}\label{eqn:real_correspondence}
    \cR: \Pi_2(\x,\y) \ra \pi_2^R(\x,\y)
\end{align}
given by gluing together $\psi$ and $R\circ \psi \circ \rho$. This map has a clear inverse, given by restriction. Furthermore, given a $J$-holomorphic representative $u: (\D, \del \D)  \ra (M, L\cup M^R)$, the strips $u$ and $R\circ u\circ \rho$ can be glued together to form a $J$-holomorphic map $(\D, \del \D) \ra (L, R(L))$. Hence, there is a natural correspondence between elements of $\widehat{\cM}(\psi)$ and elements of $\widehat{\cM}_R(\phi)$ where $\phi = \cR(\psi)$.

In light of this correspondence, we will often conflate $\Pi_2(\x, \y)$ with $\pi_2^R(\x, \y)$, and identify the corresponding moduli spaces. We will mostly work with $\pi_2^R(\x, \y)$ and $\widehat{\cM}_R(\phi)$.

\subsection{Transversality}\label{subsec:transversality}
Given a class $\phi \in \pi_2^R(\x, \y)$, the moduli spaces $\widehat{\cM}_R(\phi)$ are transversely cut out for a generic $J \in \cJ$. If $J$ lies in $\JR$, we may identify $\widehat{\cM}_R(\phi)$ with the moduli space of real invariant strips $\widehat{\cM}_R(\phi)$. However, a priori,  it is not clear whether $\widehat{\cM}_R(\phi)$ is transversely cut our for a generic $J$ in $\cJ_R$. 

For every (not necessarily $J$-holomorphic) strip $u$ in the class $\phi$, the linearization of Floer's equation \eqref{eq:floer} is an operator
$$ D_u^R : \Gamma(u^*TM)^R \to \Gamma(\Lambda^{0,1}T^*\D \otimes_{J} u^*TM)^R,$$
where the superscript $R$ refers to the following requirement: We are considering sections of bundles on $\D$, and we have a bundle map $R_*$ that covers the conjugation $\rho$ on $\D$. We ask that the sections be invariant with respect to $R_*$.

The operator $D_u^R$ is (a relative version of) an elliptic Real operator in the sense of \cite{ASReal}. From the fact that $D_u^R$ is the restriction of the ordinary Fredholm linearization of Floer's equation, we can deduce that $D_u^R$ is also Fredholm (when extended between appropriate Banach completions). Its index is denoted
$$\indR(\phi) = \dim \ker (D^R_u) - \dim \coker(D^R_u).$$
The index is invariant under Fredholm deformation, so it depends only on the class $\phi$ (rather than on the specific $u$). This justifies the notation $\indR(\phi)$.

\begin{remark}
The real index is the same as the ordinary index of the corresponding class $\mathcal{R}^{-1}(\phi) \in \Pi_2(\x, \y) $ between $\Ta$ and $M^R$.    
\end{remark}

\begin{proposition}\label{prop:transversality}
For generic $J \in \JR$, the moduli space $\cMR(\phi)$ is a smooth manifold of dimension $\indR(\phi)$.
\end{proposition}

\begin{proof}
The proof is similar to that of the usual transversality result in Lagrangian Floer homology; see \cite{FHS}, \cite{oh} or \cite{McDuffSalamonBig}. Let us sketch the argument and point out the necessary modifications. We will work with Banach completions throughout but we ignore them from notation for simplicity. 

Consider the bundle $\End(TM, J, \omega)$ whose fiber at $p \in M$ consists of a smooth family of linear maps $Y_t: T_pM \to T_pM$, $t \in [0,1]$ such that
$$Y_tJ_t + J_tY_t = 0, \ \ \omega(Y_tv, w) + \omega(v, Y_tw)=0.$$

 A variation  $Y=\delta J$ of the data $J \in \JR$ is a section of $\End(TM, J, \omega)$ that satisfies 
\begin{equation}
\label{eq:Ycondition}
    Y_t \circ R_* = - R_* \circ Y_{1-t}
\end{equation}
for all $t \in [0,1]$.  The space of such variations is the tangent space $T_J \JR$.

Denote by $\bar\del_J (u)$ the left hand side of Floer's equation \eqref{eq:floer}. By the Sard-Smale theorem, to ensure transversality for generic $J$, it suffices to prove that the map $(u, J) \to \bar\del_J (u)$ has surjective linearization whenever its value is $0$. A standard argument shows that if this were not the case, we would find $\eta$ in the cokernel of $D_u^*$, such that 
\begin{equation}
\label{eq:intY}
    \int_{\D} \langle \eta, Y(u) \circ du \circ j \rangle =0
\end{equation}
for all $Y \in T_J \JR$. (Here, $j$ is the complex conjugation on $T\D$.)

By \cite[Theorem 5.1]{oh}, there is a dense set of points $(s, t) \in \D$ such that
$$du(s, t) \neq 0, \ \ u(s, t) \cap u(\R -\{s\}, t) =\emptyset.$$
Given such an $(s, t)$, we can find $Z$ in the fiber of $\End(TM, J, \omega)$ over $(s, t)$ such that
$$ \langle \eta(s, t), Z \circ du(s, t) \circ j(s, t) \rangle > 0.$$

Using a smooth cut-off function, we can extend $Z$ to a section $Y$ supported in a neighborhood of $(s, t)$, such that the integral in \eqref{eq:intY} is positive. This section may not satisfy \eqref{eq:Ycondition}. To make it so, we average $Y$ with its conjugate $\bar Y$ given by
$$\bar Y_t = -R_* \circ Y_{1-t} \circ R_*.$$  
Replacing $Y$ by this average does not change the integral, because of the symmetry property of $u$. Thus, we obtain a contradiction with \eqref{eq:intY}.
\end{proof}

\subsection{The real Maslov index}\label{subsec:real_maslov}
For $\x, \y \in (L \cap R(L))^R$, there is a forgetful map
$$ \pi^R_2(\x, \y) \to \pi_2(\x, \y), \ \ \phi \to \tphi$$
where $\pi_2(\x, \y)$ denotes the space of homotopy classes of all strips from $\x$ to $\y$ between $L$ and $R(L)$ (not necessarily $R$-invariant).

The ordinary moduli space $\cM(\tphi)$ of $J$-holomorphic curves (without the $R$-invariance condition) has virtual dimension equal to the Maslov index 
$$ \ind(\tphi) = \dim \ker (D_u) - \dim \coker(D_u),$$
where $u$ is any strip in the class  $\tphi$ ($J$-holomoprhic or not). Similarly, the virtual dimension of the moduli space $\cMR(\phi)$ is given by the real index
$$\indR(\phi) = \dim \ker (D^R_u) - \dim \coker(D^R_u),$$
where now $u$ is $R$-invariant. 

There is a formula which enables us to compute the real index from the usual one.
Before stating it, we define a quantity $\sigma(L, \x)$ associated to the Lagrangian $L$ and each intersection point $\x \in (L \cap R(L))^R$. Since $L$ and $R(L)$ intersect transversely at $\x$, they must also intersect transversely with $M^R$ there. Thus, inside $T_{\x}M$ we have three linear Lagrangian subspaces that are pairwise intersecting transversely: $T_\x L$, $T_\x (R(L))$, and $T_\x(M^R)=(T_\x M)^R$. 
We define $\sigma(L, \x)$ to be the Kashiwara-Wall index (also known as the triple Maslov index) associated to these three subspaces:
$$\sigma(L, \x) = s(T_\x L, T_\x (R(L)), T_\x M^R).$$
For the definition of the triple index $s$, see \cite{wall:nas,lion-vergne,CLM,deGosson}. In our setting, it can be interpreted as follows: we let $T_\x L$ be the graph of a symmetric linear function $$F: T_\x M^R \to J \cdot T_\x M^R $$
and let $\sigma(L, \x)$ be the signature of $F$. Note that $F$ is nonsingular if the three Lagrangians intersect transversely (which as assume to be the case in our setting) and that $\sigma(L, \x) \equiv n \pmod{2}$, where $n= \dim M^R$. 

\begin{prop}
\label{prop:indR}
    Let $\phi \in \pi_2^R(\x, \y)$ and $\tphi$ its image in $\pi_2(\x, \y)$. Then:
    \begin{equation}
    \label{eq:indR}
        \ind(\tphi) = 2 \indR(\phi) + \frac{\sigma(L, \x) - \sigma(L, \y)}{2}.
    \end{equation}
\end{prop}

\begin{proof}
    Observe that $\indR$ is invariant under perturbations of the Fredholm operator. Thus, after trivializing $u^*TM$ over the disk, we can think of $\indR$ as associated to a path $\lambda = (\Lambda_t)_{t \in [0,1]}$ of Lagrangian subspaces of $\C^n$, starting and ending at Lagrangians transverse to their conjugates; that is,
    $$ \Lambda_i \pitchfork \bar{\Lambda}_i, \ \ i=0,1.$$
    (In our case, $\Lambda_0 = T_{\x}L$ and  $\Lambda_1 = T_{\y} L$.) 

We have both an ordinary Maslov index $\mu(\lambda)=\ind(\tphi)$ and a real Maslov index $\mu_R(\lambda)=\indR(\phi)$ associated to the path $\lambda$. Consider the quantity
$$ \delta(\lambda)= 2\mu_R(\lambda)-\mu(\lambda).$$

If the endpoints $\Lambda_0$ and $\Lambda_1$ are kept constant, there is a $\Z$ choice of possible paths $\lambda$, corresponding to the fundamental group of the Lagrangian Grassmannian $\pi_1(U(n)/O(n))=\Z$. Concatenating with a loop in the generator of $\Z$ adds $2$ to $\mu(\lambda)$, because we get a contribution of $1$ to the Maslov index for both paths $(\Lambda_t)$ and $(\bar{\Lambda}_t)$. (For the definition and properties of the Maslov index for paths, we refer to \cite{viterbo-index}, \cite{robbin-salamon}, \cite{robbin-salamon2}.) On the other hand, the contribution to the real index $\mu_R(\lambda)$ is $1$, because we divide out by the symmetry. Therefore, $\delta(\lambda)$ stays constant under concatenation with a loop. It follows that $\delta$ depends only on the endpoints, so we can write:
$$\delta(\lambda) = \delta(\Lambda_0, \Lambda_1).$$

Furthermore, both $\ind$ and $\indR$ are additive under concatenation. Therefore,
\begin{equation}
    \label{eq:concat}
    \delta(\Lambda_0,\Lambda_2) = \delta(\Lambda_0,\Lambda_1) + \delta(\Lambda_1,\Lambda_2).
\end{equation}

If $\Lambda \subset \C^n$ is a Lagrangian transverse to $\bar \Lambda$, it must be of the form 
$$\Lambda=\{(x, iAx) \mid x \in \R^n\} \subset \C^n= \R^n \oplus i\R^n,$$
where $A: \R^n \to \R^n$ is nonsingular and self-adjoint, i.e. representable by a symmetric matrix. We let $\sigma(L)$ be the signature of that matrix, and write $\Lambda = \Gamma(A)$.

In particular, let $\Delta = \Gamma(I)$ be the diagonal (associated to the identity matrix). 
Setting $f(\Lambda) = \delta(\Lambda, \Delta)$, we deduce from \eqref{eq:concat} that $\delta$ is of the form
$$\delta(\Lambda_0, \Lambda_1) = f(\Lambda_0) - f(\Lambda_1).$$

We claim that
\begin{equation}
\label{eq:fLambda}
 f(\Lambda) = \frac{n-\sigma(\Lambda)}{2}.
\end{equation}

To prove this, note that the ordinary and real indices, and hence $f$, are invariant under deformations. The connected components of the space of symmetric $n$-by-$n$ matrices are characterized by the signature, so it suffices to compute $f(\Lambda)$ when $\Lambda=\Gamma(A)$ and $A$ is diagonal with entries only $1$ and $-1$. 

Furthermore, the indices and hence $f$ are additive under direct sums; so is the right hand side of \eqref{eq:fLambda}. Therefore, it suffices to check \eqref{eq:fLambda} for the graphs of $I$ and $-I$ when $n=1$. This can be done explicitly. To compute $f(\Delta)=\delta(\Delta, \Delta)$, we can take $\lambda$ to be the constant path, in which case both the index and the real index are zero; we get $f(\Delta)=0$. To compute $f(\nabla)=\delta(\nabla, \Delta)$ where $\nabla =\Gamma(-I)$ is the anti-diagonal, we let the path $\lambda=(\Lambda_t)$ consist of the lines of slope $-1+2t$ in $\R^2$, for $t \in [0,1]$. Then, the class $\phi$ is the bigon shown in Figure~\ref{fig:bigon}. The bigon has a unique holomorphic representative   modulo translation by $\R$, and the moduli space is transversely cut out; thus, $\ind(\tilde \phi)=1$. The same goes for the invariant moduli space, so we also have $\indR(\phi)=1$. We get that $f(\nabla)= \delta(\lambda) = 1$, and therefore 
 Equation~\eqref{eq:fLambda} holds in this case, too. We deduce that it holds in general.

    The formula \eqref{eq:indR} easily follows from \eqref{eq:fLambda}, because we have $\sigma(L, \x) = \sigma(T_{\x} L_0)$ and $\sigma(L, \y) = \sigma(T_{\y} L_0)$.
\end{proof}

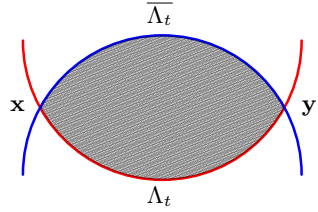
\begin{figure}[h]
\def\svgwidth{.3\linewidth}
\begingroup%
  \makeatletter%
  \providecommand\color[2][]{%
    \errmessage{(Inkscape) Color is used for the text in Inkscape, but the package 'color.sty' is not loaded}%
    \renewcommand\color[2][]{}%
  }%
  \providecommand\transparent[1]{%
    \errmessage{(Inkscape) Transparency is used (non-zero) for the text in Inkscape, but the package 'transparent.sty' is not loaded}%
    \renewcommand\transparent[1]{}%
  }%
  \providecommand\rotatebox[2]{#2}%
  \newcommand*\fsize{\dimexpr\f@size pt\relax}%
  \newcommand*\lineheight[1]{\fontsize{\fsize}{#1\fsize}\selectfont}%
  \ifx\svgwidth\undefined%
    \setlength{\unitlength}{226.77165354bp}%
    \ifx\svgscale\undefined%
      \relax%
    \else%
      \setlength{\unitlength}{\unitlength * \real{\svgscale}}%
    \fi%
  \else%
    \setlength{\unitlength}{\svgwidth}%
  \fi%
  \global\let\svgwidth\undefined%
  \global\let\svgscale\undefined%
  \makeatother%
  \begin{picture}(1,0.7)%
    \lineheight{1}%
    \setlength\tabcolsep{0pt}%
    \put(0,0){\includegraphics[width=\unitlength,page=1]{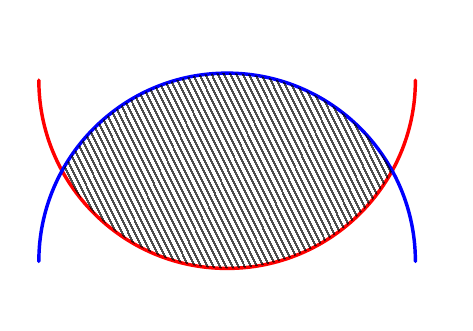}}%
    \put(0.43656178,0.06534767){\color[rgb]{0,0,0}\makebox(0,0)[lt]{\smash{\begin{tabular}[t]{l}{\small$\Lambda_t$}\end{tabular}}}}%
    \put(0.43656178,0.58128519){\color[rgb]{0,0,0}\makebox(0,0)[lt]{\smash{\begin{tabular}[t]{l}{\small$\overline{\Lambda_t}$}\end{tabular}}}}%
    \put(0.87947311,0.32829354){\color[rgb]{0,0,0}\makebox(0,0)[lt]{\smash{\begin{tabular}[t]{l}{\small$\y$}\end{tabular}}}}%
    \put(0.04603557,0.32829354){\color[rgb]{0,0,0}\makebox(0,0)[lt]{\smash{\begin{tabular}[t]{l}{\small$\x$}\end{tabular}}}}%
  \end{picture}%
\endgroup%

\caption{A bigon domain from a path from $\Lambda_0 = \Gamma(-I)$ to $\Lambda_1=\Gamma(I)$.}\label{fig:bigon}
\end{figure}

\begin{remark}
\label{rem:et}
    In some situations we may have $\indR(\phi) > \ind(\phi)$ (see Example~\ref{ex:annulus}). In such a case, note that there is no $J \in \JR(\phi)$ for which we can obtain transversality for both the space of ordinary $J$-holomorphic curves between $L$ and $R(L)$, and the space of $R$-invariant ones. Indeed, if such a $J$ existed, the cokernels of the two linearized operators would be zero, and we would have
    $$ \indR(\phi)=\dim \ker (D^R_u) \leq \dim \ker(D_u) = \ind(\phi).$$
\end{remark}

\subsection{The real Lagrangian Floer complex}
We can now define the Floer complex $\CFR_*(L)=\CF_*(L, M^R)$ in terms of invariant holomorphic strips. For simplicity, we work with coefficients in $\F=\Z/2\Z$.

We let $\CFR_*(L)$ be generated over $\F$ by intersection points in $(L\cap R(L))^R$, and equipped with the differential
$$\del \x = \sum_{\y \in (L\cap R(L))^R} \sum_{\substack{\phi \in \pi_2^R(\x, \y) \\ \indR(\phi)=1}} \# \widehat{\cM}_R(\phi) \cdot \y.$$
Here, $\# \widehat{\cM}_R(\phi) \in \F$ is the count of real invariant holomorphic strips, with respect to a generic $J \in \JR$. When we want to emphasize the dependence on $J$, we write $\CFR_*(L, J)$.

We denote the homology of $\CFR_*(L)$ by $\HFR_*(L)$.

\begin{lem}\label{lem:independence of acs}
    The real Lagrangian Floer homology $\HFR_*(L)$ is independent of the choice of invariant almost complex structure (up to isomorphism). 
\end{lem}
\begin{proof}
    The usual continuation arguments from Floer theory apply. We sketch them here for convenience.
    
    Since $\JR$ is contractible, any two choices of symmetric almost complex structures $J_0$ and $J_1$ can be connected by a path $J_s$ (which is a two-parameter family of almost complex structures.) Pick a generic such path. There is a map 
    \begin{align*}
        \Phi_{J_s}: \CFR_*(L, J_0) \ra \CFR_*(L, J_1), 
    \end{align*}
    defined by 
    \begin{align*}
        \Phi_{J_s}(\x) = \sum_{\y} \sum_{\substack{\phi \in \pi_2(\x,\y)\\ \ind_R(\phi) = 0}} \# \cM_R^{J_s}(\phi) \y,
    \end{align*}
    where $\cM_R^{J_s}(\phi)$ is the moduli space of holomorphic strips with time-dependent almost complex structures on the codomain. By considering the ends of one-dimensional moduli spaces, one can deduce that $\Phi_{J_s}$ is a chain map. 

    To see that $\Phi_{J_s}$ is a homotopy equivalence, choose a path $J_{s, T}$ of 2-parameter families of almost complex structures, so that $J_{s,0}$ is $J_s*J_{1-s}$ and $J_{s,1}$ is independent of $s$. We may assemble the spaces $\cM^{J_{s,T}}_R(\phi)$ for $T \in [0,1]$ into a single moduli space, and define 
    \begin{align*}
        H(\x) = \sum_{\y} \sum_{\substack{\phi \in \pi_2(\x,\y)\\ \ind_R(\phi) = -1}} \# \cM_R^{J_{s,T}}(\phi) \y.
    \end{align*}
    By once again examining the ends of $\cM_R^{J_{s,T}}(\phi)$ for curves with $\ind_R(\phi) = 0$, it follows that $H$ defines a homotopy between $\Phi_{J_s}\circ \Phi_{J_{1-s}}$ and the identity. 
\end{proof}

\subsection{Gradings} \label{sec:gradings}
If the original Lagrangian Floer complex $\CF_*(L, R(L))$ has a (relative or absolute) $\Z/2N$-grading $\gr$, then the real Lagrangian Floer complex $\CFR_*(L)$ has a (relative or absolute) $\Z/N$-grading $\gr_R$. Indeed, in view of Proposition~\ref{prop:indR}, we can set
\begin{align}\label{eqn:CFR_grading}
    \gr_R(\x) := \frac{1}{2} \Bigl( \gr(\x) - \frac{\sigma(L, \x) - n}{2} \Bigr).
\end{align}

In the case when the grading is absolute and $N$ is even (so that the original Floer complex is $\Z/4$-graded), the real Lagrangian Floer complex is at least $\Z/2$-graded and we can define its real Euler characteristic 
$$\chiR(L) := \chi(\HFR_*(L)).$$
This is reminiscent of an invariant defined in \cite{FKM}, in the context of symmetric periodic orbits. 

When we only have a relative $\Z/2$-grading $\gr_R$, the Euler characteristic $\chiR(L)$ is well-defined only up to sign.

Another way to induce a $\Z/2$-grading on the real Lagrangian Floer complex is by choosing orientations on the Lagrangians $L$ and $M^R$ (when such orientations exist). In that case, the intersection points acquire signs, and the resulting Euler characteristic is simply the intersection product:
$$ \chihat_R(L) := [L] \cdot [M^R].$$
We use the notation $\chihat_R$ to distinguish it from $\chiR$. The two Euler characteristics agree (at least up to sign) when any two intersection points $\x$ and $\y$ are related by a homotopy class of strips, i.e. $\pi_2(\x, \y)\neq \emptyset$. However, there exist situations when $\chiR$ and $\chihat_R$ are different; see Example~\ref{ex: different euler characteristics}.

When $L$ and $M^R$ are orientable  but we do not choose orientations, the Euler characteristic $\chihat_R(L)$ is well-defined only up to sign.

\section{Real Heegaard diagrams}\label{sec: real HD}
In this section, we adapt some of the basic notions of classical Heegaard Floer theory to the invariant setting. Throughout, $Y$ is taken to be a closed, connected, oriented 3-manifold equipped with an involution $\tau$ whose fixed set, which will be denoted $C$, is non-empty and of codimension 2. Note that this ensures that the involution is orientation-preserving. (Indeed, by choosing an equivariant Riemannian metric, we find that near the fixed point set $\tau$ is conjugate to its linearization, which must be a linear involution with fixed point set of codimension $2$; the only possibility is rotation by $\pi$ about a line, which is orientation-preserving.)

We refer to the pair $(Y,\tau)$ as a \emph{real 3-manifold}.

\subsection{Real Heegaard Diagrams} \label{sec:rhd} Recall that a genus $g$ Heegaard splitting of $Y$ is a decomposition of $Y$ as $U_\alpha \cup_\Sigma U_\beta$ where $\Sigma$ is an oriented surface of genus $g$ and $U_\alpha$ and $U_\beta$ are handlebodies. A handle decomposition for a handlebody $U$ can be specified by choosing a collection of $g$ disjoint, simple closed curves $\{\alpha_1, \hdots, \alpha_g\}$ (which are usually called \emph{attaching circles}) which are linearly independent in $\Sigma$; then, $U$ is built by attaching 2-handles along these curves followed by a single 3-handle. Therefore, by fixing two sets of attaching circles $\bm \alpha = \{\alpha_1, \hdots, \alpha_g\}$ and $\bm \beta = \{\beta_1, \hdots, \beta_g\}$, we specify handle decompositions for $U_\alpha$ and $U_\beta$, and we can reconstruct $Y$. The data $(\Sigma, \bm \alpha, \bm \beta)$ is called a \emph{Heegaard diagram} for $Y$.

In the presence of an involution $\tau$ on $Y$, we may consider Heegaard splittings which respect the symmetry in some way. We will be interested in \emph{real Heegaard splittings}, in which the two handlebodies are swapped by the involution, i.e. $\tau(U_\alpha) = U_\beta$. Since $\tau$ preserves the orientation on $Y$, observe that it must reverse the orientation on $\Sigma = \del U_{\alpha} = -\del U_{\beta}$. 

Real Heegaard splittings can be specified by \emph{real Heegaard diagrams}.  

\begin{defn}
    A \emph{real Heegaard diagram} is a pair $(\cH, R)$, where $\cH$ is a Heegaard diagram 
    \begin{align*}
        \cH = (\Sigma,\{\alpha_1, \hdots, \alpha_g\}, \{\beta_1, \hdots, \beta_g\})
    \end{align*}
    and $R$ is an orientation-reversing involution $R:\Sigma \ra \Sigma$ with the property that $R(\alpha_i) = \beta_i$ for all $i$. 
\end{defn}

Clearly $(\cH, R)$ specifies a real 3-manifold: we attach 2-handles along the attaching curves and fill the remaining boundary with a pair of 3-balls to obtain a 3-manifold $Y(\cH)$. We extend $R$ over $Y(\cH)$ by mapping the 2-handle attached along $\alpha_i$ to the 2-handle attached along $\beta_i = R(\alpha_i)$ and by mapping the 3-ball filling the alpha handlebody to the 3-ball filling the beta handlebody. We say that $(\cH, R)$ is a real Heegaard diagram for $(Y, \tau)$ if there is a diffeomorphism $f: Y(\cH) \ra Y$ so that $\tau \circ f = f \circ R.$ 
\begin{prop}
    Every real 3-manifold $(Y, \tau)$ has a real Heegaard diagram.
\end{prop}
\begin{proof}
     According to a theorem of Nagase \cite[Proposition 2.4]{nagase}, every real 3-manifold admits a real Heegaard splitting. A handle decomposition for either of the resulting handlebodies determines a handle decomposition for the other by symmetry. We obtain an involution of the splitting surface by restricting $\tau$. This data specifies a real Heegaard diagram.
\end{proof}

Since we will need to construct Heegaard diagrams for specific real 3-manifolds, we present below a proof of Nagase's theorem that can be used in practice. Note that real 3-manifolds $(Y, \tau)$ are in one-to-one correspondence with double branched covers (of the quotient $X:=Y/\tau$ over the projection $L$ of $C$). The idea is to build the real Heegaard splitting by starting from the quotient.
 
Let $L$ be a link in a closed 3-manifold $X$.  It is well-known that $L$ admits a branched double cover if and only if $[L] = 0 \in H_1(X;\Z/2)$; and if it does, then it admits $|H_2(X;\Z/2)|$ many. (Compare \cite[Remark 3.15]{li:HMR}.) 
 
\begin{definition}
A surface $F$ smoothly embedded in $X$ is a \emph{spanning surface for $L$} if $\partial F = L$. (Note that $F$ is not necessarily a Seifert surface, as we do not require that $F$ be orientable.) 
\end{definition}

\begin{definition}
We say that a spanning surface $F$ is \emph{free} if $X \smallsetminus \nu(F)$ is a handlebody.
\end{definition}
 
We begin with the observation that any link which admits a spanning surface in fact admits a free one. This argument was suggested to us by Robert Lipshitz.

\begin{lem}\label{lem:free spanning surface}
    Let $L$ be a link in a 3-manifold $X$ such that $[L] = 0 \in H_1(X;\Z/2)$. Then, $L$ admits a free spanning surface $F \subset X$.
\end{lem}
\begin{proof}
    Let $X_L := X \smallsetminus \nu(L)$. Since $L$ represents a trivial mod 2 homology class, there is a class $\alpha \in H_2(X, L;\Z/2)$ with $\partial \alpha = [L]$. By excision and Poincar\'e duality, we have $H_2(X, L;\Z/2) \cong H_2(X_L, \partial X_L;\Z/2) \cong H^1(X_L;\Z/2)$, so we may identify $\alpha$ with a class in $[X_L, \RP^\infty]$. By cellular and smooth approximation, $\alpha$ can be represented by a smooth map $f_\alpha: X_L \ra \RP^3$ which is transverse to $\RP^2 \sub \RP^3$. Hence, $F_0 := f_{\alpha}^{-1}(\RP^2)$ is a spanning surface for $L$. We note that the natural map $H_2(X;\Z/2) \ra H_2(X, L;\Z/2)$ is injective; hence, $H_2(X;\Z/2)$ acts freely on the set of relative homology classes of spanning surfaces for $L$. Clearly, the action is transitive (if $A$, $B$ are spanning surfaces, their classes $[A],[B] \in H_2(X,L;\Z/2)$ differ by the image of $[A\cup_L B] \in H_2(X;\Z/2)$).

    We may modify $F_0$ within its homology class to obtain a free spanning surface for $L$. As we may tube together the components of $F_0$, we may assume that $F_0$ is connected. Fix a handle decomposition of $F_0$ into disks and bands. Fix a triangulation $\cT$ of $X$ which is adapted to $F_0$ in the following sense: we choose vertices so that the cores of the disks of $F_0$ are vertices, and the cores of the bands of $F_0$ are unions of edges of $\cT$.  We shall now modify $F_0$ using $\cT$. For each vertex of $\cT$ not contained in $F_0$ add an additional disk and for each edge not in $F_0$, perform a band sum. The resulting surface, $F_1$, is isotopic to the union of $F_0$ together with a collection of bands, $B_i$. Note, the complement of $F_1$ is clearly a handlebody -- it is a neighborhood of the 2- and 3-simplices of $\cT$ -- though its boundary is no longer isotopic to $L$. To amend this, for each new band $B_i$ in $F_1$, we attach a dual band $B_i^*$; the resulting surface $F_2$ is still free (the dual bands change the handlebodies by a stabilization) and has boundary isotopic to $L$, as the bands $B_i$ with their duals form tubes, which do not change the isotopy class of the boundary. See \Cref{fig:make_free}. 
\end{proof}

\begin{figure}[h]
\def\svgwidth{.5\linewidth}
\begingroup%
  \makeatletter%
  \providecommand\color[2][]{%
    \errmessage{(Inkscape) Color is used for the text in Inkscape, but the package 'color.sty' is not loaded}%
    \renewcommand\color[2][]{}%
  }%
  \providecommand\transparent[1]{%
    \errmessage{(Inkscape) Transparency is used (non-zero) for the text in Inkscape, but the package 'transparent.sty' is not loaded}%
    \renewcommand\transparent[1]{}%
  }%
  \providecommand\rotatebox[2]{#2}%
  \newcommand*\fsize{\dimexpr\f@size pt\relax}%
  \newcommand*\lineheight[1]{\fontsize{\fsize}{#1\fsize}\selectfont}%
  \ifx\svgwidth\undefined%
    \setlength{\unitlength}{541.41732283bp}%
    \ifx\svgscale\undefined%
      \relax%
    \else%
      \setlength{\unitlength}{\unitlength * \real{\svgscale}}%
    \fi%
  \else%
    \setlength{\unitlength}{\svgwidth}%
  \fi%
  \global\let\svgwidth\undefined%
  \global\let\svgscale\undefined%
  \makeatother%
  \begin{picture}(1,0.48691099)%
    \lineheight{1}%
    \setlength\tabcolsep{0pt}%
    \put(0,0){\includegraphics[width=\unitlength,page=1]{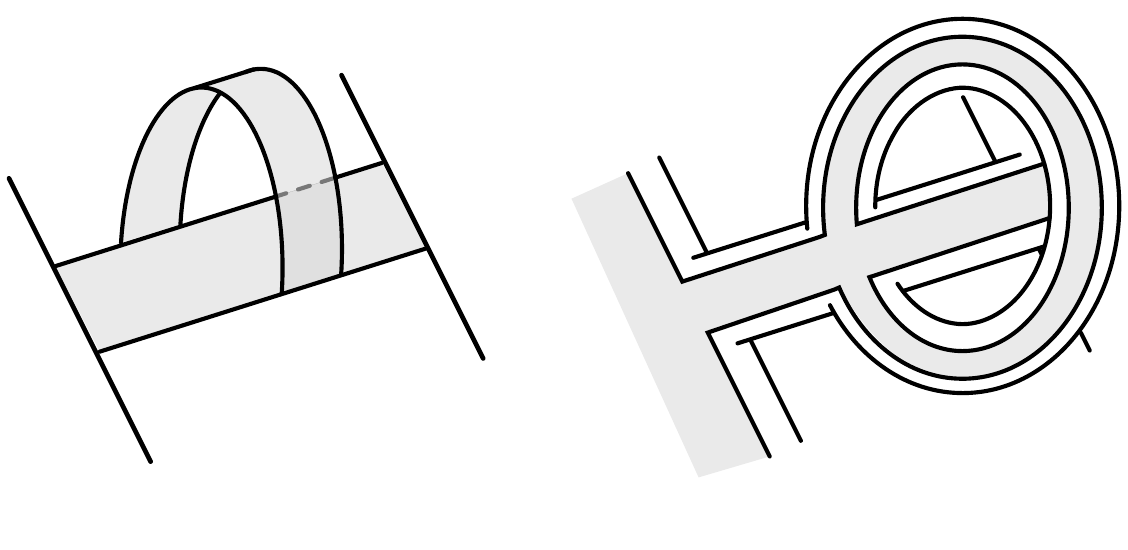}}%
    \put(0.22467608,0.02920856){\color[rgb]{0,0,0}\makebox(0,0)[lt]{\smash{\begin{tabular}[t]{l}{\small$(a)$}\end{tabular}}}}%
    \put(0.77877728,0.02920856){\color[rgb]{0,0,0}\makebox(0,0)[lt]{\smash{\begin{tabular}[t]{l}{\small$(b)$}\end{tabular}}}}%
    \put(0.15615412,0.4333842){\color[rgb]{0,0,0}\makebox(0,0)[lt]{\smash{\begin{tabular}[t]{l}{\small$B^*$}\end{tabular}}}}%
    \put(0.11260258,0.23222994){\color[rgb]{0,0,0}\makebox(0,0)[lt]{\smash{\begin{tabular}[t]{l}{\small$B$}\end{tabular}}}}%
  \end{picture}%
\endgroup%

\caption{(a) A band $B$ and its dual band $B^*$ attached to the surface $F_0$; (b) The resulting thickened surface.}
    \label{fig:make_free}
\end{figure}
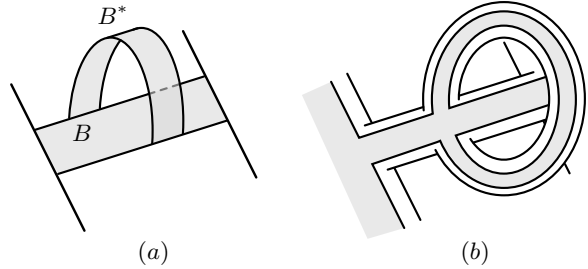

\begin{prop}\label{prop:constructing real HDs}
    Let $L \sub X$ be a link such that $[L] = 0 \in H_1(X;\Z/2)$. Then, for any branched double cover $Y$ of $X$ along $L$, there is a real Heegaard splitting of $X$ with respect to the branching involution.
\end{prop}
\begin{proof}
    According to \Cref{lem:free spanning surface}, there exists a free spanning surface $F$ for $L$. Let us explain how such a surface naturally give rise to a real Heegaard splitting for some branched double cover along $L$.  Since $F$ is free, its complement is a handlebody $H$. The boundary of $\nu(F)$ has a natural involution. If $F$ is orientable, then $\partial (\nu F) \cong F \times \partial I \cup L \times I$; in this case, let $R$ be the involution which swaps the two components of $F \times \partial I$ and reflects $L\times I$ through $L$. If $F$ is nonorientable, then $\partial (\nu F) \cong \partial (F \Tilde{\times} I) \cong \tilde{F} \cup L \times I$, where $\tilde{F}$ is the orientation double cover of $F$. In this case, we define $R$ to be the involution which acts by deck transformations on $\tilde{F}$ and by reflection through $L$ on $L \times I$. In either case, we may form the space
    \begin{align*}
        Y_{[F]} = H \cup_\phi H,
    \end{align*}
    which clearly is a branched double cover of $X$ along $L$. We note that up to diffeomorphism, this cover only depends on the homology class of $F$; if $F$ and $F'$ are homologous, the sequence of surgeries taking $F$ to $F'$ lifts to a sequence of real Heegaard moves relating the two splittings (in the sense of Section~\ref{sec:moves} below).
    
    In order to see that this construction realizes all branched covers along $L$, we note that it defines a map 
    \begin{equation}
    \label{eq:spansurf}
        \left\{\substack{\text{Mod 2 homology classes}\\ \text{of spanning surfaces for $L$}} \right\} \ra \left\{\substack{\text{Branched double}\\ \text{covers along $L$}}\right\}.
    \end{equation}
    Both the domain and the codomain are affinely identified with $H_2(X;\Z/2)$ as $H_2(X;\Z/2)$-spaces. If $F$ is a spanning surface for $L$ and $\alpha$ is a class in $H_2(X;\Z/2)$, the branched covers $Y_{[F]}$ and $Y_{[F] + \alpha}$ restrict to honest covers of $X_L$. Therefore, they determine classes in $H^1(X_L;\Z/2)$. Furthermore, since these covers restrict to the same cover of $\partial \nu(L)$, they must differ by an element $x$ of $H^1(X;\Z/2)$, namely, the Poincar\'e dual of $\alpha$. Hence, the map \eqref{eq:spansurf} respects the group actions on the two sides, and is therefore bijective.
\end{proof}

\begin{remark}\label{rem:orientable quotients}
    We note that when $[L] = 0 \in H_1(X;\Z)$, we can produce a real Heegaard diagram for its branched double cover whose quotient is orientable. Indeed,  we can choose $\alpha \in H_2(X,L;\Z) \cong H^1(X;\Z)$ as above, represent it by a map $f_\alpha: X \ra S^1$ and choose $F_0 = f_\alpha^{-1}(\pt)$, which will be orientable.
\end{remark}

\subsection{Real Heegaard Moves}
\label{sec:moves}
As in the classical setting, real Heegaard diagrams are unique up to the appropriate notion of stable equivalence. Things are slightly more subtle in the real setting, as we not only need to track pairs of attaching circles, but also involutions on the Heegaard surface. Therefore, before addressing the uniqueness of real Heegaard diagrams, it will be helpful to briefly describe the kinds of surfaces with involutions we will encounter. 
 
Surfaces with involutions have been classified. (For a modern perspective on the matter, see \cite{dugger_invol}.) Since we will assume throughout that our involutions have fixed sets of codimension 2, there are only two types of $\Z/2$-surfaces which can arise up to isomorphism. 

The first class of involutions on a genus $g$ surface $\Sigma_g$ is given by embedding $\Sigma_g$ in $\R^3$ and reflecting through the $xy$-plane as on the left frame of Figure \ref{fig:z2_surfaces}. In this case, the fixed set $C$ is a separating submanifold, and the quotient is an orientable surface with $|C|$ boundary components. 

The second class of involutions can be built from the first. The antipodal map on $S^2$ restricts to an involution on the annulus given by $$a: S^1 \times I \to S^1 \times I, \quad (x, \pm 1 )\mapsto (-x,\mp x).$$ Fix a pair of points $p, \tau(p) \in \Sigma_g$ and an identification of $\partial (\Sigma^g \smallsetminus (\nu(p \cup \tau(p))))$ with $S^1 \times S^0$ so that $\tau(x, \pm 1) = (r(x), \mp 1)$ where $r$ is reflection $r(e^{i \theta}) = -e^{-i\theta}$. Attach the equivariant tube $(S^1 \times I, a)$ to $\Sigma^g$ by the map $g:=\id \sqcup -r: (S^1 \times \{-1\})\sqcup (S^1 \times \{+1\}) \to (S^1 \times \{-1\})\sqcup (S^1 \times \{+1\})$. The gluing map respects the involution: 
\begin{align*}
    g(a(x, +1)) = g(-x, -1) = (-x, -1) = \tau(-r(x), +1) = \tau(g(x, +1))\\
    g(a(x, -1)) = g(-x, +1) = (-r(-x), -1) = (r(x), -1) = \tau(g(x, -1)).
\end{align*}
The underlying surface is of course $\Sigma_{g+1}$, but the quotient is now nonorientable and the fixed set $C$ is no longer separating. By attaching $r$ tubes in this fashion, we obtain an involution on a surface of genus $g + r$ with nonorientable quotient and $|C|$ boundary components. See the right frame of Figure \ref{fig:z2_surfaces}.

\begin{figure}[h]
\def\svgwidth{.8\linewidth}
\begingroup%
  \makeatletter%
  \providecommand\color[2][]{%
    \errmessage{(Inkscape) Color is used for the text in Inkscape, but the package 'color.sty' is not loaded}%
    \renewcommand\color[2][]{}%
  }%
  \providecommand\transparent[1]{%
    \errmessage{(Inkscape) Transparency is used (non-zero) for the text in Inkscape, but the package 'transparent.sty' is not loaded}%
    \renewcommand\transparent[1]{}%
  }%
  \providecommand\rotatebox[2]{#2}%
  \newcommand*\fsize{\dimexpr\f@size pt\relax}%
  \newcommand*\lineheight[1]{\fontsize{\fsize}{#1\fsize}\selectfont}%
  \ifx\svgwidth\undefined%
    \setlength{\unitlength}{841.88976378bp}%
    \ifx\svgscale\undefined%
      \relax%
    \else%
      \setlength{\unitlength}{\unitlength * \real{\svgscale}}%
    \fi%
  \else%
    \setlength{\unitlength}{\svgwidth}%
  \fi%
  \global\let\svgwidth\undefined%
  \global\let\svgscale\undefined%
  \makeatother%
  \begin{picture}(1,0.3030303)%
    \lineheight{1}%
    \setlength\tabcolsep{0pt}%
    \put(0.23935242,0.2715508){\color[rgb]{0,0,0}\makebox(0,0)[lt]{\smash{\begin{tabular}[t]{l}{$\tau$}\end{tabular}}}}%
    \put(0,0){\includegraphics[width=\unitlength,page=1]{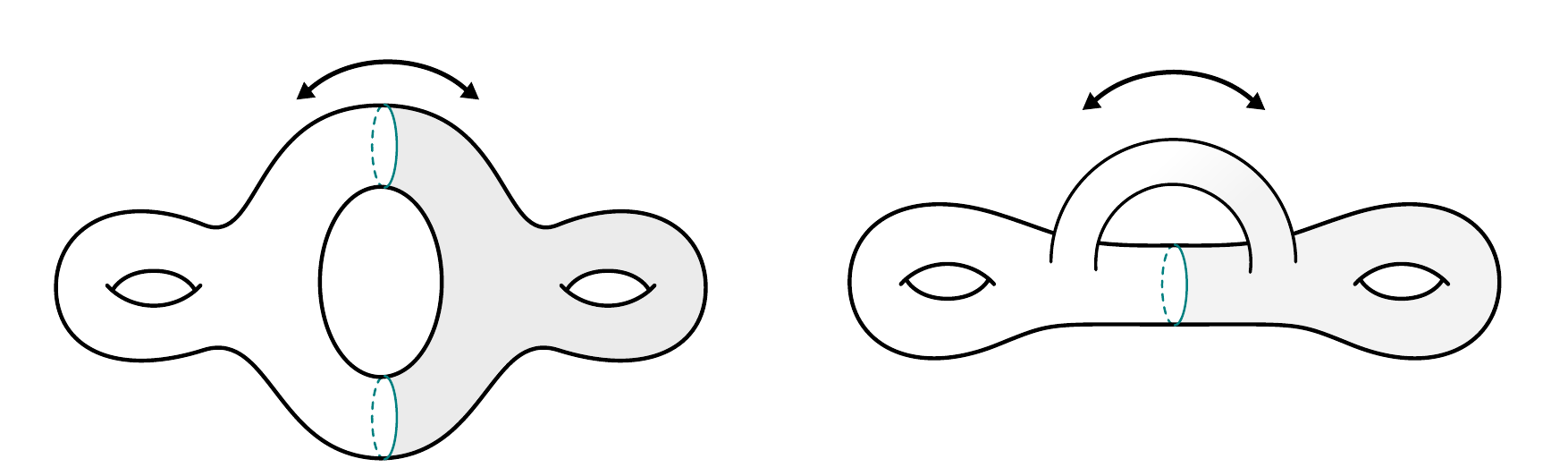}}%
    \put(0.74160217,0.265941){\color[rgb]{0,0,0}\makebox(0,0)[lt]{\smash{\begin{tabular}[t]{l}{$\tau$}\end{tabular}}}}%
  \end{picture}%
\endgroup%

\caption{Left: a $\Z/2$-surface with orientable quotient. The involution is given by reflection. Right: A  $\Z/2$-surface with nonorientable quotient; the involution is reflection on the complement of the handle $h$ and the antipodal map on $h$.}
    \label{fig:z2_surfaces}
\end{figure}

A given real 3-manifold can admit Heegaard diagrams with either type of underlying $\Z/2$-surface. 

\begin{example}
\label{ex:stab}
Consider the diagrams in \Cref{fig:s3_diagrams}. The fixed sets are drawn in green. It is clear that each diagram represents $S^3$. The leftmost diagram has an orientable quotient and the involution is given by reflection. The remaining diagrams have nonorientable quotients.  Near the fixed set, the involution is given by reflection. The complements of regular neighborhoods of the fixed sets are annuli, and the involution is the antipodal map. Since there is a unique involution on $S^3$ with fixed set $S^1$ up to conjugation, each of these diagrams must represent the same real 3-manifold.
\end{example}

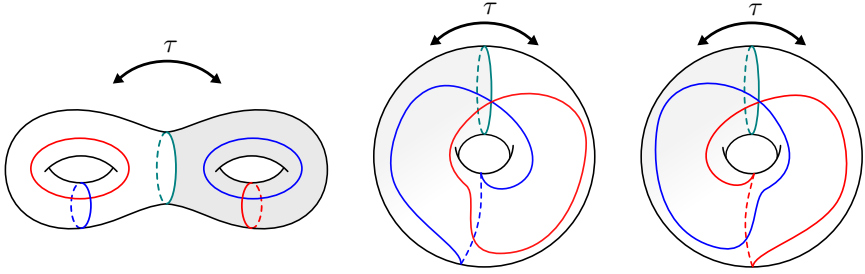
\begin{figure}[h]
\def\svgwidth{.8\linewidth}
\begingroup%
  \makeatletter%
  \providecommand\color[2][]{%
    \errmessage{(Inkscape) Color is used for the text in Inkscape, but the package 'color.sty' is not loaded}%
    \renewcommand\color[2][]{}%
  }%
  \providecommand\transparent[1]{%
    \errmessage{(Inkscape) Transparency is used (non-zero) for the text in Inkscape, but the package 'transparent.sty' is not loaded}%
    \renewcommand\transparent[1]{}%
  }%
  \providecommand\rotatebox[2]{#2}%
  \newcommand*\fsize{\dimexpr\f@size pt\relax}%
  \newcommand*\lineheight[1]{\fontsize{\fsize}{#1\fsize}\selectfont}%
  \ifx\svgwidth\undefined%
    \setlength{\unitlength}{841.88976378bp}%
    \ifx\svgscale\undefined%
      \relax%
    \else%
      \setlength{\unitlength}{\unitlength * \real{\svgscale}}%
    \fi%
  \else%
    \setlength{\unitlength}{\svgwidth}%
  \fi%
  \global\let\svgwidth\undefined%
  \global\let\svgscale\undefined%
  \makeatother%
  \begin{picture}(1,0.3030303)%
    \lineheight{1}%
    \setlength\tabcolsep{0pt}%
    \put(0,0){\includegraphics[width=\unitlength,page=1]{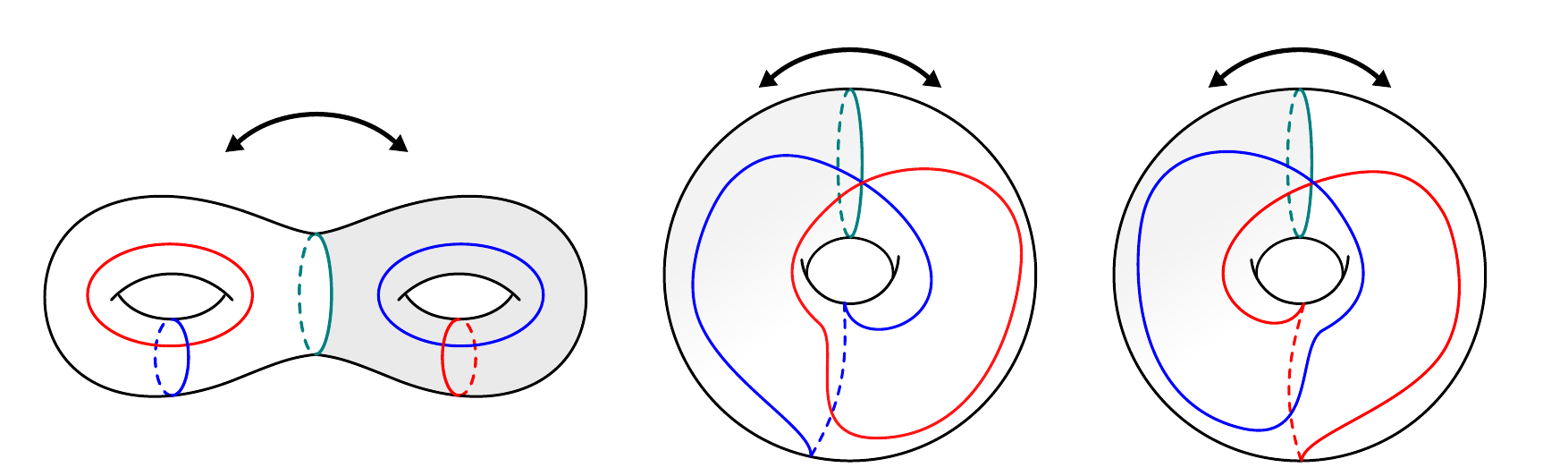}}%
    \put(0.19800783,0.24026667){\color[rgb]{0,0,0}\makebox(0,0)[lt]{\smash{\begin{tabular}[t]{l}{$\tau$}\end{tabular}}}}%
    \put(0.54095557,0.28088177){\color[rgb]{0,0,0}\makebox(0,0)[lt]{\smash{\begin{tabular}[t]{l}{$\tau$}\end{tabular}}}}%
    \put(0.82543581,0.27987399){\color[rgb]{0,0,0}\makebox(0,0)[lt]{\smash{\begin{tabular}[t]{l}{$\tau$}\end{tabular}}}}%
  \end{picture}%
\endgroup%

\caption{Three real Heegaard diagrams for $(S^3, \tau)$. The fixed sets of the various involutions are shown in green.}
    \label{fig:s3_diagrams}
\end{figure}

Now we move to the question of the uniqueness of real Heegaard diagrams. As expected, any two real Heegaard diagrams should be related by a finite sequence of \emph{real Heegaard moves}. The first two moves are quite simple:

\begin{definition}
A \emph{real isotopy} is an isotopy  $\phi_t$ of the attaching curves such that for all $t$ we have $\tau(\phi_t(\alpha_i)) = \phi_t(\beta_i)$. 
\end{definition}

\begin{definition} If $\alpha_i'$ is the curve obtained by doing a handleslide of  $\alpha_i$ over $\alpha_j$, there is a corresponding beta curve, $\beta_i'$, which can be obtained by applying $\tau$ to $\alpha_i'$. A \emph{real handleslide} is the move which replaces $(\Sigma, \alpha_1, \hdots, \alpha_i, \hdots, \alpha_g,\beta_1, \hdots, \beta_i, \hdots, \beta_g)$ with $(\Sigma, \alpha_1, \hdots, \alpha_i', \hdots, \alpha_g,\beta_1, \hdots, \beta_i', \hdots, \beta_g)$. We think of this move as simultaneously sliding $\alpha_i$ and $\beta_i$ over $\alpha_j$ and $\beta_j$ respectively. 
\end{definition}

Finally, there are stabilizations. In light of Example~\ref{ex:stab}, it is probably unsurprising that there are several kinds of stabilization operations that can be performed on a real Heegaard diagram. 

\begin{definition}
Let $\cH = (\Sigma, \bm \alpha, \bm \beta, \tau)$ be a real Heegaard diagram. Let $E$ the standard genus 1 Heegaard diagram for $S^3$ with attaching curves $\alpha_0$ and $\beta_0$ that intersect in single point. Let $E'$ be the diagram obtained from $E$ by swapping the alpha and beta curves (which we call $\alpha_1$ and $\beta_1$). Fix a point $p \in \Sigma \smallsetminus \nu(C)$. Perform a connected sum with $E$ at the point $p$ and a connected sum with $E'$ at $\tau(p)$. Let $\cH' = (\Sigma' = T^2\# \Sigma\# \tau T^2, \bm \alpha \cup \{\alpha_0, \alpha_1\}, \bm \beta\cup \{\beta_0, \beta_1\})$ be the resulting Heegaard diagram, and extend $\tau$ over the new handles so that $\alpha_0$ and $\beta_1$ are swapped, as are $\beta_0$ and $\alpha_1$. See \Cref{fig:stabilizations}. We call this move  a \emph{free stabilization}.
\end{definition}

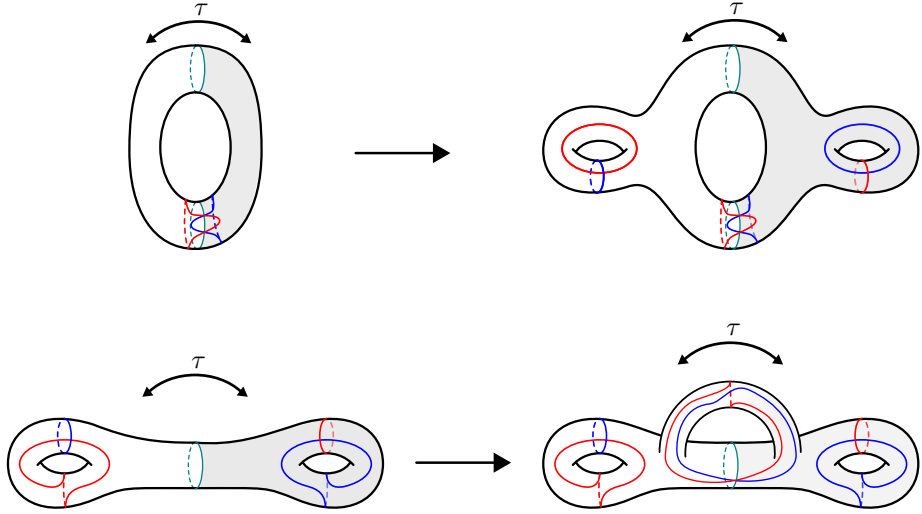
\begin{figure}[h]
\def\svgwidth{.8\linewidth}
\begingroup%
  \makeatletter%
  \providecommand\color[2][]{%
    \errmessage{(Inkscape) Color is used for the text in Inkscape, but the package 'color.sty' is not loaded}%
    \renewcommand\color[2][]{}%
  }%
  \providecommand\transparent[1]{%
    \errmessage{(Inkscape) Transparency is used (non-zero) for the text in Inkscape, but the package 'transparent.sty' is not loaded}%
    \renewcommand\transparent[1]{}%
  }%
  \providecommand\rotatebox[2]{#2}%
  \newcommand*\fsize{\dimexpr\f@size pt\relax}%
  \newcommand*\lineheight[1]{\fontsize{\fsize}{#1\fsize}\selectfont}%
  \ifx\svgwidth\undefined%
    \setlength{\unitlength}{867.4015748bp}%
    \ifx\svgscale\undefined%
      \relax%
    \else%
      \setlength{\unitlength}{\unitlength * \real{\svgscale}}%
    \fi%
  \else%
    \setlength{\unitlength}{\svgwidth}%
  \fi%
  \global\let\svgwidth\undefined%
  \global\let\svgscale\undefined%
  \makeatother%
  \begin{picture}(1,0.55555556)%
    \lineheight{1}%
    \setlength\tabcolsep{0pt}%
    \put(0,0){\includegraphics[width=\unitlength,page=1]{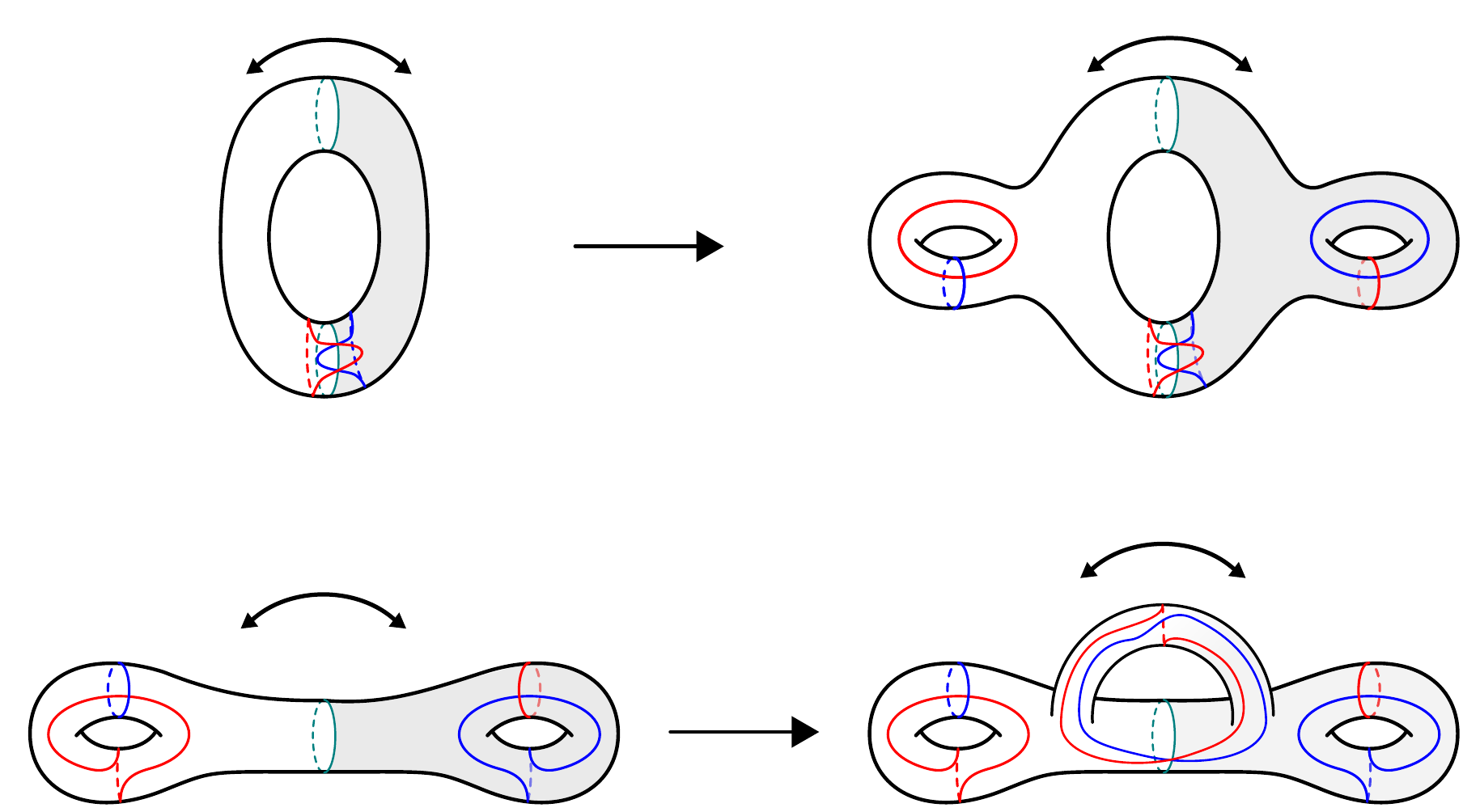}}%
    \put(0.22013708,0.53515179){\color[rgb]{0,0,0}\makebox(0,0)[lt]{\smash{\begin{tabular}[t]{l}{$\tau$}\end{tabular}}}}%
    \put(0.21723753,0.15615506){\color[rgb]{0,0,0}\makebox(0,0)[lt]{\smash{\begin{tabular}[t]{l}{$\tau$}\end{tabular}}}}%
    \put(0.79443175,0.53651865){\color[rgb]{0,0,0}\makebox(0,0)[lt]{\smash{\begin{tabular}[t]{l}{$\tau$}\end{tabular}}}}%
    \put(0.79116971,0.19210793){\color[rgb]{0,0,0}\makebox(0,0)[lt]{\smash{\begin{tabular}[t]{l}{$\tau$}\end{tabular}}}}%
  \end{picture}%
\endgroup%

\caption{Top: A free stabilization. Bottom: A fixed point stabilization.}
    \label{fig:stabilizations}
\end{figure}

\begin{definition}
Let $(F^+, R)$ be the real Heegaard diagram for $(S^3, \tau)$ shown in the center of \Cref{fig:s3_diagrams}, where $F^+= (T^2, \alpha_+, \beta_+)$. Let $c_+$ be a point in the fixed set of $F^+$ and let $c$ be a point in the fixed set of $\cH$. Let $\cH'' = (\Sigma'' = \Sigma\# T^2, \bm \alpha \cup \{\alpha_+\}, \bm \beta\cup \{\beta_+\})$ be the diagram obtained by taking a connected sum of $\cH$ with $F^+$ at the points $c$ and $c_+$. We call this move a \emph{positive fixed point stabilization.} Let $(F^-, R)$ be the real Heegaard diagram shown on the right of \Cref{fig:s3_diagrams}. A \emph{negative fixed point stabilization} is given by taking a connected sum with $(F^-, R)$. 
\end{definition}

We will refer to free stabilization, and positive and negative fixed point stabilization collectively as \emph{real stabilizations.}

For the time being, we distinguish positive and negative fixed point stabilizations. We do so for the following reason. A choice of real Heegaard splitting for $(Y, \tau)$ gives rise to a preferred framing for the fixed set, $C$. When $\Sigma$ has orientable quotient, $C$ is a separating curve, so cutting $\Sigma$ in half along $C$ yields a Seifert surface for $C$. It follows that $C$ is nullhomologous and the framing determined by $C$ agrees with the Seifert framing. If $\Sigma$ has a nonorientable quotient, the fixed set will no longer be separating, and $C$ may no longer be nullhomologous (though its image in the quotient $X=Y/\tau$ will represent a trivial class in $H_1(X;\Z/2)$), so does not have a canonical framing. Nevertheless, we can consider the effect our various stabilization operations have on the framing of $C$. It is clear that free stabilization preserves the framing. From \Cref{fig:s3_diagrams}, it is also apparent that the fixed set in $F^+$ has framing $+1$, and therefore, when a positive fixed point stabilization is performed, the framing will be increased by 1. Similarly, a negative fixed point stabilization will decrease the framing by 1. 

\begin{definition}
We say that two real Heegaard diagrams $(\cH, R)$ and $(\cH, R')$ are \emph{real stably equivalent} if they become diffeomorphic after a finite sequence of real isotopies, handleslides, and stabilizations. 
\end{definition}

\begin{prop}
\label{prop:sequiv}
    Any two real Heegaard diagrams for $(Y, \tau)$ are real stably equivalent.
\end{prop}
\begin{proof}
    Let $(\cH, R)$ and $(\cH, R')$ be two real Heegaard diagrams for $(Y, \tau)$. Without loss of generality, we may assume that both of these diagrams induce the same framing of $C$, because any framing of $C$ can be realized by a sequence of fixed point stabilizations. According to Nagase \cite[Theorem 1]{nagase}, any two real Heegaard splittings $(Y, U, \tau_1(U), \tau_1)$ and $(Y, V, \tau_2(V), \tau_2)$ which induce the same framing can be related by a sequence of the following moves: free stabilizations and real equivalences $\phi: (Y, U, \tau_1(U), \tau_1) \to (Y, V, \tau_2(V), \tau_2)$, i.e. diffeomorphisms which respect the splittings and conjugate the involutions $\tau_1$ and $\tau_2$ on $Y$. In our definition of free stabilization, the feet of each added tube is attached in the same region of $\Sigma \smallsetminus (\bm \alpha \cup \bm \beta)$, though Nagase allows for more general free stabilizations, in which the attaching tubes may have have non-adjacent feet. The latter move can be decomposed into a sequence of the real Heegaard moves, as we have defined them. Indeed, when such a stabilization is performed, we choose a new pair of alpha curves; the first can be taken to be a core of the first tube, but the other alpha curve must pass over the second tube and connect its two feet. If the this alpha curves does not cross $C$, it is easy to see that such a generalized stabilization is equivalent to ours after a sequence of real handleslides and a diffeomorphism. If this path crosses $C$, we reduce to the first case by performing a sequence of positive and negative fixed point stabilizations and diffeomorphisms. See \Cref{fig:gen_stab} for this procedure. 
    
    Once we fix the real Heegaard splitting, \cite[Proposition 2.2]{os_holodisks} shows that any two sets of attaching alpha circles are related by isotopies and handleslides. We can upgrade these moves to real isotopies and handleslides by moving the beta  curves accordingly. 
\end{proof}

\begin{figure}[h]
\def\svgwidth{.8\linewidth}
\begingroup%
  \makeatletter%
  \providecommand\color[2][]{%
    \errmessage{(Inkscape) Color is used for the text in Inkscape, but the package 'color.sty' is not loaded}%
    \renewcommand\color[2][]{}%
  }%
  \providecommand\transparent[1]{%
    \errmessage{(Inkscape) Transparency is used (non-zero) for the text in Inkscape, but the package 'transparent.sty' is not loaded}%
    \renewcommand\transparent[1]{}%
  }%
  \providecommand\rotatebox[2]{#2}%
  \newcommand*\fsize{\dimexpr\f@size pt\relax}%
  \newcommand*\lineheight[1]{\fontsize{\fsize}{#1\fsize}\selectfont}%
  \ifx\svgwidth\undefined%
    \setlength{\unitlength}{691.65354331bp}%
    \ifx\svgscale\undefined%
      \relax%
    \else%
      \setlength{\unitlength}{\unitlength * \real{\svgscale}}%
    \fi%
  \else%
    \setlength{\unitlength}{\svgwidth}%
  \fi%
  \global\let\svgwidth\undefined%
  \global\let\svgscale\undefined%
  \makeatother%
  \begin{picture}(1,0.75)%
    \lineheight{1}%
    \setlength\tabcolsep{0pt}%
    \put(0,0){\includegraphics[width=\unitlength,page=1]{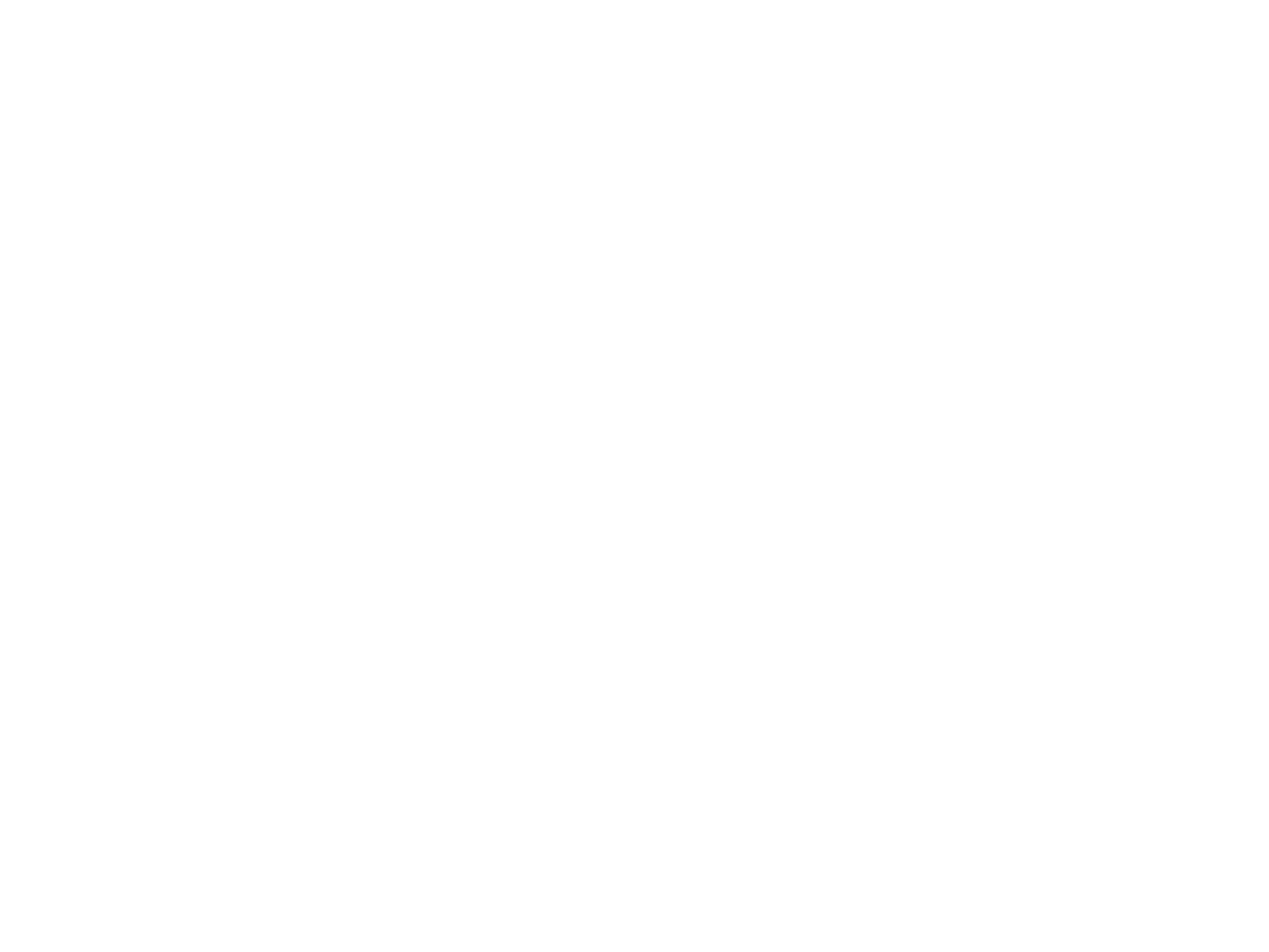}}%
    \put(-1.65215239,0.54673948){\color[rgb]{0,0,0}\makebox(0,0)[lt]{\smash{\begin{tabular}[t]{l}{(a)}\end{tabular}}}}%
    \put(-1.65865864,0.1043214){\color[rgb]{0,0,0}\makebox(0,0)[lt]{\smash{\begin{tabular}[t]{l}{(c)}\end{tabular}}}}%
    \put(-1.0904551,0.54673948){\color[rgb]{0,0,0}\makebox(0,0)[lt]{\smash{\begin{tabular}[t]{l}{(b)}\end{tabular}}}}%
    \put(-1.0904551,0.1043214){\color[rgb]{0,0,0}\makebox(0,0)[lt]{\smash{\begin{tabular}[t]{l}{(d)}\end{tabular}}}}%
    \put(0,0){\includegraphics[width=\unitlength,page=2]{gen_stab.pdf}}%
  \end{picture}%
\endgroup%

\caption{Top row: a free stabilization in which the curves running between the feet intersect once on $C$. This can be reduced to a pair of fixed point stabilizations by a real handleslide and a diffeomorphism. Bottom row: the curves connecting the feet of the tubes intersect several times along $C$; by applying a free stabilization followed by a real handleslide and diffeomorphism as in the configuration in the top row, we can reduce the number of intersections on $C$. Repeatedly using this trick, we can reduce to the standard free stabilization shown in \Cref{fig:stabilizations}.}
    \label{fig:gen_stab}
\end{figure}

\subsection{Multi-pointed real Heegaard diagrams}
As is typical in Heegaard Floer theory, we will work with pointed Heegaard diagrams. Inspired by \cite{os_linkinvts}, we allow a more general setting, in which there can be several basepoints and more than $g$ curves of each type.

\begin{defn}
    A \emph{multi-pointed real Heegaard diagram} for $(Y, \tau, \w)$ consists of the following data: 
    \begin{enumerate}
        \item A real Heegaard splitting for $Y$ into handlebodies $U \cup \tau(U)$ such that $\w \sub C$; 
        \item two collections of $m = g(\Sigma) + |\w| - 1$ disjoint simple, closed curves $\balpha = \{\alpha_1, \hdots, \alpha_m\}$ and $\bbeta = \{\beta_1, \hdots, \beta_m\}$ in $\Sigma$ which bound compressing disks in $U$ and $\tau(U)$ respectively, with the property that each component of $\Sigma \smallsetminus \balpha$ and $\Sigma \smallsetminus \bbeta$ contains exactly one basepoint;
        \item an involution $\tau: \Sigma \ra \Sigma$ which exchanges $\balpha$ and $\bbeta$. 
    \end{enumerate}
\end{defn}

The real Heegaard moves from Section~\ref{sec:moves} admit straightforward generalizations to the case of more curves. Furthermore, {\em pointed real Heegaard moves} are those real Heegaard moves supported in the complement of the basepoints.  

\begin{lem}\label{lem:pointed real Heegaard moves}
    Any two multi-pointed real Heegaard diagrams representing $(Y,\tau,\w)$ can be connected by a sequence of pointed real Heegaard moves.
\end{lem}
\begin{proof}
The proof of Proposition~\ref{prop:sequiv} extends to the case of arbitrary $m \geq g(\Sigma)$ to show that the  diagrams can be connected by a sequence of real Heegaard moves. We can easily arrange the handleslides and stabilizations to happen in the complement of the basepoints, by a small real isotopy. It remains to show that we can also avoid real isotopies that cross a basepoint. 
    Since the alpha and beta curves transform together, it suffices to show that any isotopy of an alpha curve that crosses a basepoint can be replaced by a sequence of handle slides in the complement of the basepoint. For the case of a single basepoint, this follows just as in \cite[Proposition 7.2]{os_holodisks} and for multiple basepoints from \cite[Proposition 3.3]{os_linkinvts} or \cite[Section 2]{Juhasz_HolomorphicDisks_SuturedMflds}. 
\end{proof}

When $(Y, \tau)$ is represented by a real Heegaard diagram, we will abuse notation and simply write $\tau$ (instead of $R$) also for the involution on the Heegaard surface. 

\begin{remark}
Of course, the moves above cannot move basepoints between different components of the fixed set. Hence, when $|\w| < |C|$, some components of the fixed set are necessarily distinguished. In the case that $|\w| = 1 < |C|$, it would be interesting to know whether the resulting real Heegaard Floer homologies are isomorphic. \end{remark}

\begin{remark}
It may also be possible to extend the theory to free actions, but due to the need to choose pairs of basepoints ($w$ and $\tau(w)$), the resulting theory would depend on the $\tau$-invariant knot determined by $w$ and $\tau(w)$.  In some cases, the basepoints can be chosen to lie in the same component of the Heegaard diagram, and so the associated knot is a local unknot. However, it is unclear to us if one can always choose the basepoints in this manner. 
\end{remark}

\subsection{Real Invariant Domains}
\label{sec:RID}
Given a multi-pointed Heegaard diagram $\cH = (\Sigma, \balpha, \bbeta, \w)$ for $Y$, the Heegaard Floer invariants are defined in terms of Lagrangian intersection Floer homology. One considers the symmetric product $\Sym^m(\Sigma)$ which contains two half-dimensional submanifolds $\Ta = \alpha_1 \times \hdots \times \alpha_m$ and $\Tb = \beta_1 \times\hdots\times \beta_m$ determined by $\cH$. By the work of Perutz \cite{perutz}, there is a symplectic form on $\Sym^m(\Sigma)$ with respect to which $\Ta$ and $\Tb$ are Lagrangians. (Perutz only considered the case where $m$ is the genus $g$ of the surface, but his proof applies in general.) The Heegaard Floer complex of $Y$ is defined to be $\CF_*(\Ta, \Tb)$ relative to some divisors determined by the basepoints. 

An involution of $\Sigma$ induces an involution of $M = \Sym^m(\Sigma)$. As suggested in the Introduction, the real Heegaard Floer homology of $(Y, \tau)$ is to be the Lagrangian Floer homology of $\Ta$ and $M^R$ in $M$. This will be a particular case of the construction from ~\Cref{sec:real_lf}.

Given a pair $\x, \y \in \Ta \cap M^R$, we are therefore interested in pseudo-holomorphic representatives of Whitney disks $\psi \in \Pi_2(\x, \y)$. Rather than work directly with classes in $\Pi_2(\x, \y)$, we instead work with their corresponding invariant classes $\phi=\cR(\psi) \in \pi_2^R(\x, \y)$; see the discussion at the end of Section~\ref{sec:setup}.

As is standard in Heegaard Floer theory, it is often more convenient to work with domains in $\Sigma$, rather than disks in $\Sym^m(\Sigma)$. Given a point $w \in \Sigma$, there is a map
\begin{align*}
    n_w: \pi_2(\x, \y) \ra \Z
\end{align*}
taking a class $\phi$ to the algebraic intersection number of $\phi$ with $\{w\} \times \Sym^{m-1}(\Sigma).$ More generally, given a collection of points $\w = \{w_1, \hdots, w_r\}$, there is a map $n_\w: \pi_2(\x, \y) \ra \Z^{r}$ given by $$n_\w(\phi) = (n_{w_1}(\phi), \hdots, n_{w_r}(\phi)).$$ Let $\cD_1, \hdots, \cD_\ell$ be the connected components of $\Sigma \smallsetminus \{\balpha \cup \bbeta\}$, and pick a point $p_i$ in $\cD_i$. Given a holomorphic strip $u: (\D, \del \D)  \ra (\Sym^m(\Sigma), \bT_\alpha \cup \bT_\beta)$, its associated \emph{domain} is the 2-chain $\cD(u) = \sum_i n_{p_i}(u) \cD_i$. In light of the correspondence with real invariant strips, to a holomorphic strip $u: (\D, \del \D)  \ra (\Sym^m(\Sigma), \bT_\alpha \cup M^R)$ we may associate the domain of the real invariant strip $\cR(u)$ to aid in visualization. 

A \emph{periodic domain} is a integral 2-chain on $\Sigma$ whose boundary is a linear combination of the alpha and beta circles, and whose multiplicity at all $w_i$ is zero. A Whitney disk $\phi \in \pi_2(\x, \x)$ represented by such a domain is called a \emph{periodic class}. The set of integral periodic domains is denoted by $\Pi_\Z$. If we use rational instead of integral coefficients, the resulting set is denoted $\Pi_\Q$. There are obvious analogues in the real setting. We let $\Pi_\Z^R$ be the subset of $\Pi_\Z$ which consists of real invariant periodic domains. We similarly define $\Pi_\Q^R \subset \Pi_\Q$. 

The spaces $\Pi_\Z$ and $\Pi_\Z^R$ are closely related to the topology of $Y$. There is a map 
\begin{align*}
    H: \Pi_\Z \ra H_2(Y \smallsetminus \w)
\end{align*}
given by capping off a periodic domain with the compressing disks for the alpha and beta curves in its boundary. This is an isomorphism; see \cite[Proposition 2.15]{os_holodisks} and  \cite[Lemma 4.8]{zemke_graphcob}. Moreover, it is clear that real periodic domains are taken to $-\tau_*$-invariant elements of $H_2(Y\smallsetminus \w)$. In fact, we have the following.

\begin{lem}\label{lem:periodic_domains_iso}
    The map $H: \Pi_\Z^R \ra H_2(Y\smallsetminus \w)^{-\tau_*}$ is an isomorphism.
\end{lem}
\begin{proof}
    This follows just as in the proof of \cite[Lemma 4.8]{zemke_graphcob}. Injectivity is immediate. Surjectivity follows from the observation is that $Y \smallsetminus \w$ is obtained from $\Sigma$ by attaching 2-handles along the alpha and beta curves. A 2-cycle $Z$ generically intersects the cocores transversely, and therefore can homotoped to a surface which is homologous to a 2-chain $P$ in $\Sigma$ capped off by the cores of the handles it spans. Since a periodic domain is determined by its boundary, the $-\tau_*$-invariance of its boundary forces $P$ to be a real invariant domain. 
\end{proof}

Given $\x$ and $\y$ in $\Ta \cap \Tb$, there is an obstruction to the existence of a disk from $\x$ to $\y$, as follows. Let $a$ be a path from $\x$ to $\y$ in $\Ta$ and $b$ a path from $\x$ to $\y$ in $\Tb$. Their difference is an element of $\pi_1(\Sym^m(\Sigma))\cong H_1(\Sigma)$. Define $\varepsilon(\x, \y)$ to be the image of this loop under the identification 
\begin{align}\label{eqn:homology_of_Y_from_Sigma}
    \dfrac{H_1(\Sym^m(\Sigma))}{H_1(\Ta) \oplus H_1(\Tb)} \cong \dfrac{H_1(\Sigma)}{[\alpha_1], \hdots, [\alpha_m], [\beta_1], \hdots, [\beta_m]} \cong H_1(Y).
\end{align}
This class is independent of the choice of paths $a$ and $b$ and vanishes if and only if $\pi_2(\x,\y)$ is non-empty. 

If $Y$ is equipped with an involution, the path $a$ from $\x$ to $\y$ determines a path $b$, by applying the involution. The difference $(a-b)$ represents a class in $H_1(Y)^{-\tau_*}$. In fact, we can say slightly more. Fix a $\tau$-equivariant CW-complex for $Y$. There is a transfer map  
\begin{align*}
    \Theta: C^*(Y) \ra C^*(Y/\tau),
\end{align*}
defined as follows: if $e$ is a cell of $Y$ and $\beta$ is a cochain, we define 
\begin{align*}
    \Theta(\beta)([e]) = \beta(e) + \beta(\tau(e)).
\end{align*}
Descending to cohomology, this map fits into a commutative diagram:
\begin{align*}
    \begin{tikzcd}[ampersand replacement = \&]
        H^*(Y) \ar[rr,"1+\tau^*"]\ar[dr,"\Theta"] \& \& H^*(Y)^{\tau^*} \\
        \& H^*(Y/\tau). \ar[ur,"\pi^*"]
    \end{tikzcd}
\end{align*}
In particular, $\ker(\Theta) \subseteq \ker(1 + \tau^*) = H^*(Y)^{-\tau_*}$, and as we shall see, $PD(\varepsilon(\x,\y))$ is actually an element of $\ker(\Theta) \subset H^2(Y)^{-\tau_*}$ (see \Cref{lem:ker_theta}). 

In the real case, there is a refinement of $\varepsilon(\x,\y)$, which detects the existence of a real invariant domain from $\x$ to $\y$. Given a (possibly noninvariant) class $\phi \in \pi_2(\x,\y)$, consider its domain $\cD$ as a 2-chain in $\Sigma$. Applying $\tau$ to $\cD$ yields a new domain $\tau(\cD)$. Since $\tau$ is an orientation reversing involution of $\Sigma$, the domain $\tau(\cD)$ should be viewed as going from $\tau(\y)$ to $\tau(\x)$;  since $\x$ and $\y$ were assumed to be invariant intersection points, $\tau(\cD)$ represents a class in $\pi_2(\y, \x)$. Hence, $\cD + \tau(\cD)$ is an invariant periodic domain starting at $\x$. By capping it off, we obtain an element $H(\cD + \tau(\cD)) \in H_2(Y)^{\tau_*}$. If we had chosen a different domain, $\cD'$, the difference
\begin{align*}
    H(\cD + \tau(\cD)) - H(\cD' + \tau(\cD')) = (1 + \tau_*)(H(\cD - \cD'))
\end{align*}
would be in the image of $(1 + \tau_*): H_2(Y) \to H_2(Y)$.  Hence, assuming $\varepsilon(\x, \y) =0$, we can associate to $\x$ and $\y$ a well-defined class
\begin{align*}
    \zeta(\x, \y) \in \dfrac{H_2(Y)^{\tau_*}}{\operatorname{im}(1 + \tau_*)},
\end{align*}
which does not depend on the choice of domain. Note that $\zeta(\x, \y) = 0$ if there is a real invariant domain connecting $\x$ and $\y$. Conversely, if both $\epsilon(\x,\y) = 0$ and $\zeta(\x, \y) = 0$, there exists a real invariant domain from $\x$ to $\y$. Indeed, since $\epsilon(\x,\y) = 0$, there is an ordinary domain $\cD$ from $\x$ to $\y$. Since $\zeta(\x,\y) = 0$, the invariant periodic domain $\cD + \tau(\cD)$, viewed as a class in $H_2(Y)^{\tau_*}$, is in the image of $(1 + \tau)$, i.e., there is a periodic domain $\cP$ at $\x$ such that $\cP + \tau(\cP) = \cD + \tau(\cD)$. Hence, $\cD - \cP$ is a real invariant domain from $\x$ to $\y$.

Therefore, we may partition $(\Ta \cap \Tb)^R$ into equivalence classes, where $\x  \sim \y$ if and only if $\varepsilon(\x,\y) = 0$ and $\zeta(\x,\y) = 0$.


\begin{remark}
Here and later, when we write $\im(1+\tau_*)$, we mean the image of the map $(1+\tau_*)$ defined on $H_2(Y)$, not  on $H_2(Y)^{\tau_*}$. Similarly, after applying Poincar\'e duality, $\imt$ will mean the image of the map $(1+\tau^*)$ defined on $H^1(Y)$, not on $H^1(Y)^{\tau^*}$.
\end{remark}

\subsection{Real $\SpinC$-structures}
As in the monopole Floer setting, the real Heegaard Floer invariants split over real $\SpinC$-structures. In this section, we review the relevant definitions from \cite{li:HMR}. Then, following Turaev \cite{Turaev_Tors_Inv_SpinC}, we define a notion of real Euler structures, and show that these naturally correspond to real $\SpinC$-structures. 

\begin{defn}
    Let $(Y, \tau)$ be a real 3-manifold. A \emph{real vector bundle} over $(Y, \tau)$ is a complex vector bundle $E \ra Y$ equipped with a anti-complex linear involution $R: E \ra E$ which covers $\tau$. The data $(E, R)$ is called a \emph{real structure} on $(Y, \tau)$.
\end{defn}

Rather than work with pairs $(E, R)$, we will often pass to the associated $U(n)$-bundles; the involution determines a map $R:\Fr(E) \ra \overline{\Fr(E)}$ given by composition. Here, we need to work with the conjugate bundle $\overline{\Fr(E)}$, since the anti-complex linear map $R$ takes unitary frames $\C^n \xra{\simeq} E_y$ to anti-unitary frames $\C^n \xra{\simeq} E_y \xra{R_{y}} E_{\tau(y)}$. Compare to \cite[Section 2.2]{tian-wang}.

\begin{defn}\label{def:Real_SpinC_gauge_theory}
    Fix a $\SpinC$ structure, $\frs= (S,\rho)$ on $Y$, i.e. a Hermitian rank 2 bundle $S \ra Y$ together with a Clifford multiplication $\rho: TY \ra \End(S)$. A \emph{real structure compatible with $\frs$} is a real structure $R: S \ra S$ which is compatible with the Clifford multiplication, in the sense that the diagram
    \begin{align*}
        \begin{tikzcd}[ampersand replacement = \&]
            TY\ar[r, "\rho"] \ar[d,"d \tau"] \& \End(S) \ar[d,"\Phi_R"] \\
            TY \ar[r, "\rho"] \& \End(S) 
        \end{tikzcd}
    \end{align*}
    commutes; here, $R$ acts on $\End(S)$ by conjugation, $\Phi_R(f) = R \circ f \circ R$.
\end{defn}

Alternatively, we can define real $\SpinC$-structures in terms of principal bundles over $Y$. Recall that a $\SpinC$-structure on $Y$ can equivalently be defined to be a principal $\SpinC(3)$-bundle $P$ over $Y$ together a $U(1)$-equivariant covering $\phi: P \ra \Fr(Y)$, with respect to the map $\SpinC(3) \cong U(2) \ra U(2)/U(1) \cong SO(3)$. In particular, $P/U(1) \cong \Fr(Y)$, so we will often think of $\SpinC$-structures as pairs $(P, \eta)$, where $\eta: P/U(1) \cong \Fr(Y)$. To pass from the first definition to the second, we take $P = \Fr(S)$ and use the Clifford multiplication to identify $TY$ with the subbundle $\mathfrak{su}(S)$ of traceless, skew-adjoint endomorphisms of $S$. It follows that $\Fr(S)\ra \Fr(S)/U(1) \cong \Fr(Y)$ is a $\SpinC$ structure on $Y$ in the sense of the second definition. 

To define real $\SpinC$-structures in this language, we simply push the involution $R$ and its compatibility relations through the identification above. It is clear that the involution $R: S \ra S$ is equivalent to a bundle map
\begin{align*}
    I: \Fr(S) \ra \Fr(\overline{S}),
\end{align*}
covering $\tau$, where $\overline{S}$ is the conjugate bundle of $S$. The compatibility of $R$ with $\rho$ is equivalent to a compatibility of $I$ with the map induced by $d\tau$ on $\Fr(Y)$, i.e. we have a commutative diagram:
\begin{align*}
    \begin{tikzcd}[ampersand replacement = \&]
        \Fr(S)\ar[r, "I"] \ar[d] \&  \Fr(\overline{S}) \ar[d] \\
        \Fr(Y) \ar[r, "-d\tau"] \& \Fr(Y).
    \end{tikzcd}
\end{align*}
Note that this is equivalent to saying that the $U(1)$-bundle $P \ra P/U(1) \cong \Fr(Y)$ has a real structure.

By \cite[Lemma 3.5]{li:HMR}, a principal $U(1)$-bundle $L \ra Y$ admits a real structure precisely if $\Theta(c_1(L)) = 0$. Two real structures $(L_0, I_0)$ and $(L_1, I_1)$ are equivalent if there is an isomorphism $\phi: L_0 \ra L_1$ making the following diagram commute:
\begin{align*}
    \begin{tikzcd}[ampersand replacement = \&]
        L_0\ar[r, "I_0"] \ar[d,"\phi"] \&  L_0 \ar[d,"\phi"] \\
        L_1 \ar[r, "I_1"] \& L_1.
    \end{tikzcd}
\end{align*}
If $I_0$ and $I_1$ are two real structures on $L \ra Y$, they differ by a $\tau$-invariant map $\phi: Y \ra S^1$. Hence, equivalence classes of real structures on $L \ra Y$ are affinely given by
\begin{align}\label{eqn:count of real line bundles}
    H^1(Y)^{\tau^*}/\imt.
\end{align}
For our purposes, the following perspective is more useful. If $(L_0, I_0)$ and $(L_1, I_1)$ are real line bundles with $L_0 \cong L_1$, the obstruction to their equivalence can be represented by choosing any isomorphism $\phi: L_0 \ra L_1$, and comparing the maps $I_0$ and $\phi^{-1} \circ I_1 \circ \phi$, which differ by an element of $H^1(Y)^{\tau^*}$. The image of this class in $H^1(Y)^{\tau^*}/\imt$ does not depend on our choice of identification $\phi: L_0 \ra L_1$. If $\psi$ were a different isomorphism, then two obstructions differ by 
\begin{align*}
    (1 + \tau^*)[\phi - \psi] \in H^1(Y)^{\tau^*}.
\end{align*}
Hence, the obstruction to the equivalence of the two bundles lies in  $H^1(Y)^{\tau^*}/\imt$.

\subsection{Real Euler Structures}
\label{sec:RealEuler}
We start by recalling Turaev's interpretation of $\SpinC$-structures from \cite{Turaev_Tors_Inv_SpinC}. We say that two nonvanishing vector fields $v_0$ and $v_1$ on $Y$ are \emph{homologous} if $v_0|_{Y \smallsetminus B^3}$ and $v_1|_{Y \smallsetminus B^3}$ are homotopic. An Euler structure on a 3-manifold $Y$ is a homology class of nonvanishing vector fields on $Y$.  We will denote the homology class of $v$ by $[v]$ and will write $\Vec(Y)$ for the space of Euler structures on $Y$. The only obstruction to a homotopy between $v_0$ and $v_1$ in $Y \smallsetminus B^3$ is a class $[v_0-v_1] \in H^2(Y)$. This endows $\Vec(Y)$ with an action of $H^2(Y)$, and it is straightforward to see that this action is free and transitive. It follows that $\Vec(Y)$ is an affine space over $H^2(Y)$. Given a vector field $v$ on $Y$, we will write $\overline{v}$ for its reverse, given by $\overline{v}_x = - v_x$ at all $x \in Y$. 

Given a vector field $v$ on a real manifold $(Y, \tau)$, we can of course push $v$ forward by $d\tau$. Let $\tau^*(v)$ be the resulting  vector field. We say that a vector field $v$ on $(Y, \tau)$ is \emph{real} if $\tau^*(v) = \overline{v}$. 

A real vector field on $(Y, \tau)$ naturally determines a real $\SpinC$-structure rise to a complex line bundle on $Y$ equipped with a real structure, as follows. By choosing an orientation and a $\tau$-invariant Riemannian metric on $Y$, we may decompose $TY \cong \langle v \rangle \oplus \langle v \rangle^\perp$. The choices of orientation and metric make  $\langle v \rangle^\perp \ra Y$ a complex line bundle, and the restriction of $d\tau$ to $\langle v \rangle^\perp \ra Y$ is a complex anti-linear involution covering $\tau$. We will write $R_0$ for this restriction.

It is easy to see that if $v_0$ and $v_1$ are homologous then $\langle v_0 \rangle^\perp$ and $\langle v_1 \rangle^\perp$ are isomorphic. However, when $v_0$ and $v_1$ are real, $(\langle v_0 \rangle^\perp, R_0)$ and $(\langle v_1 \rangle^\perp, R_1)$ are not necessarily equivalent \emph{real} line bundles. Roughly, we would like to define \emph{real Euler structures} to be equivalence classes of real vector fields which determine equivalent real complex line bundles. 

\begin{lem}
    Every real 3-manifold with fixed set of codimension two admits a real vector field. Furthermore, for any real vector field $v$ and $[x] \in \ker(\Theta) \subset H^2(Y)$, the Euler structure $[v] + [x]$ can be represented by a real vector field whose restriction to the 1-skeleton of $Y$ is homotopic to $v$ though real vector fields.
\end{lem}
\begin{proof}
    To construct a real vector field, one can fix a $\tau$-invariant cell structure on $Y$, choose some real vector field on $C$, and extend over the free cells in pairs. (See Section~\ref{sec:HDSpinC} for an alternative construction using Morse theory.)

    To prove the second claim, assume $v$ is a given real vector field and let $x \in \ker(\Theta) \sub H^2(Y)$ be represented by $x \in C^2(Y)$. We will build a real vector field $u$ such that $[v - u] = [x]$. Define $u$ on the 1-skeleton of $Y$ to agree exactly with $v$. For a 2-cell $e_2$, we extend $u$ so that $[v - u](e_2) = x(e_2)$ on the chain level. For $u$ to be real, its extension over $\tau(e)$ is determined; indeed $[v-u](\tau(e_2)) = -x(e)$ (again, on the chain level). However, by assumption, $[x] \in \ker(\Theta)$, so there exists some $y \in C^1(Y/\tau)$ so that 
    \begin{align*}
        \Theta(x)([e_2]) = x(e_2) + x(\tau(e_2)) = \delta y([e_2]).
    \end{align*}
    Defining $\tilde{y} = y \circ \pi$, for $\pi: Y \ra Y/\tau$, we see that 
    \begin{align*}
        -x(e_2) = x(\tau(e_2)) + \delta \tilde{y}(e_2).
    \end{align*}
    Hence, up to a coboundary, $[v-u](\tau(e_2)) = x(\tau(e))$ as desired. Any extension of $u$ over the 3-skeleton is sufficient, as we only care about the homology class. Therefore, $u$ is a real vector field representing $[v] + [x]$.
\end{proof}

We say that an Euler structure \emph{admits a real structure} if it can be represented by a real vector field. According to the previous lemma, there is a well-defined action of $\ker(\Theta)$ on the set of Euler structures admitting a real structure.

\begin{lemma}\label{lem:ker_theta}
    The set of Euler structures on $(Y, \tau)$ which admit a real structure is a torsor over $\ker(\Theta) \subset H^2(Y)$.
\end{lemma}
\begin{proof}
    Given $[u], [v] \in \Vec(Y)$, the primary obstruction to their homotopy in $Y \smallsetminus B^3$ is represented by a class $[u- v]$ in $C^2(Y)$ which assigns to a 2-cell $e \sub Y$ the degree of the map 
    \begin{align*}
        (u|_e \cup v|_{e}) \in \pi_2(S^2).
    \end{align*}
    Since $S^2$ is simply connected, we may assume $u$ and $v$ are constant on the 1-skeleton of $Y$, and therefore we assume that $(u|_e \cup v|_{e})$ factors through $S^2 \vee S^2$. It follows that $[u-v](e) = \deg(u|_e) - \deg(v|_e)$. Since $\tau$ is a diffeomorphism, it follows that $\tau^*([u- v]) = [\tau^*(u) - \tau^*(v)]$. 
    
    Assume that $v_0$ is a real vector field on $(Y, \tau)$. Consider the composition
    \begin{align*}
        \Vec(Y) \xra{[v] \mapsto [v-v_0]} H^2(Y) \xra{\Theta} H^2(Y/\tau).
    \end{align*}
    Let $[v] \in \Vec(Y)$. Observe that for a 2-cell $e \in C_2(Y)$, we have
    \begin{align*}
        \Theta([v-v_0])([e]) &= [v-v_0](e) + [\tau^*(v)-\tau^*(v_0)](e) \\
        &= [v-v_0](e) + [\tau^*(\overline{v_0})-\tau^*(\overline{v})](e) \\
        & = [v-\tau^*(\overline{v})](e) + [\tau^*(\overline{v_0})-v_0](e).\\
        & = [v-\tau^*(\overline{v})](e) + 0.
    \end{align*}
    Therefore, if $[v]$ is an Euler structure represented by a real vector field $v$, then the class $[v-v_0]$ is contained in the kernel of $\Theta$. Conversely, if $[v-v_0] \in \ker(\Theta)$, the previous lemma implies that $[v] = [v_0] + [v-v_0]$ can be represented by a real vector field. 
\end{proof}

\begin{definition}
\label{def:realEuler}
We say that two real vector fields $v_0$ and $v_1$ are \emph{real homologous} if they are homologous and the associated real complex line bundles $(\langle v_0 \rangle^\perp, R_0)$ and $(\langle v_1 \rangle^\perp, R_1)$ are equivalent. We define a \emph{real Euler structure on $(Y,\tau)$} to be a real homology class of real vector fields on $Y$ and we write $\RVec(Y, \tau)$ for the set of real Euler structures. 
\end{definition}

\begin{corollary}
\label{cor:classification}
Let $\Theta: H^2(Y) \to H^2(Y/\tau)$ be the transfer map. The set $\RVec(Y, \tau)$ of real Euler structures on $(Y, \tau)$ is in non-canonical bijection with the group
$$ \ker(\Theta) \oplus \bigl( H^1(Y)^{\tau^*}/\imt\bigr).$$
More precisely, the image of the restriction map
$$ \RVec(Y, \tau) \to \Vec(Y)$$
 is a torsor over $\ker(\Theta)$, and the preimage of any fixed element in that image is a torsor over $\bigl( H^1(Y)^{\tau^*}/\imt\bigr).$
\end{corollary}
\begin{proof}
According to \Cref{lem:ker_theta}, the set of Euler structures which admit real structures form a torsor over $\ker(\Theta)$ and, essentially by definition, the set of possible real Euler structures such a class can admit is a torsor over $H^1(Y)^{\tau^*}/\imt$;  see \Cref{eqn:count of real line bundles}.
\end{proof}

\begin{remark}
    We note that we do \emph{not} claim that the set $\RVec(Y, \tau)$ is a torsor over $\ker(\Theta) \oplus \bigl( H^1(Y)^{\tau^*}/\imt\bigr)$, as there may be extension problems. Indeed, there does not seem to be a natural action of $\ker(\Theta)$ on $\RVec(Y, \tau)$. The same caveat applies to the set of real $\SpinC$ structures, which is in natural bijection with $\RVec(Y, \tau)$; see Lemma~\ref{lem:equivreal} below.
\end{remark}

We can describe the real homology relation more intrinsically as follows. If $v_0$ and $v_1$ are real vector fields in the same homology class, there is an obstruction to their homotopy through real vector fields. As $v_0$ and $v_1$ are assumed to be homologous, there is some homotopy $v_t$ between $v_0$ and $v_1$ in $Y^3 \smallsetminus B^3$. Since $v_0$ and $v_1$ are real vector fields, the pushforward of $\overline{v_t}$ by $d\tau$ yields another homotopy between $v_0$ and $v_1$. The obstruction to a homotopy between $v_t$ and $\tau^*(\overline{v_t})$ over the 1-skeleton of $Y$ is a class $[v_t - \tau^* (\overline{v_t})] \in H^1(Y;\pi_2(S^2)).$ Clearly this class vanishes if the vector field $v_t$ is itself real for each $t$. Furthermore, this class is $\tau^*$-invariant:
\begin{align*}
    \tau^*([v_t - \tau^* (\overline{v_t})]) &= [\tau^*v_t - (\overline{v_t})] \\
    &= [-\tau^*(\overline{v_t}) + v_t]\\
    &= [v_t - \tau^*(\overline{v_t})].
\end{align*}
Of course, this class depends on our choice of homotopy; if $v_t'$ is another path between $v_0$ and $v_1$, we may obtain a different element $[v' - \tau^* (\overline{v'_t})] \in H^1(Y)^{\tau_*}$, but their difference 
\begin{align*}
    [v - \tau^* (\overline{v_t})] - [v' - \tau^* (\overline{v'_t})] & = (1 + \tau^*)([v - v'])
\end{align*}
lies in the image of $(1 + \tau^*)$. Hence, there is a well-defined obstruction, $\zeta(v_0, v_1)$, to the existence of a homotopy from $v_0$ to $v_1$ through real vector fields, which resides in $H^1(Y)^{\tau^*}/\im(1 + \tau^*).$ \\

To see the equivalence of these two perspectives, it is useful to consider explicitly how the group $H^1(Y)^{\tau^*}/\imt$ acts on the set of real structures. Let $v_0$ be a real vector field and let $[x]\in H^1(Y)^{\tau^*}$ be represented by the Poincar\'e dual of a $\tau$-invariant surface $Z$ in $Y$ ($[x]$ can be represented by a $\tau$-invariant map $Y \ra S^1$; take $Z$ to be the preimage of a regular value of this map). We will define $v_0 + [x]$ by choosing some homotopy of $v_0$ supported on $\nu(Z)$. Concretely, fix a $\tau$-equivariant CW structure for $Y \smallsetminus \nu(Z)$, and extend this to an equivariant CW structure for $Y$; use $[x]$ to choose a homotopy $v_t$ of $v_0$ on the 1-cells of this extension so that for each such 1-cell in the extension, we have that $[v_t - \tau^*(\overline{v_t})](e) = x(e)$. We only require that the homotopy $v_t$ be real at time $t = 0,1$. Extend $v_1$ over the remaining 2-cells in any way, provided that $v_0$ and $v_1$ are homotopic. By construction, the resulting vector fields are homologous, but the obstruction $\zeta(v_0, v_1)$ is nontrivial, and is Poincar\'e dual to $Z$. 

Consider the associated real complex line bundles $(\langle v_0 \rangle^\perp, R_0)$ and $(\langle v_1 \rangle^\perp, R_1)$. The homotopy $v_t$ from $v_0$ to $v_1$ determines an isomorphism $\langle v_0 \rangle^\perp \cong \langle v_1 \rangle^\perp$, allowing us to compare $R_0$ and $R_1$, giving an equivariant map $Y \ra S^1$. By construction, this map is supported on $\nu(Z)$ (as $v_t$ is constant outside $\nu(Z)$). Hence, the obstruction to the equivalence of these two real structures as well as the obstruction $\zeta(v_0, v_1)$ to a real homotopy agree in $H^1(Y)^{\tau^*}/\im(1 + \tau^*),$ as in this quotient, they lie in the same equivalence class as $\mathrm{PD}(Z)$. \\

There is a natural map 
\begin{align*}
    \mu: \Vec(Y, \tau) \ra \SpinC(Y, \tau),
\end{align*}
as in \cite{Turaev_Tors_Inv_SpinC}. We briefly review its construction following the treatment in \cite[Lemma 2.2]{Lipshitz_cylindrical}. Fix a Riemannian metric on $Y$. A non-vanishing vector field $v$ on $Y$ gives rise to a splitting $TY \cong \langle v\rangle \oplus \langle v\rangle^\perp$. This reduces the structure group of $TY$ from $SO(3)$ to $SO(1) \oplus SO(2)$, determining a principal $\SpinC(3)$-bundle
\begin{align*}
    \mu(v) := \Fr(v^\perp) \times_{U(1)} U(2)
\end{align*}
over $Y$ which covers $\Fr(Y)$. Hence, $\mu(v)$ is a $\SpinC$-structure on $Y$. The bundle $\mu(v) \ra Y$ is determined by its restriction to $Y \smallsetminus B^3$, and therefore, $\mu$ descends to a map $\Vec(Y) \ra \SpinC(Y)$. 

In the real case, we can upgrade this to a map
\begin{align*}
    \mu^R: \RVec(Y, \tau) \ra \RSpinC(Y, \tau),
\end{align*}
as follows. If $v$ is a real (non-vanishing) vector field, then $d\tau$ restricts to an involution $I(v)$ of the rank two real vector bundle $\langle v \rangle^\perp \ra Y$. As $v$ is real, this involution is complex anti-linear in the fibers, and therefore induces a structure on $\mu(v)$. In particular, the pair $$\mu^R(v) := (\mu(v), I(v))$$ is a real $\SpinC$-structure for $(Y,\tau)$. By definition of real Euler structures, this descends to a map $\RVec(Y, \tau) \ra \RSpinC(Y, \tau)$ as promised.

\begin{lem}
\label{lem:equivreal}
    The map 
    \begin{align*}
        \mu^R: \RVec(Y, \tau) \ra \RSpinC(Y, \tau)
    \end{align*}
    is a bijection, and intertwines the actions of $ H^1(Y)^{\tau^*}/\imt$. 
\end{lem}
\begin{proof}
We have a commutative diagram \begin{align*}
        \begin{tikzcd}[ampersand replacement = \&]
            \RVec(Y, \tau)\ar[r, "\mu^R"] \ar[d] \& \RSpinC(Y, \tau) \ar[d] \\
            \Vec(Y) \ar[r, "\mu"] \& \SpinC(Y).
        \end{tikzcd}
    \end{align*}   
     It follows from \cite{Turaev_Tors_Inv_SpinC,Lipshitz_cylindrical} that the bottom horizontal map $\mu$ is a $\ker(\Theta)$-equivariant bijection. Furthermore, if we look at either vertical map, the preimage of any element in its image is a torsor over $ H^1(Y)^{\tau^*}/\imt$. Thus, to show bijectivity of $\mu^R$ it suffices to prove that this map $H^1(Y)^{\tau^*}/\imt$-equivariant. 

    Assume $v_0$ and $v_1$ are homologous real vector fields. We saw that the obstruction to the existence of a real homotopy between $v_0$ and $v_1$ agrees with the obstruction to the equivalence of the real line bundles $(\langle v_0 \rangle^\perp, R_0)$ and $(\langle v_1 \rangle^\perp, R_1)$. However, the obstruction to the equivalence of the real line bundles $(\langle v_0 \rangle^\perp, R_0)$ and $(\langle v_1 \rangle^\perp, R_1)$ is identified with the obstruction to the equivalence of $\mu^R(v_0)$ and $\mu^R(v_1)$ under the isomorphism $H^1(Y) \xra{\pi^*} H^1(\Fr(Y))$. 
\end{proof}

\subsection{Real Heegaard diagrams and real $\SpinC$-structures}
\label{sec:HDSpinC}

Fix a real pointed Heegaard diagram $(\cH, \tau)$ for $(Y, \tau)$. In \cite[Section 2.6]{os_holodisks} and \cite[Section 3.1]{os_linkinvts}, Ozsv\'ath and Szab\'o define a function 
\begin{align*}
    \frs_\w: \Ta \cap \Tb \ra \SpinC(Y),
\end{align*}
 as follows. Using the diagram $\cH$, they construct a self-indexing Morse function $f: Y \ra [0,3]$ with $g$ critical points of index $1$ and $2$ and $|\w|$ critical points of index $0$ and $3$, such that its level set $F^{-1}(3/2)$ is the Heegaard surface. Each $\x \in  \bT_\alpha \cap \bT_\beta$ determines a $g$-tuple of trajectories for the gradient flow of $f$ connecting critical points of index $1$ and $2$ as well as $|\w|$ trajectories between critical points of index $0$ and $3$ such that each $w_i \in \w$ is contained in exactly one trajectory. On the complement of a neighborhood of these trajectories the gradient of $f$ is non-vanishing, and by the Poincar\'e-Hopf theorem, this this vector field can be modified near the flow lines to be nonvanishing. This specifies an Euler-structure, which in turn determines a $\SpinC$-structure in $Y$, which is denoted $\frs_\w(\x)$. The difference between two such $\SpinC$-structures is measured precisely by the class $\varepsilon$:
\begin{align*}
    \frs_\w(\y) - \frs_\w(\x) = PD([\varepsilon(\x, \y)]).
\end{align*}
Hence, the equivalence classes of intersection points determined by $\varepsilon(\x, \y)$ are exactly given by the associated $\SpinC$-structures.
    
In the real setting, there is a map
\begin{align*}
    \frs_\w^R: (\Ta \cap \Tb)^R \ra \RVec(Y) \cong \RSpinC(Y).
\end{align*}
Again, using the real Heegaard diagram $(\cH, \tau)$, we construct a self-indexing Morse function $f: Y \ra \R$, now with the additional property that $f\circ \tau = 3 - f$. (This can be done by first defining it in a neighborhood of the Heegaard surface, extending it to a handlebody, and then using reflection to specify it on the other handlebody.) The index 1 and 2 critical points of such Morse functions come in pairs, which are exchanged by the involution. Furthermore, the ascending manifold of an index 1 critical point $p$ is mapped to the descending manifold of the index 2 critical point $\tau(p)$. An intersection point $\x \in (\Ta \cap \Tb)^R$ determines a $m$-tuple of flows between critical points. Since $\x$ is an invariant intersection point, the associated flows are also preserved by the involution (though the orientations are reversed). Moreover, since $f \circ \tau = 3 - f$, we have that $df \circ d \tau = -df$. It follows that in the complement of a neighborhood of the flow lines determined by $\x$ and $\w$, the gradient $\nabla f$ is a nonvanishing anti-invariant vector field. For the neighborhoods of the flow lines between index 1 and 2 critical points which appear in pairs and we can extend $\nabla f$ to a nonvanishing vector field over these 3-balls in pairs. Flow lines between critical points of index 1 and 2 which are exchanged by the involution intersect the fixed point set; the same is true for the flow lines between the index 0 and 3 critical points. Neighborhoods of these flow lines are diffeomorphic to the standard 3-ball in $\R^3$ rotating around the $x$-axis. Since the restriction of the gradient vector field to the boundary is degree zero, there is a homotopy $H: S^2 \times I \ra \R^3 \smallsetminus \{0\}$ to a constant vector field; we may assume this vector field is anti-invariant. By replacing $H$ with $\frac{1}{2}(H-d\tau \circ H)$, we may assume the homotopy is anti-invariant, and by perturbing $H$ equivariantly we may assume $H$ is nonzero. This defines an anti-invariant extension over all of $Y$, and therefore specifies a real Euler structure on $(Y, \tau)$, and hence a real $\SpinC$-structure $\frs_\w^R(\x)$. Given $\frs_\w^R(\x)$, we write $\frs_\w(\x)$ for the underlying $\SpinC$-structure. 

\begin{lem}\label{lem:change_RSpinC}
    Let $\frs_{\w}^R(\x)$ and $\frs_\w^R(\y)$ be two real $\SpinC$ structures. Then, 
    \begin{align*}
        \frs_{\w}(\y) - \frs_{\w}(\x) = \mathrm{PD}(\varepsilon(\x,\y)).
    \end{align*}
    Moreover, if $\varepsilon(\x,\y) = 0$, we have that 
    \begin{align*}
       \frs_\w^R(\y) - \frs_{\w}^R(\x) = \mathrm{PD}(\zeta(\x, \y)).
    \end{align*}
\end{lem}
\begin{proof}
    The first statement follows from \cite{Turaev_Tors_Inv_SpinC} and \cite{os_holodisks}. Though, to motivate the proof of the second statement, we recall the proof in \cite{Lipshitz_cylindrical}. 

    Let $v_\x$ and $v_\y$ be real vector fields inducing $\frs_{\w}^R(\x)$ and $\frs_{\w}^R(\y)$. Choose a map $A: Y \ra SO(3)$ which satisfies $v_\y = A v_\x$. The $\SpinC$-structures $\frs_{\w}(\x)$ and $\frs_{\w}(\y)$ are equivalent if and only if $A|_{Y\smallsetminus B^3}$ is homotopic to a map $Y \ra SO(2)$. This is detected by the homotopy class of the map 
    \begin{align}
        Y \smallsetminus B^3 \xra{A} SO(3) \ra SO(3)/SO(2) \cong S^2,
    \end{align}
    which represents the obstruction class $\frs_{\w}(\y) - \frs_{\w}(\x) \in H^2(Y\smallsetminus B^3) \cong  H^2(Y)$. Since $v_\x$ and $v_\y$ are identical away from the flow lines $\gamma_\x$ and $\gamma_\y$ through $\x$ and $\y$, and the map $h_A$ is homotopic to a Thom collapse map of a neighborhood of a smoothing of $\gamma_\x \cup \gamma_\y$. Hence, for a regular value $p \in S^2$, $[h_A^{-1}(p)] = \varepsilon(\x,\y)$ which is therefore Poincar\'e dual to $\frs_{\w}(\y) - \frs_{\w}(\x)$. 

    Now, assume that $\varepsilon(\x,\y) = 0$, so that the vector fields $v_\x$ and $v_\y$ are homologous. A choice of homotopy between them determines an isomorphism $\Phi$ between the associated frame bundles. The obstruction $\frs_\w^R(\y) - \frs_{\w}^R(\x)$ to the equivalence of the two real structures is given by comparing $I(v_\x)$ with $(\overline{\Phi})^{-1}\circ I(v_\y)\circ \Phi$, which determines a $\tau$-invariant map $g_\phi: Y \ra S^1$. Again, $v_\x$ and $v_\y$ are identical away from $\gamma_\x \cup \gamma_\y$, and moreover, we can take the homotopy between them to be supported in a neighborhood of a surface $Z$, obtained by capping off some periodic domain $\cD + \tau(\cD)$. But, as $\Phi$ is determined by the homotopy, $g_\Phi$ is also supported in a neighborhood of $Z$. Hence, just as in the usual case, for a regular value $p$ of $g_\Phi$, we have that $[g_\Phi^{-1}(p)] = [Z] \in H_2(Y)^{\tau_*}/\im(1+\tau_*)$ is Poincar\'e dual to both the obstruction $\frs_\w^R(\y) - \frs_{\w}^R(\x)$ to the equivalence of the real $\SpinC$-structures as well as to the obstruction $\zeta(v_\x,v_\y)$ to the existence of a real homotopy from $v_\x$ to $v_\y$. 
\end{proof}

Hence, elements of $(\Ta \cap \Tb)^R$ are partitioned according to real $\SpinC$-structures.

\subsection{Relative real $\SpinC$-structures}
\label{sec:relative}

In ordinary Heegaard Floer theory, two intersection points $\x$ and $\y$ are connected by a domain if and only if they are connected by a domain that avoids the basepoints $\w$. Indeed, we can always add or subtract copies of the components of $\Sigma \smallsetminus \bm \alpha$ to arrange so that the domain has multiplicity zero at all $w \in \w$.
In terms of $\SpinC$-structures, this is due to the fact that the natural map
\begin{equation}
\label{eq:spinrel}
     \SpinCrel(Y, \w) \to \SpinC(Y)
\end{equation}
is a bijection. Here, we denoted by $\SpinCrel(Y, \w)$ the space of relative $\SpinC$ structures on $Y \smallsetminus \w$; that is, those that are fixed in a neighborhood of $\w$. (Our notation is inspired by that in \cite[Section 2.3]{os_knotinvts}.) The set $\SpinCrel(Y, \w)$ is affinely identified with 
$$H^2(Y\setminus \nbhd(\w), \del (Y\setminus \nbhd(\w))) \cong H_1(Y\smallsetminus \w),$$
whereas $\SpinC(Y)$ is affinely identified with $H^2(Y) \cong H_1(Y)$. The fact that \eqref{eq:spinrel} is a bijection is a consequence of the fact that
$$ H_1(Y\smallsetminus \w) \to H_1(Y)$$
is an isomorphism.

The situation is quite different in the real setting. We have a set of relative real $\SpinC$ structures $\RSpinCrel(Y, \tau,\w)$, but the natural map
\begin{equation}
\label{eq:Rspinrel}
     \RSpinCrel(Y, \tau, \w) \to \RSpinC(Y, \tau)
\end{equation}
is no longer injective.

\begin{lemma}
    The map \eqref{eq:Rspinrel} is surjective.
\end{lemma}

\begin{proof}
The discussion in Sections~\ref{sec:RID} through \ref{sec:RealEuler} 
can be adapted to the complement $Y \smallsetminus \w$. We have a set of relative real Euler structures $\RVecrel(Y, \tau, \w)$ (that are fixed in a neighborhood of $\w$), which is naturally identified with  $\RSpinCrel(Y, \tau, \w)$. An analogue of Corollary~\ref{cor:classification} shows that the space of relative $\SpinC$ structures that admit a (relative) real structure is still affine over the same $\ker(\Theta) \subset H^2(Y)$, but the space of relative real $\SpinC$ structures in the same relative $\SpinC$ structure is now affine over
$$\frac{H^1(Y\setminus \nbhd(\w), \del (Y\setminus \nbhd(\w)))^{\tau^*}}{\imt} \cong \frac{H_2(Y\smallsetminus \w)^{\tau_*}}{\imth}.$$
Thus, the map \eqref{eq:Rspinrel} is affinely identified with the following map induced by inclusion:
\begin{equation}
\label{eq:Rspinrel2}
     \frac{H_2(Y\smallsetminus \w)^{\tau_*}}{\imth} \to \frac{H_2(Y)^{\tau_*}}{\imth}. 
\end{equation}
On the left, by $\imth$ we mean the image of the map on $H_2(Y \smallsetminus \w)$, whereas on the right we mean the one on $H_2(Y)$.

To see that the map \eqref{eq:Rspinrel2} is surjective, it suffices to prove that
\begin{equation}
    \label{eq:Rspinrel3}
H_2(Y \smallsetminus \w)^{\tau_*} \to H_2(Y)^{\tau_*}
\end{equation}
is surjective. Given a $\tau_*$-invariant class $x \in H_2(Y)$, its Poincar\'e dual $\mathrm{PD}(x)$ can be realized as a $\tau$-invariant map $f_x:Y \ra S^1$; according to \cite{wasserman}, $f_x$ can be approximated by a smooth equivariant map, and we can represent $x$ by the preimage of a regular value. It follows that $x$ is represented by a surface $S$ satisfying $\tau(S) = S$. Furthermore, let us arrange so that $S$ is disjoint from $\w$. If $S \cap \w$ is nonempty, it is contained in $C$. Locally, such an intersection can modeled by the intersection of the $x$-axis with the $yz$-plane in $\R^3$, where $\tau$ acts by rotation around the $x$-axis; here, we make use of the fact that $S$ is $\tau$-invariant, rather than anti-invariant. From this picture, it is clear that we can perform an equivariant perturbation of $S$ to ensure it is disjoint from $\w$. Hence, $S$ also represents a class in $H_2(Y \smallsetminus \w)^{\tau_*}$, which maps onto $x$. 

It follows that the map \eqref{eq:Rspinrel3} is surjective, and hence so are \eqref{eq:Rspinrel2} and \eqref{eq:Rspinrel}. 
\end{proof}
\begin{example}
    To see that \eqref{eq:Rspinrel2} and hence \eqref{eq:Rspinrel} need not be injective, let $Y$ be a rational homology sphere equipped with two basepoints, $\w = \{w_1, w_2\}$. Then 
    \begin{align*}
        \frac{H_2(Y)^{\tau_*}}{\imth} = 0,
    \end{align*}
    but 
    \begin{align*}
        \frac{H_2(Y\smallsetminus \w)^{\tau_*}}{\imth} \cong \Z/2.
    \end{align*}
    Indeed, in this case, $H_2(Y\smallsetminus \w)^{\tau_*}$ is generated by the class $[\partial \nbhd(w_1)]$, on which $\tau_*$ acts by the identity. 
\end{example}
The discussion in Section~\ref{sec:HDSpinC} can be applied to $Y \smallsetminus \w$ instead of $Y$. The result is a decomposition of the real intersection points into more refined equivalence classes, according to whether they are connected by a real domain that avoids the basepoints. This is encoded in a map
\begin{align*}
    \frsrel_\w^R: (\Ta \cap \Tb)^R \ra  \RSpinCrel(Y, \tau, \w).
\end{align*}

\subsection{Real Admissibility} 

In order to obtain well-defined invariants in Heegaard Floer theory, one must work with \emph{admissible diagrams}. Given a $\SpinC$ structure $\frs \in \SpinC(Y)$, we say that a Heegaard diagram $\cH$ is \emph{strongly $\frs$-admissible} if for every $N > 0$ and each non-trivial periodic domain $P \in \Pi_\Z$, the inequality
\begin{align*}
    \langle c_1(\frs), H(P) \rangle = 2N \ge 0
\end{align*}
implies that $P$ has some multiplicity greater than $N$. Strong $\frs$-admissibility ensures that, for a fixed integer $j$, there are only finitely many classes $\phi \in \pi_2(\x, \y)$ which satisfy $\ind(\phi) = j$ and $\cD(\phi) \ge 0.$ 

There is also a weaker notion. A Heegaard diagram is \emph{weakly $\frs$-admissible} if each nontrivial periodic domain $P$ with 
\begin{align*}
    \langle c_1(\frs), H(P) \rangle = 0 
\end{align*}
has positive and negative multiplicities. Strong admissibility implies weak admissibility. The weaker notion is sometimes useful because it can be achieved for all $\SpinC$ structures at once.

\begin{lem}
    Suppose $\cH$ is a strongly $\frs$-admissible real Heegaard diagram for an ordinary $\SpinC$-structure $\frs$. Suppose that $\x, \y \in (\Ta \cap \Tb)^R$ are such that $\frs^R_\w(\x) = \frs^R_\w(\y)$ and both $\frs^R_\w(\x)$ and $ \frs^R_\w(\y)$ have underlying $\SpinC$-structure $\frs$. Then, there are only finitely many classes $\phi \in \pi^R_2(\x, \y)$ with $\ind_R(\phi) = j$ and $\cD(\phi) \ge 0.$
\end{lem}
\begin{proof}
    A strongly admissible real Heegaard diagram is admissible; since the real index is determined by the ordinary index using Proposition~\ref{prop:indR}, the result follows exactly as in the usual setting. See \cite{os_holodisks,os_linkinvts,zemke_graphcob} for the proof in that setting. 
\end{proof}

Every Heegaard diagram which realizes a $\SpinC$-structure $\frs$ is isotopic to a strongly $\frs$-admissible diagram. The way this is typically done is by winding the alpha curves; to obtain admissible real Heegaard diagrams, we will wind the alpha curves to obtain an admissible diagram, and then perform the symmetric winding of the beta curves, and show the resulting diagram is still admissible. 

Fix a real Heegaard diagram $\cH = (\Sigma, \balpha, \bbeta, \w)$. Following \cite{zemke_graphcob}, we fix some auxiliary data: let $\alpha_1, \hdots, \alpha_g \sub \balpha$ be a collection of alpha curves  and let $\eta_1, \hdots, \eta_g$ be a collection of curves which are geometrically dual to $\alpha_1, \hdots, \alpha_g$ and span $H_1(Y)$. Define $\theta_1, \hdots, \theta_g$ by $\theta_i = \tau(\eta_i)$, which are geometrically dual to the beta curves. We note that it may be the case that $\theta_i = \eta_i$. 

The surface $\Sigma \smallsetminus(\alpha_1\cup\hdots\cup \alpha_g)$ is a connected, planar surface. Let $r = |\w| $ and choose a collection of arcs $\lambda_1, \hdots, \lambda_{r-1}$ in $\Sigma^R$ such that each $\lambda_i$ has boundary on two distinct basepoints of $\w$ such that each basepoint of $\w$ is the boundary of at least one $\lambda_i$ and $\lambda_1 \cup \hdots \cup \lambda_{r-1}$ is connected. Let $\rho_1 \cup \hdots \cup \rho_{r-1}$ be the images of these arcs under $\tau$.

Following \cite[Proposition 4.11]{zemke_graphcob}, given a real Heegaard diagram $(\cH, \tau)$ and positive integer $N$, we define a new diagram $(\cH_N, \tau)$ given by performing the following moves: 

\begin{enumerate}
    \item (Winding): Let $\eta_i^\pm$ and $\theta_i^\pm$ be two small parallel push offs of $\eta_i$ and $\theta_i$ respectively. Wind the alpha curves $N$ times positively around $\eta_i^+$ and $N$ times negatively around $\eta_i^-$; wind the beta curves $N$ times positively around $\theta_i^+$ and $N$ times negatively around $\theta_i^-$;   
    \item (Zigzagging): For each $\lambda_i$ choose a small rectangle $R_i$ which intersects $\lambda_i$ in a connected arc, intersects no other $\lambda_j$, and contains all intersections of $\lambda_i$ with the various attaching curves. The reflection $\tau(R_i)$ plays an analogous role for $\rho_i$. Zigzag each of the alpha curves along $\lambda_i$ inside $R_i$ and each of the beta curves along $\rho_i$ inside $\tau(R_i)$. 
\end{enumerate}
See \Cref{fig:wind_zigzag} for an illustration of the two moves.

We shall borrow some notation from \cite{zemke_graphcob}. Given a oriented path $\lambda$ in $Y$ connecting basepoints $w$ and $w'$, we may homotope it to an immersed arc in $\Sigma$ which intersects the alpha and beta curves transversely and avoids their intersection points. It may be that $w= w'$. Let $a_1, \hdots, a_k$ be an enumeration of the intersection points between $\lambda$ and $\bm \alpha.$ Given a class $\phi \in \pi_2(\x, \y)$, define 
\begin{align*}
    a(\lambda, \phi) = \sum_{i = 1}^k d_i^{\alpha, \lambda}(\phi), 
\end{align*}
where $d_i^{\alpha, \lambda}$ is the difference between the multiplicities of $\phi$ on the two sides of $a_i$, using the orientation of $\lambda$. Equivalently, if $\cD(\phi)$ is the domain corresponding to $\phi$, we may define $a(\phi) = \partial_\alpha \cD(\phi)$ and 
\begin{align}\label{eqn: arc-domain count}
    a(\lambda, \cD(\phi)) = \lambda \cdot a(\phi).
\end{align}
We will make use of this notation again in \Cref{sec:def}.

\begin{lem}\label{lem:real admissibility}
    If $(\cH, \tau) = (\Sigma, \balpha, \bbeta, \w, \tau)$ is a real Heegaard diagram for $(Y, \tau)$ and $\frs \in \RSpinC(Y, \tau)$ is a fixed real $\SpinC$-structure, then $\cH$ is real-isotopic to a strongly $\frs$-admissible real Heegaard diagram. 
\end{lem}
\begin{proof}
    It suffices to show that $(\cH, \tau)$ is real-isotopic to a diagram which is strongly admissible with respect to the underlying $\SpinC$-structure. In the usual setting, this is shown by winding the alpha-curves along the eta-curves and zig-zagging along the lambda arcs. In the real setting, we simply wind and zigzag the alpha and beta curves simultaneously. The proof of \cite[Proposition 4.11]{zemke_graphcob} adapts to our setting to show that the resulting Heegaard diagram is strongly admissible. We sketch the main idea for the reader's convenience. 

    Let $\Pi'_\Q$ and $\Pi'_{\Q, N}$ denote the space of rational $\frs$-renormalized periodic domains for $\cH$ and $\cH_N$, i.e., those domains of the form 
    \begin{align*}
        P - \dfrac{\langle c_1(\frs), H(P) \rangle}{2} \cdot [\Sigma],
    \end{align*}
    where $P$ is a rational periodic domain. Capping off periodic domains,  as in \Cref{lem:periodic_domains_iso}, gives rise to an identification between these two spaces $\cW_N: \Pi'_\Q \ra \Pi'_{\Q, N}$. Let $||\cdot||^{Y}_\infty$ be the $L^\infty$-norm on $H_2(Y;\Q)$. We can express 
    \begin{align*}
        PD(c_1(\frs)) = \sum_i a_i [\eta_i],
    \end{align*}
    and choose some $\epsilon > 0$ so that $\varepsilon < \frac{1}{2}$ and $\epsilon \cdot \sum_i |a_i| < \frac{1}{2}.$ If $||H(P)||^{Y}_\infty > \epsilon$,  one considers a point $x$ on $\eta_i$ as well as two nearby points $x^+$ and $x^-$ in the positive and negative winding region respectively. In this case, 
    \begin{align*}
        n_{x^+}(\cW_N(P)) = n_x(P) + N \cdot a(\eta_i, P), \quad n_{x^-}(\cW_N(P)) = n_x(P) - N \cdot a(\eta_i, P),
    \end{align*}
    and hence $\cW_N(P)$ has positive and negative multiplicities for large $N$. In the real setting, we simply assume $x$ is chosen away from the fixed set, as are $x^\pm$, and the same argument holds.
    
    If $||H(P)||^{Y}_\infty \le \epsilon$, consider an arc $\lambda_i$ connecting $w, w' \in \w$. Let $p^+$ and $p^-$ be nearby points in the winding region, as in \Cref{fig:wind_zigzag}. An argument similar to that in the first case shows that the regions containing $p^+$ and $p^-$ have multiplicities 
    \begin{align*}
        n_{p^+}(\cW_N(P)) &= -\langle c_1(\frs), H(P) \rangle + a(\lambda_i, \cW_N(P)), \\
        n_{p^-}(\cW_N(P)) &=-\langle c_1(\frs), H(P) \rangle - a(\lambda_i, \cW_N(P)).
    \end{align*}
    According to the assumptions on $||H(P)||^{Y}_\infty$, this guarantees $\cW_N(P)$ has positive and negative multiplicities. Again, this holds just as well in the real setting. 
\end{proof}

\begin{figure}
    \centering
    \includegraphics[width=0.75\linewidth]{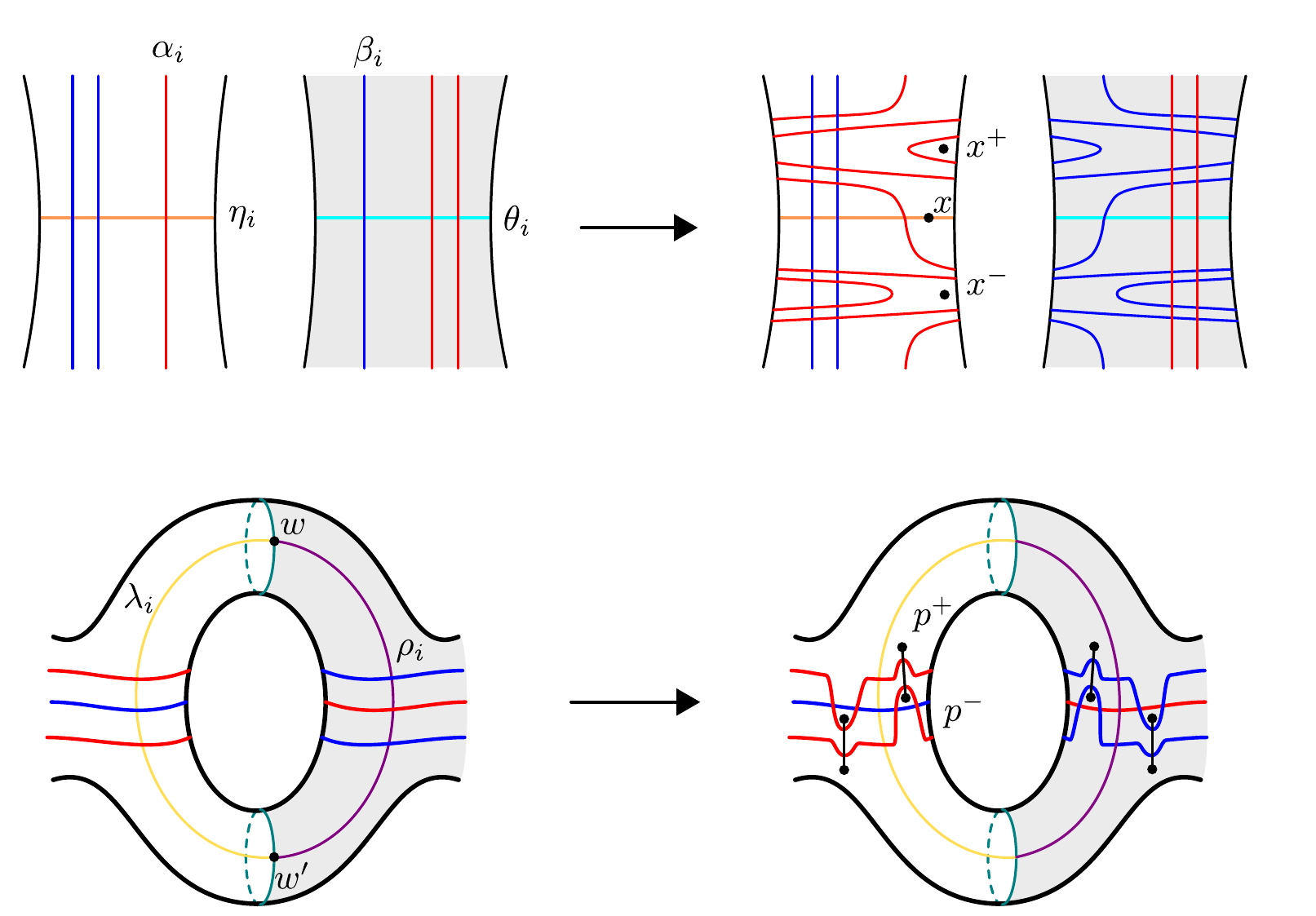}
    \caption{The winding and zig-zagging moves to achieve admissibility.}
    \label{fig:wind_zigzag}
\end{figure}

\begin{lem}\label{lem:isotopy of real admissible}
    Any two strongly admissible real Heegaard diagrams are related by a finite sequence of real pointed Heegaard moves through strongly admissible real Heegaard diagrams. 
\end{lem}
\begin{proof}
    Given a real $\SpinC$-structure $\frs^R$, we may obtain an $\frs$-strongly admissible real Heegaard diagram, $\cH_0$. If $\cH_1$ is another such diagram, there is some finite sequence of pointed real Heegaard moves by \Cref{lem:pointed real Heegaard moves} relating them; by applying the winding and zigzagging moves, we can guarantee that the intermediary diagrams are $\frs$-strongly admissible as well.  
\end{proof}

\begin{lem}\label{lem:isotopy-weak-admissible}
    Any two weakly admissible real Heegaard diagrams are related by a finite sequence of real pointed Heegaard moves through weakly admissible real Heegaard diagrams. 
\end{lem}
\begin{proof}
    Say $\cH_0$ and $\cH_1$ are weakly admissible diagrams for $(Y, \tau, \w)$.
    Any weakly admissible real Heegaard diagram can be related to a strongly admissible real Heegaard diagram using the moves in \Cref{lem:real admissibility}. Therefore, there is a sequence of weakly admissible diagrams connecting $\cH_0$ to $\cH_0'$ and $\cH_1$ to $\cH_1'$ where $\cH_0'$ and $\cH_1'$ are strongly admissible. By \Cref{lem:isotopy of real admissible}, $\cH_0'$ and $\cH_1'$ can be connected through strongly admissible diagrams. Since strongly admissible real Heegaard diagrams are weakly admissible, this provides the desired sequence. 
\end{proof}
 
\section{Definition of the real Heegaard Floer homology}\label{sec:def}

In this section, we define the different versions of real Heegaard Floer invariants. 

\subsection{The symplectic manifold}
\label{sec:Sym}
Let $(\Sigma, \balpha, \bbeta, \w, \tau)$ be a multi-pointed real Heegaard diagram for $(Y, \tau)$. Let $m = g+r-1$ where $g$ is the genus of $\Sigma$ and $r = |\w|$. Recall that the involution $\tau$ on $\Sigma$ induces an involution $R$ on the symmetric product $M=\Sym^{m}(\Sigma)$:
$$R(\{z_1, \dots, z_{m}\})=\{\tau(z_1), \dots, \tau(z_{m})\}.$$

The fixed point set, $M^R$, has a natural stratification into strata of the form 
$$\Sym^k(\Sigma') \times \Sym^{m-2k}(C) = \bigl\{ \{ z_1, \tau(z_1), \dots, z_k, \tau(z_k), c_1, \dots, c_{m-2k} \} \mid z_i \in \Sigma-C, c_i \in C\bigr\},$$
for $k=0, \dots, \lfloor m/2 \rfloor$ and $\Sigma' = (\Sigma - C)/\tau$. Nevertheless, note that $M^R$ is a smooth manifold, being locally modeled on the real part of $\Sym^{m}(\C) \cong \C^{m}$, which is $\R^{m}$.

Following \Cref{sec:real_lf}, $M$ must be equipped with a symplectic form with respect to which $R$ is anti-symplectic. 

To do this, let us first choose an area form $dA$ on $\Sigma$ (consistent with the given orientation on $\Sigma$). Since $R$ reverses orientation, the form $-R^*(dA)$ is also an area form consistent with the given orientation on $\Sigma$. Hence, we can replace $dA$ with $dA - R^*(dA)$ and obtain an $R$-anti-invariant area form. Similarly, we can replace any Riemannian metric $g_\Sigma$ on $\Sigma$ with $g_\Sigma + R^*g_\Sigma$ to obtain $R$-invariant one. From the metric and the area form we get an $R$-anti-invariant complex structure $j$ on $\Sigma$, which induces one on $M=\Sym^m(\Sigma)$.

When $m=1$, we simply take the $R$-anti-invariant area form as the symplectic form on $\Sym^1(\Sigma) = \Sigma$. 

For $m > 1$, by the work of Perutz, $M$ can then be equipped with a K\"ahler form $\omega$, which agrees with the product symplectic form away from the diagonal. (See  \cite[Proposition 1.1]{perutz}, which is stated there for $g=m$; however, the underlying analysis in \cite[Section 7]{perutz} is done for general $m$.)  
There is an identification of $H^2(M)$ with  $H^0(\Sigma)\oplus \Lambda^2 H^1(\Sigma)$.  The cohomology class of Perutz's form is 
\[ [\omega] = \eta + \lambda \theta,\] 
where $\eta$ is the $\R$-cohomology class corresponding to $1 \in H^0(M)$, $\theta$ is a class in $\Lambda^2 H^1(\Sigma)$ (invariant under the action of the mapping class group), and $\lambda > 0$ is some small real number. Then $-R^* \omega$ is another K\"ahler form, satisfying $[-R^*\omega]=[\omega]$. Replacing $\omega$ with $(\omega - R^*\omega_\lambda)/2$, we can assume that $\omega$ is $R$-anti-invariant (and in the same cohomology class as before).

\begin{lemma}\label{lem:sym_monotone}
    The symplectic manifold $(\Sym^m(\Sigma), \omega)$ is spherically monotone.
\end{lemma}
\begin{proof}
     According to \cite{Macdonald}, we have $c_1(TM) = (m-g+1)\eta - \theta$. The class $\theta$ evaluates trivially on the generator $S\in \pi_2(M)$, whereas $\langle \eta, S \rangle = 1$. Hence, 
    \begin{align*}
        \omega|_{\pi_2(M)}= A \cdot c_1(TM)|_{\pi_2(M)}.
    \end{align*}
    for $A= 1/ (m-g+1)>0.$
    \end{proof}

\subsection{Real Invariant Domains and their Index}

Our primary computational tool will be a correspondence between strips between $\Ta$ and $M^R$ and real invariant strips between $\Ta$ and $\Tb$. Before proceeding to address the analytical aspects of the theory, we give some examples of real invariant domains to give the reader a sense of this correspondence.

Let us first discuss the difference between the real and classical Maslov index. According to \Cref{prop:indR}, for a class $\phi \in \pi_2^R(\x, \y)$, 
\begin{align*}
    \ind_R(\phi) = \frac{1}{2}\left(\ind(\tilde{\phi}) - \frac{\sigma(\Ta,\x) - \sigma(\Ta,\y)}{2} \right).
\end{align*}
The quantities $\sigma(\Ta,\x)$ and $\sigma(\Ta,\y)$ can be computed directly from the diagram. If $\x$ is an invariant intersection point, and $x \in \x$ is an intersection point between $\alpha_i$ and the fixed point set $C$, we define a quantity $\sigma(\alpha, x) \in \{\pm 1\}$ according to the cyclic order of the curves $C$, $\alpha_i$, and $\beta_i = \tau(\alpha_i)$ at $x$. See \Cref{fig:triple_index}. 

\begin{figure}[h]
\def\svgwidth{.8\linewidth}
\begingroup%
  \makeatletter%
  \providecommand\color[2][]{%
    \errmessage{(Inkscape) Color is used for the text in Inkscape, but the package 'color.sty' is not loaded}%
    \renewcommand\color[2][]{}%
  }%
  \providecommand\transparent[1]{%
    \errmessage{(Inkscape) Transparency is used (non-zero) for the text in Inkscape, but the package 'transparent.sty' is not loaded}%
    \renewcommand\transparent[1]{}%
  }%
  \providecommand\rotatebox[2]{#2}%
  \newcommand*\fsize{\dimexpr\f@size pt\relax}%
  \newcommand*\lineheight[1]{\fontsize{\fsize}{#1\fsize}\selectfont}%
  \ifx\svgwidth\undefined%
    \setlength{\unitlength}{623.62204724bp}%
    \ifx\svgscale\undefined%
      \relax%
    \else%
      \setlength{\unitlength}{\unitlength * \real{\svgscale}}%
    \fi%
  \else%
    \setlength{\unitlength}{\svgwidth}%
  \fi%
  \global\let\svgwidth\undefined%
  \global\let\svgscale\undefined%
  \makeatother%
  \begin{picture}(1,0.4)%
    \lineheight{1}%
    \setlength\tabcolsep{0pt}%
    \put(0,0){\includegraphics[width=\unitlength,page=1]{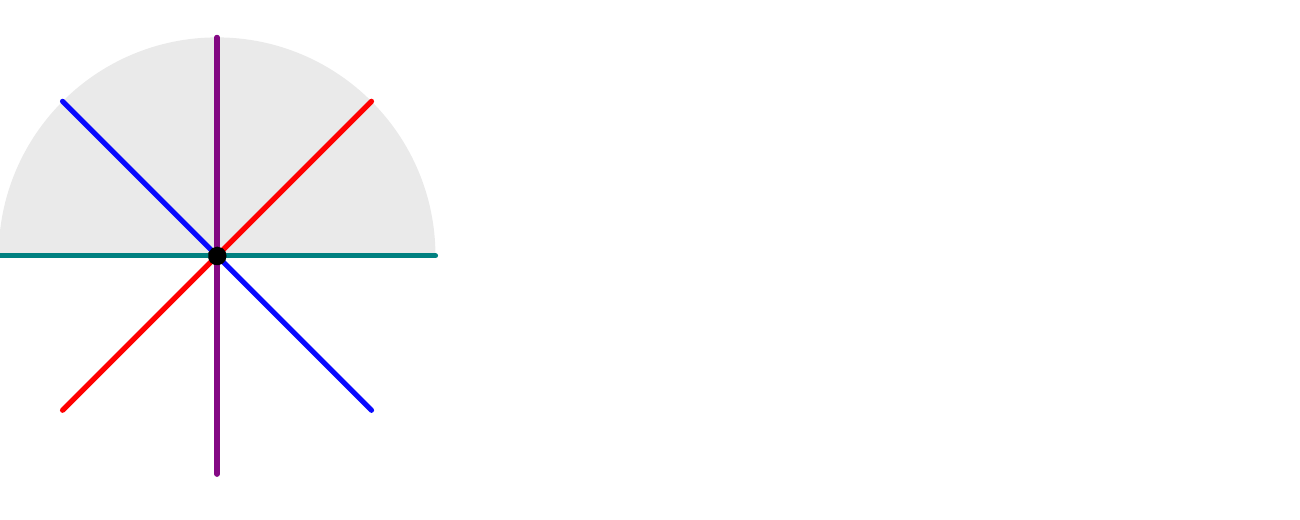}}%
    \put(0.28998479,0.32386783){\color[rgb]{0,0,0}\makebox(0,0)[lt]{\smash{\begin{tabular}[t]{l}{\small$T_x \alpha_i$}\end{tabular}}}}%
    \put(0.33047201,0.17194479){\color[rgb]{0,0,0}\makebox(0,0)[lt]{\smash{\begin{tabular}[t]{l}{\small$T_x C$}\end{tabular}}}}%
    \put(0.17030277,0.3811465){\color[rgb]{0,0,0}\makebox(0,0)[lt]{\smash{\begin{tabular}[t]{l}{\small$J\cdot T_x C$}\end{tabular}}}}%
    \put(0,0){\includegraphics[width=\unitlength,page=2]{triple_index.pdf}}%
    \put(0.58644402,0.32365264){\color[rgb]{0,0,0}\makebox(0,0)[lt]{\smash{\begin{tabular}[t]{l}{\small$T_x \alpha_i$}\end{tabular}}}}%
    \put(0.92628458,0.16889503){\color[rgb]{0,0,0}\makebox(0,0)[lt]{\smash{\begin{tabular}[t]{l}{\small$T_x C$}\end{tabular}}}}%
    \put(0.76972844,0.38105208){\color[rgb]{0,0,0}\makebox(0,0)[lt]{\smash{\begin{tabular}[t]{l}{\small$J\cdot T_x C$}\end{tabular}}}}%
  \end{picture}%
\endgroup%

\caption{Left: an intersection point $x$ with $\sigma(\alpha_i, x) = +1$; Right: an intersection point $x$ with $\sigma(\alpha_i, x) = -1$.}
    \label{fig:triple_index}
\end{figure}

\begin{lemma}\label{lem:local sigma computations}
    Let $\x \in (\Ta \cap \Tb)^R$ be an invariant intersection point. Then,
    \begin{align*}
        \sigma(\Ta, \x) = \sum_{x \in \x\cap C} \sigma(\alpha, x).
    \end{align*}
\end{lemma}
\begin{proof}
    Fix a complex structure $\frak{j}$ on $\Sigma$ and let $J = \Sym^m(\frak{j})$. Near $\x$, there are local coordinates so that $T_\x\Ta$ is the graph of a symmetric linear function $F: T_\x M^R \ra J \cdot T_\x M^R$; then, $\sigma(\Ta,\x)$ can be computed as the signature of this linear map. Write $\x = (z_1, \tau(z_1), \hdots, z_k, \tau(z_k), c_1, \hdots, c_{m-2k})$. Near $\x$, we can model $\Sym^m(\Sigma)$ on $$\prod_{i=1}^k \Sym^2(\C) \times\prod_{i=1}^{m-2k} \Sym^1(\C),$$ and therefore, we can write $F$ as a block diagonal matrix $A_1 \oplus\hdots \oplus A_k\oplus B_1 \oplus \hdots \oplus B_{m-2k}$, where $A_i: \R^2 \ra J\cdot \R^2$ and $B_i: \R \ra J \cdot \R$, and $\sigma(\Ta, x)$ will be the sum of the signatures of these maps. For points $x \in \x \cap C$, it is clear that $T_\x \alpha_i$ is either the graph of $\id$ or $-\id$, so $\sigma(B_i) \in \{\pm 1\}$. Indeed, $\sigma(B_i)$ is precisely $\sigma(\alpha, x)$. Compare with \Cref{fig:triple_index}. 

    We must consider the remaining points of $\x$ in pairs. Take some $\{z, \tau(z)\} \in \x$ with $z \in \alpha_i$ and $\tau(z) \in \alpha_j$. Identity $T_{\{z, \tau(z)\}} \Sym^2(\Sigma)$ with $(\R \oplus J \cdot \R)^{\oplus 2}$ by fixing a basis $\{e_1, J(e_1), e_2, J(e_2)\}$ where $e_1$ and $e_2$ span $T_z\alpha_i$ and $T_{\tau(z)}\tau(\alpha_i)$ respectively. We make the following observations: 
    \begin{enumerate}
        \item $R(e_1) = - e_2$ and $R(J(e_1)) = J(e_2)$;
        \item $T_{z, \tau(z)}M^R$ is spanned by $u_1 = e_1-e_2$ and $u_2 = J(e_1 + e_2)$ (and therefore, $J \cdot T_{z, \tau(z)}M^R$ is spanned by $J(u_1)$ and $J(u_2)$);
        \item $T_{z, \tau(z)} (\alpha_i \times \alpha_j)$ is spanned by $e_1$ and $J(e_2)$, and so, in the $\{u_1, u_2, J(u_1), J(u_2)\}$ basis, $T_{z, \tau(z)} (\alpha_i \times \alpha_j)$ is spanned by $(u_1 - J(u_2))$ and $(u_2 - J(u_1))$.
    \end{enumerate}
    From these observations, it is clear that $T_{z, \tau(z)} \alpha_i \times \alpha_j$ is the graph of the linear map $T_{z, \tau(z)}M^R \ra J \cdot T_{z, \tau(z)}M^R$ given by:
    \begin{align*}
        u_1 \mapsto -J(u_2), \quad u_2 \mapsto - J(u_1).
    \end{align*}
    The signature of this linear transformation is zero. Hence, $\sigma(\Ta, \x) = \sum_{x \in \x\cap C} \sigma(\alpha, x)$ as claimed.
\end{proof}
The usual Maslov index can be computed combinatorially, due to work of Lipshitz \cite[Corollary 4.3]{Lipshitz_cylindrical}. Therefore, the previous lemma implies the real index can be computed combinatorially as well:

\begin{corollary}
For a class $\phi \in \pi_2^R(\x, \y)$ represented by a domain $\cD$, we have
\begin{align}\label{eqn:combinatorial real index}
    \ind_R(\phi) = \frac{1}{2}\left(n_\x(\cD) + n_\y(\cD) + e(\cD) - \frac{\sum_{x \in \x\cap C} \sigma(\alpha, x) - \sum_{y \in \y\cap C} \sigma(\alpha, y)}{2} \right).
\end{align}
where $e(\cD)$ is the Euler measure of $\cD$, and $n_\x(\cD)$, $n_\y(\cD)$ are the average vertex multiplicities. (See \cite[Section 4]{Lipshitz_cylindrical} for the definitions.)
\end{corollary}

\subsection{Real Invariant Domains}

\begin{figure}[h]
\def\svgwidth{.8\linewidth}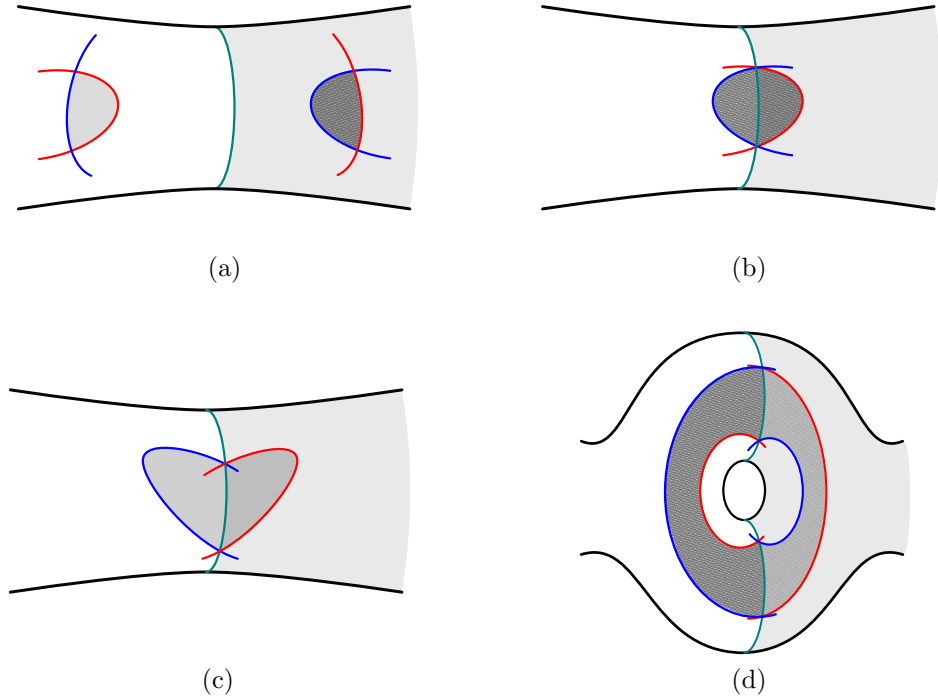
\caption{Examples of real invariant domains.}
    \label{fig:domains}
\end{figure}

In this section, we consider domains which exemplify some of the subtleties in the real setting. To carry out these computations, we arrange to work with a constant path of almost complex structures induced by a complex structure on the Heegaard surfaces. To do so, we follow the strategy in \cite[Chapter 11]{OSSunfinished}.

\begin{defn}\label{def:indecomposable}
    A real domain $\cD(\phi)$ is called \emph{indecomposable} if for any decomposition $\cD(\phi) = \cD_1 + \cD_2$ into positive domains, either $\mu_R(\cD_1)$ or $\mu_R(\cD_2)$ is negative.
\end{defn}

\begin{proposition}\label{prop:indecomposable}
    Let $\cD$ be a real indecomposable domain for a class $\phi \in \pi_2^R(\x, \y)$ with real Maslov index 1. Then, for a generic $J \in \cJ^R$, the count $\#\cM_R^J(\phi)$ is independent of $J$.
\end{proposition}

We relegate the proof of \Cref{prop:indecomposable} to the end of \Cref{subsec:bubbles} (after our discussion of bubbles).

The following examples are all indecomposable, so we may work with constant paths of complex structures induced by the Heegaard surface. Compare the following examples to \Cref{fig:domains}. 

\begin{example}
    Consider a (classical) domain $\cD_0$ from $\x\times R(\x)$ to $\y\times R(\x)$ which does not intersect the fixed set $C$. We can define a real invariant domain $\cD=\cD_0-\tau(\cD_0)$ as a domain from $\x \times R(\x)$ to $\y \times R(\y)$. For instance, consider the case that $\cD_0$ is an index 1 bigon from $\x\times R(\x)$ to $\y\times R(\x)$ as in frame (a) of \Cref{fig:domains}; the Riemann mapping theorem implies there is a unique holomorphic representative up to translation in the domain. The doubled domain $\cD$ from $\x \times R(\x)$ to $\y \times R(\y)$ also admits a holomorphic representative, though the classical index is now two. The real index, however, is one, as the reparamentrizations of the two disks must agree in order to preserve the symmetry. In such cases, the real index is exactly half that of the ordinary index. This agrees with \Cref{eqn:combinatorial real index}. The average vertex multiplicities $n_\x(\cD)$ and $n_\y(\cD)$ are both $\frac{1}{2}$ and the Euler measure is 1. Altogether, they sum to the usual index, equal to $2$. The triple Maslov index is trivial, since neither $\x$ nor $\y$ intersect $C$.
\end{example}

\begin{example}
\label{ex:bigon}
    Now, let $\cD$ be a bigon between invariant intersection points $\x$ and $\y$ which intersects the fixed set $C$ as in frame (b) of \Cref{fig:domains}, so that the two interior angles are less than $180^\circ$. Then there is a unique holomorphic representative in the classical case as well as in the real case, since translation in the domain respects the involution. So, the two Maslov indices are both $1$. (This example was already discussed in the proof of Proposition~\ref{prop:indR}. Compare Figure~\ref{fig:bigon}.) Here the Maslov triple index is nontrivial; $\sigma(\alpha, x) = -1$ while $\sigma(\alpha, y) = 1$. Hence, $\ind_R(\phi) = \frac{1}{2}(1 - (-1)) = 1$. 
\end{example}

\begin{example}    
    Next, consider the case of a bigon as in the third frame of \Cref{fig:domains}, so that one of the interior angles is greater than $180^\circ$. Then, in the classical case, the expected dimension of the moduli space is 2, as there is 1-dimensional parameter of cuts which can be made to the disks. In the real case however, no cuts can be made while preserving the symmetry of the disk. Therefore, the real index is 1. The strip between $\bT_\alpha$ and $M^R$ can be seen quite explicitly -- it is the bigon with one half of its boundary on the alpha curve shown, and the other on $C$. No matter the interior angles of $\cD$, the bigon with one edge on $M^R$ always has interior angles less than $180^\circ$, so will have real index 1. Comparing with \Cref{eqn:combinatorial real index}, we note that the terms $\sigma(\alpha, x)$ and $\sigma(\alpha, y)$ are equal, and therefore cancel. 
\end{example}

\begin{example}\label{ex:annulus}
    Consider the annulus in the last frame of \Cref{fig:domains}. The domain $\cD$ will double cover the disk if and only if ratio of the angles spanned by the outer and inner alpha boundaries is 1. (Compare \cite[Lemma 9.3]{os_holodisks}.) In the present case, where all interior angles are less than $180^\circ$, no cuts can be made; therefore, for a generic choice of almost complex structure, this domain has no holomorphic representatives. But, in the symmetric real setting, this ratio must be preserved -- i.e., for a generic choice of \emph{symmetric} almost complex structure, there is a holomorphic representative! In this example, the classical index $\ind=0$ is actually smaller than the real index $\ind_R=1$. In particular, this example demonstrates that there need not exist any $J \in \cJ_R(\phi)$ for which transversality can be achieved for the space of ordinary pseudo-holomorphic curves and the space of invariant ones simultaneously. This is reflected in \Cref{eqn:combinatorial real index}, as the correction term coming from the triple Maslov indices is 2, so indeed, the real index is 1. Compare to Remark \ref{rem:et}.
\end{example}

Here are two final examples which demonstrate what sorts of ends we expect from our 1-dimensional moduli spaces.

\begin{example}
    Consider the genus 0 real Heegaard diagram with two basepoints on the fixed set, which is shown in \Cref{fig:disk_bubbles}. There are two invariant strips from $x$ to $y$ of real index 1 (one which covers $w_1$ and another which covers $w_2$). There are also two invariant strips from $y$ to $x$; the usual index is 3 for each of these strips (since they each have two vertices with $270^\circ$ angles). However, as we have seen, the real index is insensitive to the angles of vertices on the fixed set; each of these strips have real index 1. Indeed, in this case, we have $\partial^2 x = (U_1^2 + U_2^2)x$. 

    As is typical, terms of $\partial^2 x$ correspond to disk bubbles on $\Ta$. There are two obvious domains representing such classes, namely the two disks $A_1$ and $A_2$ with boundary $\alpha$. Such classes have real index 2 (and ordinary index 2 as well); the corresponding real invariant domains are $A_1 - \tau(A_1)$ and $A_2 - \tau(A_2)$, which each have ordinary index 4, and hence real index 2 by our formula.\footnote{\Cref{prop:indecomposable} does not apply in this example, as these domains have index 2. However, these domains are \emph{alpha-injective}, in the language of \cite[Definition 3.8]{os_holodisks}. According to \cite[Lemma 3.9]{os_holodisks}, we may work with a constant path of almost complex structures and achieve transversality by perturbing the alpha curves instead.}

    We can also see bubbles on $M^R$ in this example. The fixed point set is a great circle $C$, and indeed the two hemisphere represent classes with boundary entirely in $M^R$. These disks have  index 2. The corresponding real invariant domain is the entire sphere, which has ordinary index 4, which is twice the real index 2.

    As we shall see, this behavior is typical: when the number of alpha curves exceeds the genus of the Heegaard surface, boundary degenerations on $\Ta$ do indeed appear and need not cancel in pairs. This is similar to the setting of Heegaard Floer homology with multiple basepoints \cite{os_linkinvts}. On the other hand, boundary degenerations on $M^R$ \emph{necessarily} come in pairs, and their contributions cancel. 
\end{example}

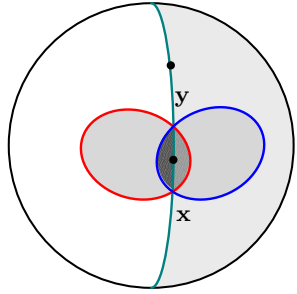
\begin{figure}[h]
\def\svgwidth{.25\linewidth}
\begingroup%
  \makeatletter%
  \providecommand\color[2][]{%
    \errmessage{(Inkscape) Color is used for the text in Inkscape, but the package 'color.sty' is not loaded}%
    \renewcommand\color[2][]{}%
  }%
  \providecommand\transparent[1]{%
    \errmessage{(Inkscape) Transparency is used (non-zero) for the text in Inkscape, but the package 'transparent.sty' is not loaded}%
    \renewcommand\transparent[1]{}%
  }%
  \providecommand\rotatebox[2]{#2}%
  \newcommand*\fsize{\dimexpr\f@size pt\relax}%
  \newcommand*\lineheight[1]{\fontsize{\fsize}{#1\fsize}\selectfont}%
  \ifx\svgwidth\undefined%
    \setlength{\unitlength}{204.09448819bp}%
    \ifx\svgscale\undefined%
      \relax%
    \else%
      \setlength{\unitlength}{\unitlength * \real{\svgscale}}%
    \fi%
  \else%
    \setlength{\unitlength}{\svgwidth}%
  \fi%
  \global\let\svgwidth\undefined%
  \global\let\svgscale\undefined%
  \makeatother%
  \begin{picture}(1,1)%
    \lineheight{1}%
    \setlength\tabcolsep{0pt}%
    \put(0,0){\includegraphics[width=\unitlength,page=1]{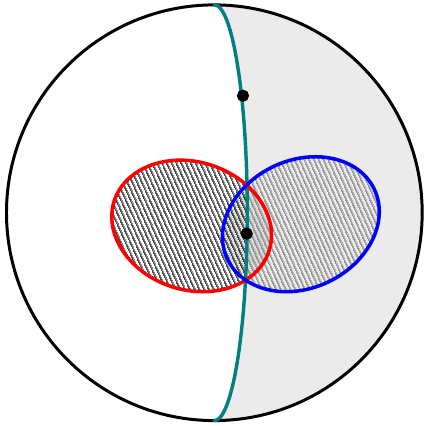}}%
    \put(0.58548085,0.65067114){\color[rgb]{0,0,0}\makebox(0,0)[lt]{\smash{\begin{tabular}[t]{l}{\small$\y$}\end{tabular}}}}%
    \put(0.20642206,0.45757861){\color[rgb]{0,0,0}\makebox(0,0)[t]{\smash{\begin{tabular}[t]{c}{\small$\bm\alpha$}\end{tabular}}}}%
    \put(0.58957624,0.24116389){\color[rgb]{0,0,0}\makebox(0,0)[lt]{\smash{\begin{tabular}[t]{l}{\small$\x$}\end{tabular}}}}%
    \put(0.9340262,0.45757861){\color[rgb]{0,0,0}\makebox(0,0)[t]{\smash{\begin{tabular}[t]{c}{\small$\bm\beta$}\end{tabular}}}}%
  \end{picture}%
\endgroup%

\caption{A genus 0, real Heegaard diagram with domains representing disk bubbles. }
    \label{fig:disk_bubbles}
\end{figure}

\subsection{Bubbles}\label{subsec:bubbles}

For $\x \in \bT_\alpha \cap M^R$ and fixed $J =(J_t)\in \cJ_R$, define the moduli space of \emph{alpha-degenerate disks} by
\begin{align*}
    \cN_J(\x, \bT_\alpha) = \left\{ u: \R \times [0,\infty) \ra M \bigg | 
    \substack{u(s,0) \in \bT_\alpha,\\
    \lim_{s+it\ra \infty} u(s,t) = \x, \\
    \\
    \frac{\partial u}{\partial s} + J_0(u(s,t))\frac{\partial u}{\partial t} = 0
    } \right\}.
\end{align*}
Note that, unlike what we did for strips in \eqref{eq:floer}, here we just use the single almost complex structure $J_0$, for which the symmetry relation \eqref{eq:JR} imposes no particular constraint.

The moduli space $\cN_J(\x, M^R)$ of \emph{$M^R$-degenerate disks} is defined similarly, but using
the almost complex structure $J_{1/2}$, which must be $R$-anti-invariant according to \eqref{eq:JR}.

Let $\Pi_2(\x, \bT_\alpha)$ and $\Pi_2(\x, M^R)$ be the spaces of homotopy classes of (not necessarily holomorphic) disks satisfying the boundary conditions above. We will write $\cN_J(\phi)$ for the moduli space of holomorphic disks in the class $\phi$ for $\phi$ in either $\Pi_2(\x, \bT_\alpha)$ or $\Pi_2(\x, M^R)$. There is a two-dimensional automorphism group acting on these spaces, by dilation in the real direction (i.e., $z \mapsto zr$ for $z \in \C, r \in \R$) and translation in the imaginary direction. Let $\widehat{\cN}_J(\phi)$ be the quotient space. We write $\ind_R(\phi)$ for the expected dimension of this moduli space, and call this the {\em real index}, to be in line with our previous convention  for strips.

\begin{lemma}\label{lem:disk_monotone}
     The Lagrangians $\bT_\alpha$ and $M^R$ are monotone on disks, i.e., Equation~\eqref{eq:monotonedisks} holds. In fact, $\bT_\alpha$ and $M^R$ have minimal Maslov numbers $2$ and $(m-g+1)$ respectively. 
\end{lemma}
\begin{proof}
    We simply compute the minimal Maslov indices. Consider a class $\phi \in \Pi_2(\x, \bT_\alpha)$. Note that $\Pi_2(\x, \bT_\alpha)$  is the same as the set of homotopy classes of disks denoted $\pi_2(\x)$ in \cite[Section 3.7]{os_holodisks} and $\pi_2^{\alpha}(\x)$ in \cite[Section 5]{os_linkinvts}. The  domain on $\Sigma$ corresponding to $\phi$ must be of the form $\sum_i a_iA_i$ where $\Sigma \smallsetminus \balpha = \amalg_i A_i$ and $a_i$ are non-negative integers.  It follows from \cite[Lemma 5.4]{os_linkinvts} that the Maslov index of each $A_i$ is $2$, and therefore this is the minimum possible index for $\phi$, as long as $\phi \neq 0$.

    In the second case, we consider a class $\phi \in \Pi_2(\x, M^R)$. Such a disk has boundary contained entirely in $M^R$, therefore it can be doubled to obtain a sphere, which must be represented by a domain which is a multiple of $\Sigma$. Note that $\Sigma$ has classical index $2(m-g+1)$, and therefore its real index is $m-g+1$, by an application of the real index formula \eqref{eq:indR}. 
\end{proof}

In defining the real Heegaard Floer complexes in the next subsection, when studying $\del^2$ we will be interested in one-dimensional moduli spaces of real invariant strips (of index $2$) between $\Ta$ or $\Tb$; or, equivalently (by the correspondence from Section~\ref{sec:setup}), in one-dimensional moduli spaces of strips between $\Ta$ and $M^R$. In principle, the ends of such moduli spaces can contain trees formed out of broken strips, disk bubbles, and sphere bubbles. 

\begin{remark}
The correspondence between disks between $\Ta$ and $M^R$ and symmetric disks between $\Ta$ and $\Tb$ makes discussing sphere bubbles slightly confusing. A disk bubble on $M^R$ doubles to become a symmetric sphere bubble in a disk between $\Ta$ and $\Tb$. When a sphere bubbles off of a disk between $\Ta$ and $M^R$, it doubles to a pair of spheres for $\Ta$ and $\Tb$. 

Spheres of the latter kind cannot be ruled out by the methods appearing in \cite{os_holodisks}. In \cite[Section 3.7]{os_holodisks}, Ozsv\'ath and Szab\'o eliminated sphere bubbles by appealing to a specific result about the symmetric product: their Lemma 3.13, which says that the set of complex structures $j$ on $\Sigma$ such that there exist $\Sym^g(j)$-holomorphic curves in $\Sym^g(\Sigma)$ through a given point is of codimension two. The analogue of this in the real setting is false: For $R$-anti-invariant complex structures on $\Sigma$, the corresponding set has codimension one. However, in our case, we do not use this kind of argument. In \cite{os_holodisks}, the theory was built for totally real tori using very special almost complex structures (equal to some $\Sym^g(j)$ on a large open set), whereas we use Perutz's set-up from \cite{perutz}, where the tori are Lagrangian and we have more flexibility in choosing the almost complex structure. In this context, we can use the standard arguments from the theory of gluing pseudo-holomorphic curves, where sphere bubbles appear in codimension two.  
\end{remark}

\begin{lemma}\label{lem:spheres}
    Let $\x \in M$ and let $S$ be the generator of $\pi_2(M)$, and let $e_\x$ be the constant disk to $\x$. Interpret $e_\x + [S]$ as a homotopy class of spheres with a fixed point mapped to $\x$, and let $\widehat{\cN}_J(e_\x + [S])$ be the moduli space of holomorphic representatives of  this class. If $m=g$, then for a generic $J \in \cJ_R$, the space $\widehat{\cN}_J(e_\x + [S])$ is empty.
\end{lemma}
\begin{proof}
    This follows from \cite[Lemma 3.1.6]{McDuffSalamonBig} and is done explicitly in the Heegaard Floer context in \cite[Corollary 9.3.2]{OSSunfinished}. We sketch the argument so as to point out that the symmetry plays essentially no role. According to \cite[Lemma 3.1.6]{McDuffSalamonBig}, for a generic $J$, the moduli space of simple curves $\widehat{\cN}_J(e_\x + [S])$ is a smooth manifold of  dimension $$\dim M+ \langle c_1(TM), [S] \rangle =  2g+ 2$$ and the evaluation map $\mathrm{ev}_{p}$ for $p \in S^2$ has $\x$ as a regular value. Arguing as in \Cref{prop:transversality}, it follows that the same is true for a generic $J \in \cJ_R$. Since any holomorphic sphere representing $[S]$ must be simple, \cite[Lemma 3.1.6]{McDuffSalamonBig} applies. Further, the automorphisms of $S^2$ fixing a point is four dimensional and the quotient of $\widehat{\cN}_J(e_\x + [S])$ by this group is $2g-2$ dimensional. The evaluation map descends to the quotient, and the subspace of points in $M$ which contain a holomorphic sphere is precisely the image of this space. Hence, the preimage of $\x$ under $\mathrm{ev}_p$ is generically empty.
\end{proof}

We will now focus on understanding disk bubbles. Those with boundary on $\Ta$ have minimal Maslov index $2$ by \Cref{lem:disk_monotone}. Therefore, the only ones that can contribute to $\del^2$ are those of index $2$ that are attached to a constant strip at an intersection point $\x$. Compare \cite[proof of Theorem 4.3]{os_holodisks}.

\begin{lemma}\label{lem:Talpha_degenerations}
    Let $\phi$ be a class in $\Pi_2(\x, \bT_\alpha)$ with real index 2. Then, 
    \begin{align*}
        \# \widehat{\cN}_J(\phi) \equiv \begin{cases}
            0 \mod 2 & \text{ if } m = g \\
            1 \mod 2 & \text{ if } m > g. \\
        \end{cases} 
    \end{align*}
\end{lemma}
\begin{proof}
    In the case that $m = g$, \cite[Proposition 3.15]{os_holodisks} shows that for a sufficiently stretched out complex structure on $\Sigma$, the moduli space $\cN_J(\x, \bT_\alpha)$ is empty. Furthermore, the count of points in  $\cN_J(\x, \bT_\alpha)$ is independent of the choice of generic almost complex structure $J_0$. (Indeed, as we vary it in a one-dimensional family, we cannot encounter sphere bubbles, because those appear in codimension two.) It follows that the count is zero for a generic $J_0$,  which comes from a generic $J=(J_t) \in \cJ_R$. 
    
    When $m > g$, it follows from \cite[Theorem 5.5]{os_linkinvts} that positive classes of Maslov index 2 admit an odd number of holomorphic representatives. Hence, for a generic $J_0$ from $J \in \cJ_R$, the count is one.
\end{proof}

For disk bubbles with boundary on $M^R$, we have:
 
\begin{lemma}\label{lem:MR_degenerations}
If $\w = \{w_1, \dots, w_r\}$ is the collection of basepoints on the fixed set $C \subset \Sigma$, fix a sequence of non-negative integers $\n=(n_1, \dots, n_r)$.  Then, $$\sum_{\substack{\phi \in \Pi_2(\x, M^R) \\ \ind_R(\phi) = 2, \ n_\w(\phi) = \n}}\# \widehat{\cN}_J(\phi) \equiv 0 \mod{2}.$$
\end{lemma}
\begin{proof}
    Let $\phi$ be an element of $\Pi_2(\x, M^R)$ with $\ind_R(\phi)=2$ and $n_{\w}(\phi)=n$. Necessarily, the reflection $R(\phi)$ is in $ \Pi_2(\x, M^R)$ as well, with $\ind_R(R(\phi)) = 2$ and $n_{\w}(R(\phi))=n$ as well. Moreover, given a holomorphic representative $u$ for $\phi$, the reflection $R\circ u$ is an anti-holomorphic representative for $R(\phi)$, therefore by precomposing by the map $\rho$ on $\R \times [0, \infty)$, given by $\rho(x, y) = (-x, y)$ we obtain a holomorphic representative. The theorem follows, provided $R \circ u \circ \rho$ and $u$ are distinct for generic $J$.
    
    Indeed, if $u = R \circ u \circ \rho$ is a $J$-holomorphic curve, then $([0,\infty) \times \R)^\rho = \{0\} \times [0, \infty)$ must map into $M^R$. Therefore, by restricting, we obtain $J$-holomorphic maps $u_+: [0, \infty) \times \R_{\ge 0}\to M$ and $u_-: [0, \infty) \times \R_{\le 0} \to M$ belonging to $\cN_J(\x, M^R)$. Therefore, $\cR([u_\pm]) \in \pi_2^R(\x)$ are symmetric sphere bubbles of index 1. However, by transversality, $\dim \cN_J([u_\pm]) = \ind_R([u_\pm]) - \dim \mathrm{Aut}([0, \infty) \times \R)$ for generic $J \in \cJ_R$. Hence, degenerations such as $u$ generically appear with $\ind_R([u]) > 2$, which is not the case here.
\end{proof}

\begin{remark}
    \label{rem:index1}
    Observe that, in the case $m=g$, the minimal Maslov number of $M^R$ is one. When looking at the ends of the one-dimensional moduli spaces of strips,  generically we may also have  disk bubbles of real index $1$, which do not have to pass through any intersection point $\x$. Nevertheless, the counts of such disks (added over all $\phi$ with fixed $n_\w(\phi)$) are zero, by the same cancellation argument as in the proof of \Cref{lem:MR_degenerations}. 
\end{remark}

Having defined all the relevant objects, we now turn to the proof of \Cref{prop:indecomposable}.
\begin{proof}[Proof of \Cref{prop:indecomposable}] This is very similar to \cite[Lemma 11.1.3]{OSSunfinished}, so we merely sketch the proof. Given paths of almost complex structures $J_0$ and $J_1$ in $\cJ^R$, choose a path $J_s$ connecting them and consider the moduli space 
    \begin{align*}
        \bigcup_{s \in [0,1]}\cM_R^{J_s}(\phi).
    \end{align*}
    The compactification of this 1-dimensional space has ends corresponding to broken strips as well as disk and sphere bubbles. 

    Suppose there is an end corresponding to a broken flow line, giving a decomposition $\cD = \cD_1 + \cD_2$ into nontrivial domains so that $\cM_R^{J_s}(\cD_1)$ and $\cM_R^{J_s}(\cD_2)$ are nonempty for some $s$. By assumption $\cD_1$ and $\cD_2$ are positive domains and have non-negative real Maslov index. This, however, contradicts indecomposability.

    Disk bubbles on $\Ta$ and sphere bubbles cannot appear because they have index at least two. Disk bubbles on $M^R$ with $\ind_R=1$ contribute zero; see Remark~\ref{rem:index1}.
\end{proof}

\subsection{Real Heegaard Floer Invariants}

We are now equipped to define the real Heegaard Floer complexes.

Given a collection of basepoints $\w = \{w_1, \hdots, w_r\}$ on $\Sigma$, we consider the  polynomial ring
\begin{align*}
    \F[U_\w] := \F[U_{w_1}, \hdots, U_{w_r}].
\end{align*}
Given a class $\phi \in \pi_2^R(\x,\y)$, we write 
\begin{align*}
    U_\w^{n_\w(\phi)} = U_{w_1}^{n_{w_1}(\phi)}\cdot \hdots \cdot U_{w_r}^{n_{w_r}(\phi)}.
\end{align*}
When $|\w| =1$, we just write $U$ for $U_{w_1}$, and $\F[U]$ for $\F[U_\w]$.

\begin{defn}\label{def: curved complexes}
    A \emph{curved chain complex} (or \emph{matrix factorization}) over a ring $\cR$ is an $\cR$-module $C$ equipped with an endomorphism $\partial: C \ra C$ satisfying
    \begin{align*}
        \partial^2 = \omega \cdot \id_C,
    \end{align*}
    for some $\omega \in \cR$, called the \emph{curvature constant}.

    A \emph{morphism of curved complexes} is an map $F: (C_1, \partial_1) \ra (C_2, \partial_2)$ satisfying 
    \begin{align*}
        \partial \circ F = F \circ\partial,
    \end{align*}
    and a chain homotopy between morphisms $F, G$ of curved complexes is a map $H: C_1\ra C_2$ so that 
    \begin{align*}
        F - G = \partial\circ H + H \circ \partial.
    \end{align*}
\end{defn}
In particular, just as in the case of regular chain complexes, we can consider curved chain complexes up to chain homotopy equivalence. 

\begin{defn}
    Let $(Y, \tau,\w)$ be a real pointed 3-manifold and let $\frs$ be a real $\SpinC$ structure on $(Y, \tau)$. Fix a real Heegaard diagram $\cH$ which is strongly $\frs$-admissible. Define $\CFR^-(\cH, \frs)$ to be the $\F[U_\w]$-module freely generated by elements $\x$ of $\bT_\alpha \cap M^R$ satisfying $\frs^\x_\w(\x) = \frs$. For the moment, we equip it with the preliminary  differential
    \begin{align}\label{eqn:CFR_diff}
        \delp \x = \mathlarger{\sum}_{\substack{\phi \in \pi_2^R(\x, \y)\\ \ind_R(\phi) = 1}} \# \widehat{\cM}_R(\phi) \cdot U^{n_{\w}(\phi)}_\w \y,
    \end{align}
    extended linearly over $\F[U_\w]$. 
\end{defn}

\begin{proposition}
\label{prop:onew}
    When $|\w|= 1$, the pair $(\CFR^-(\cH, \frs), \delp)$ is a chain complex. 
\end{proposition}
\begin{proof}
    As the Heegaard diagram is strongly $\frs$-admissible, the sum appearing in \Cref{eqn:CFR_diff} is finite. Gromov compactness applies because, even though the minimal Maslov number of $M^R$ is one, disk bubbles of real index one cancel out in pairs; compare Remark~\ref{rem:index1}. 

    To verify $\delp^2 = 0$, fix intersection points $\x$ and $\y$ in $\bT_\alpha\cap M^R$ as well as some $n \geq 0$. We consider the ends of
    \begin{align*}
        \coprod_{\{\phi \in \pi_2^R(\x,\y)| \ind_R(\phi) = 2, n_{w}(\phi) = n \}} \widehat{\cM}_R(\phi).
    \end{align*}
    
     If $\x \neq \y$, no spheres nor $\bT_\alpha$ boundary degenerations can appear by Lemmas \ref{lem:sym_monotone} and \ref{lem:disk_monotone}, since these carry real index at least 2.  Boundary degenerations on $M^R$ cannot be ruled out by the same index considerations, but by \Cref{lem:MR_degenerations}, such classes come in pairs, and therefore will not contribute the to the total sum. Hence, the ends in this case are given by
     \begin{align*}
         \coprod_{\y \in (\bT_\alpha\cap \bT_\beta)^R} \coprod_{\{\phi_1, \phi_2| \phi_1 * \phi_2 = \phi \}}\widehat{\cM}_R(\phi_1)\times\widehat{\cM}_R(\phi_2),
     \end{align*}
     which is of course the $U^{n} \y$ component of $\delp^2 (\x).$

     When $\x = \y$, sphere bubbles are ruled out by \Cref{lem:spheres} though there could be $\bT_\alpha$ boundary degenerations which contribute to $\delp^2( \x)$. Recall from Section~\ref{sec:setup} that we have a one-to-one correspondence between holomorphic strips with boundaries on $\Ta$ and $M^R$, and invariant holomorphic strips with boundaries on $\Ta$ and $\Tb$. In the latter interpretation, whenever a boundary degeneration on $\Ta$ contributes to $\delp^2$, we also have a contribution from its reflection, which is a boundary degeneration on $\Tb$. Since $w \in C$, we have that the $\Tb$ degeneration has the same value of $n_w$ as that for the $\Ta$ degenration. Altogether, the contribution of disk bubbles to $\delp^2(\x)$ consists of terms of the form
     \begin{equation}
     \label{eq:d2onew}
         \mathlarger{\sum}_{\{\phi \in \pi_2^R(\x, \bT_\alpha)| \ind(\phi) = 2\}}\# \widehat{\cN}(\phi) \cdot U^{2n_w(\phi)},
     \end{equation}
     where we see that the value of $n_w(\phi)$ in the exponent of $U$ got doubled. In any case, according to \Cref{lem:Talpha_degenerations}, when $r = 1$ (that is, $m=g$), we have $ \# \widehat{\cN}(\phi) \equiv 0 \!\!\mod 2$ so the expression \eqref{eq:d2onew} vanishes. 
\end{proof}

However, in the presence of additional basepoints, the curvature is nonzero. 

\begin{lemma}\label{lem:curvature}
    When $|\w| > 1$, $(\CFR^-(\cH, \frs), \delp)$ is a curved chain complex, with 
    \begin{align*}
        \delp^2 =\omega \cdot \id,
    \end{align*}
    where
    \begin{align*}
        \omega = U_{w_1}^2+ \hdots +U_{w_r}^2.
    \end{align*}
\end{lemma}
\begin{proof}
    The proof of Proposition~\ref{prop:onew} extends to this case, with one difference: when $\x = \y$, we have $\Ta$ boundary degenerations which contribute to $\delp^2 \x$. Let $\Sigma \smallsetminus \bm{\alpha} = A_1 \cup \hdots \cup A_r$ where $A_j$ contains the basepoint $w_j$. These domains are precisely the classes of $\Pi_2(\x, \Ta)$ with index 2  and by \Cref{lem:Talpha_degenerations} we have  $\# \widehat{\cN}(A_j) \equiv 1 \!\!\mod{2}$ for each of these classes. 
    
    A boundary degeneration on $\Ta$ with domain $A_j$ is paired with a  boundary degeneration on $\Tb$ with domain $B_j = - \tau(A_j)$. Since $w_j \in C$, $w_j$ is the unique basepoint contained in $B_j$. Hence, $U^{n_\w(A_j + B_j)}_\w = U^2_{w_j}$. Therefore,
    \[
         \delp^2 \x  = \mathlarger{\sum}_{\{\phi \in \pi_2^R(\x, \bT_\alpha)| \ind_R(\phi) = 2\} }\# \widehat{\cN}(\phi) \cdot U^{2n_\w(\phi)}\x = (U_{w_1}^2+ \hdots +U_{w_r}^2)\x,
   \]
     as claimed.
\end{proof}

In order to obtain honest complexes, we modify the differential by a bounding chain. For $|\w| > 1$, define
\begin{equation}
\label{eq:delp}
    b_\w := U_{w_1}+ \hdots +U_{w_r}
\end{equation}
and 
\begin{align*}
    \Del:= \delp + b_\w.
\end{align*}
When $|\w| = 1$, we define $b_\w := 0$, so that $\Del =\delp$.

\begin{corollary}
    The pair $(\CFR^-(\cH, \frs), \Del)$ is a chain complex for any $|\w|$.
\end{corollary}
\begin{proof}
    Since the endomorphism $\delp$ is $\F[U_\w]$-linear by definition, $b_\w$  satisfies the Maurer-Cartan equation:
\begin{align*}
    [\delp, b_\w] + b_\w^2= \omega,
\end{align*}
over $\Z/2$. It follows that
    \begin{align*}
        \Del^2 = (\delp + b_\w)^2 = \delp^2 + \omega  = 0,
    \end{align*}
    as desired.
\end{proof}

As usual, there are several algebraic variations of the real Heegaard Floer groups. Let:
\begin{align*}
    \CFR^\infty(\cH, \frs) &:= \CFR^-(\cH, \frs)\otimes_{\F[U_\w]}\F[U_\w, U_\w^{-1}], \\ 
    \CFR^+(\cH, \frs) &:= \CFR^\infty(\cH, \frs)/\CFR^-(\cH, \frs),\\ \widehat{\CFR}(\cH,\frs) &:= \CFR^-(\cH, \frs)\otimes_{\F[U_\w]}\F[U_\w]/(U_\w).
\end{align*}
These groups are related by the usual long exact sequences \cite[Section 4]{os_holodisks}. 

We obtain the different versions of {\em real Heegaard Floer homology}:
\begin{align*}
    \HFR^-(Y, \tau, \w, \s), \quad \HFR^\infty(Y, \tau, \w, \s), \quad \HFR^+(Y, \tau, \w, \s), \quad \widehat{\HFR}(Y, \tau, \w, \s).\end{align*}

For $\circ \in \{-, \infty, +, \hat{\phantom{o}} \}$, we let
$$ \HFR^\circ(Y, \tau, \w) = \bigoplus_{\s \in \RSpinC(Y, \tau)} \HFR^\circ(Y, \tau, \w, \s).$$

When the fixed point set $C$ is connected and we have a single basepoint $w$, we may also drop $w$ from the notation and denote the real Heegaard Floer groups by $\HFR^\circ(Y, \tau, \s)$ and $\HFR^\circ(Y, \tau)$.

\begin{rem}
    In the hat setting, since we only use domains that avoid the basepoints, we have a refined decomposition according to the relative real $\SpinC$ structures introduced in Section~\ref{sec:relative}:
    $$ \HFRh(Y, \tau, \w) = \bigoplus_{\srel \in \RSpinCrel(Y, \tau, \w)} \HFRh(Y, \tau, \w, \srel).$$
\end{rem}

\subsection{Grading}\label{subsec:grading} As shown in \cite[Section 4.2.1]{os_holodisks}, the Heegaard Floer complexes $\CF^\circ(\Ta, \Tb, \frs)$ have a relative $\Z/\delta(\frs)$-grading, where 
\begin{align*}
    \delta(\frs) = \gcd_{\xi \in H_2(Y)} \langle c_1(\frs), \xi \rangle.
\end{align*}

According to \Cref{eqn:CFR_grading}, when $\delta(\frs) = 2N$, we obtain a relative $\Z/N$-grading on the real Floer complexes $\CFR^\circ(\cH, \s)$. In particular, for any $(Y, \tau)$ with $H^1(Y)^{-\tau^*} = 0$, there are no real invariant periodic domains and  the associated complexes come with a relative $\Z$-grading. 

As mentioned in Section~\ref{sec:gradings}, one way to equip a Lagrangian Floer complex $\CF(L_0, L_1)$ with an absolute $\Z/2$-grading is  by orienting $L_0$ and $L_1$. In our situation, this is usually not possible. Recall that
\begin{align*}
    M^R = \coprod_k \Sym^{k}(\Sigma/\tau) \times \Sym^{m-2k}(C).
\end{align*}
The space $\Sym^{j}(S^1)$ is a disk bundle over $S^1$ which is orientable if and only if $j$ is odd; see \cite{Morton}. Thus, $\Sym^{m-2k}(C)$  may  be non-orientable. Furthermore, the quotient $\Sigma/\tau$ may also be non-orientable. See \Cref{prop:constructing real HDs}.

In some situations, we can arrange to work with real Heegaard diagram such that $\Sigma/\tau$ is orientable. This can be done when the image of $C$ in the quotient is null-homologous (see \Cref{rem:orientable quotients}). After we remove basepoints from each component of $C$, the Lagrangian $M^R$ will become orientable. (See \Cref{lem:MRo} for a proof of this fact.)  
Hence, we obtain another $\Z/2$-grading defined by a choice of orientations of the Lagrangians. This grading agrees with the relative Maslov grading when every real $\SpinC$-structure corresponds to a unique relative real $\SpinC$-structure (for example, this holds for branched covers of knots in $S^3)$; compare to \Cref{sec:relative} and \Cref{sec:euler}. In these cases, it follows from invariance that the $\Z/2$-grading obtained from orienting the Lagrangians is well-defined. This will allow us to compute the Euler characteristic of $\widehat{\HFR}(Y,\tau,\w, \frs)$ (where each component of $C$ contains at least one point of $\w$), which we do in Section~\ref{sec:euler}. In general, we expect that the $\Z/2$-grading determined by orienting the Lagrangians  depends not only on $(Y,\tau, \w)$, but also on the integral homology class of $[\Sigma/\tau]$.

\subsection{Homology Action} Let $\Omega(\Ta, \Tb)$ be the space of paths from $\Ta$ to $\Tb$ based at the constant path $\x \in \Ta \cap \Tb$. The Heegaard Floer complexes $\CF^\circ(\Ta, \Tb, \s)$ can be equipped with an action of $H^1(\Omega(\Ta, \Tb))$, by viewing a class $\phi \in \pi_2(\x, \y)$ as an arc in $\Omega(\Ta, \Tb)$ based at the constant paths $\x$ and $\y$ and defining
\begin{align*}
    A_\zeta(\x) = \mathlarger{\sum}_{\substack{\phi \in \pi_2(\x, \y)\\ \ind(\phi) = 1}} \zeta(\phi)\cdot \# \widehat{\cM}(\phi) \cdot U^{n_{\bm w}(\phi)} \y.
\end{align*}
As in \cite{os_holodisks,Lipshitz_cylindrical},  we may give a more concrete description of this action, making use of the isomorphism
\begin{align*}
    H^1(\Omega(\Ta, \Tb)) \cong \Z\oplus \Hom(H^1(Y), \Z) \cong \Z\oplus H_1(Y) / \tors.
\end{align*}
Our notation and exposition are largely based on \cite{ni_homology_actions,zemke_graphcob}. Given a oriented, closed loop $\gamma$ in $Y$ representing a class in $H_1(Y)$, we may homotope it to an immersed curve in $\Sigma$ which intersects the alpha and beta curves transversely and avoids their intersection points. Recall that we defined a quantity $a(\gamma, \phi)$, which was a signed count of intersections between $\gamma$ and $\partial_\alpha(\cD(\phi))$; see Equation \eqref{eqn: arc-domain count}. This quantity is clearly additive, in the sense that if $\phi_1 \in \pi_2(\x, \y)$ and $\phi_2 \in \pi_2(\y, \z)$, then 
\begin{align*}
    a(\gamma, \phi_1* \phi_2) = a(\gamma, \phi_1) + a(\gamma,\phi_2).
\end{align*}
This yields a map
\begin{align*}
    a: H_1(Y) \ra H^1(\Omega(\Ta, \Tb)),
\end{align*}
given by representing classes in $H_1(Y)$ by curves in $\Sigma$ and defining $a([\gamma], -) \in H^1(\Omega(\Ta, \Tb))$ by 
\begin{align*}
 a([\gamma], \phi) = \gamma \cdot \partial_\alpha \cD(\phi). 
\end{align*}
Further, this map descends to the quotient $H_1(Y)/\tors$ \cite{ni_homology_actions,zemke_graphcob}. 

In our setting, there is an action of anti-invariant classes. As a preliminary step, note that there is a map $\Omega(\Ta, M^R) \ra \Omega(\Ta, \Tb)$ given by applying our usual trick, taking a strip $\phi$ to the symmetric strip $\phi\cup R(\phi)$. This induces a map $H^1(\Omega(\Ta, \Tb)) \ra H^1(\Omega(\Ta, M^R))$. The composition
\begin{align*}
     H_1(Y)/\tors \xra{a} H^1(\Omega(\Ta, \Tb)) \ra H^1(\Omega(\Ta, M^R)),
\end{align*}
induces an action of $H_1(Y)/\tors$ on $\CF^\circ(\Ta, M^R)$ defined by
\begin{align}\label{eqn:homology action}
    A_{\gamma}(\x)= \sum_{\y \in \Ta \cap M^R} \sum_{\substack{\phi \in \pi_2^R(\x, \y)\\\ind_R(\phi) = 1}} a(\gamma, \phi)\cdot \# \widehat{\cM}_R(\phi) \cdot U^{n_\w(\phi)} \cdot \y.
\end{align}
Ultimately, we will show that this map factors through a more interesting group, but let us first establish a few basic facts about this action. 

\begin{lemma}
    Let $\cH = (\Sigma, \bm \alpha, \bm \beta, \w)$ be a real Heegaard diagram, and let $\gamma$ be an immersed closed curve in $\Sigma$ representing an element of $H_1(Y)/\tors$. Then, 
    \begin{align*}
        A_\gamma \Del + \Del A_\gamma = 0.
    \end{align*}
    If $\eta$ is another such curve, 
    \begin{align*}
        A_{\gamma* \eta} = A_\gamma+A_\eta. 
    \end{align*}
\end{lemma}
\begin{proof}
    The first claim follows just as in \cite{ni_homology_actions,zemke_graphcob} by considering classes with index 2, and looking at the ends of the associated compactified moduli spaces. Things are complicated slightly by the presence of bubbles on $\Ta$, so we spell out the proof. Since the bounding cochain $b_\w$ is acyclic, it suffices to show that $A_{\gamma}$ commutes with $\delp$. For each index 2 class $\phi \in \pi_2^R(\x,\z)$, the quantity
    \begin{align*}
        \sum_{\substack{\phi_0 \in \pi_2^R(\x, \y),\phi_1 \in \pi_2^R(\y, \z)\\ \ind_R(\phi_0) = \ind_R(\phi_1) = 1 \\\phi =\phi_0*\phi_1}} (a(\gamma, \phi_0) + a(\gamma, \phi_1))\cdot \# \widehat{\cM}_R(\phi_0)\# \widehat{\cM}_R(\phi_1) \cdot U^{n_\w(\phi_0)+n_\w(\phi_1)}  \z.
    \end{align*}
    is zero, provided $\z \ne \x$, as disk bubbles on $\Ta$ cannot appear by index considerations. Bubbles on $M^R$ cancel in pairs. Summing over all real index 2 classes yields the $\z$ component of $(A_\gamma\delp + \delp A_\gamma)(\x)$.  
    
    In the case that $\z = \x$, disk bubbles may appear, and the $\x$ component of $(A_\gamma\delp + \delp A_\gamma)(\x)$ is given by 
    \begin{align*}
        \sum_{w \in \w} a(\gamma, A_w) \cdot \#\widehat{\cN}(A_w) \cdot U_w^2 \cdot \x,
    \end{align*}
    where $A_w$ is the region of $\Sigma \smallsetminus \balpha$ containing $w \in \w$. However, since $\gamma$ is a closed curve, the quantity $a(\gamma, A_w)$ is zero for all $w$. Hence, the sum above still vanishes.
    
    Additivity follows from the additivity of the quantity $a(\gamma, \phi)$. 
\end{proof}

\begin{lemma}
\label{lem:AA}
    Let $\gamma$ be an immersed curve in $\Sigma$ such that $[\gamma] = 0 \in H_1(Y)/\tors$. Then, $A_\gamma \simeq 0$. Further, if $\eta$ is any class in $H_1(Y)/\tors$, then $A_\eta^2 \simeq 0.$
\end{lemma}
\begin{proof}
    These results follow just as in \cite[Lemmas 5.5, 5.6]{zemke_graphcob}. 
\end{proof}

We have defined an action of all of $H_1(Y)/\tors$; we claim that this factors through a certain quotient. Viewing $H_1(Y)/\tors$ as $\Hom(H^1(Y), \Z)$, observe that there is an exact sequence 
\begin{align}
\label{eq:HH1}
    0 \to (H_1(Y)/\tors)^{\tau_*}  \to H_1(Y)/\tors \xra{\mathrm{ev}} \Hom(H^1(Y)^{-\tau^*}, \Z) \to 0,
\end{align}
giving us an isomorphism 
$$
\Hom(H^1(Y)^{-\tau^*}, \Z) \cong (H_1(Y)/\tors)/ (H_1(Y)/\tors)^{\tau_*}.
$$ 
Since the notation is rather cluttered, we will write $\cH_1(Y, \tau)$ for this group. Note that the group $\cH_1(Y, \tau)$ is (non-canonically) isomorphic to the anti-invariant part of $H_1(Y)/\tors$, but the natural map $H_1(Y)^{-\tau_*}/\tors \to \cH_1(Y, \tau)$ (given by inclusion into $H_1(Y)/\tors$ followed by projection) is not surjective; classes of the form $\alpha - \tau(\alpha)$ are mapped to $2[\alpha] = -2[\tau(\alpha)]$. For this reason, $\cH_1(Y, \tau)$ is the more natural object. (Compare Example~\ref{ex:connsum} below.)

\begin{proposition}
    The map
\begin{align*}
    H_1(Y) \to \End(\CFR^\circ(Y, \tau)), \quad \gamma \mapsto A_\gamma,
\end{align*}
descends to a map
\begin{align*}
    \cH_1(Y, \tau) \to \End(\CFR^\circ(Y, \tau))
\end{align*}
In particular, for each $[\gamma] \in \cH_1(Y, \tau)$ there is a map 
\begin{align*}
    A_{[\gamma]}^R: \CFR^\circ(Y, \tau) \to \CFR^\circ(Y, \tau)
\end{align*}
well-defined up to chain homotopy giving $\HFR^\circ(Y,\tau)$ the structure of a $\Lambda^* \cH_1(Y, \tau)$-module.
\end{proposition}
\begin{proof}
    Consider a closed curve $\gamma \sub \Sigma$ representing a class in $H_1(Y)/\tors$. Recall that for $\phi \in \pi_2^R(\x, \y)$, the quantity 
    \begin{align*}
        a(\gamma, \phi):=\gamma \cdot \partial_\alpha \cD(\phi),
    \end{align*}
    defines a map 
    \begin{align*}
        H_1(Y)/\tors \to H^1(\Omega(\Ta, M^R)).
    \end{align*}
    We could also have considered 
    \begin{align*}
        b(\gamma, \phi):=\gamma \cdot \partial_\beta \cD(\phi),
    \end{align*}
    which gives rise to another map $H_1(Y)/\tors \to H^1(\Omega(\Ta, M^R))$. Because $\gamma$ is closed, we have that $a(\gamma, \phi)+b(\gamma, \phi) = \gamma \cdot \partial \cD(\phi) = 0$ . 

    Given a real domain $\cD(\phi)$, the symmetry forces the relation $\partial_\beta \cD(\phi) = -\tau_*(\partial_\alpha \cD(\phi))$. Furthermore, since $\tau$ restricts to an orientation reversing involution of $\Sigma$, we have that 
    \begin{align*}
    \tau_*(a) \cdot \tau_*(b) = - a \cdot b, \quad a, b \in H_1(\Sigma).
    \end{align*}
    It follows that
    \begin{align*}
        a(\tau(\gamma), \phi) &= \tau_*(\gamma)\cdot \partial_\alpha\cD(\phi)\\
        &= \tau_*(\gamma)\cdot (-\tau_*(\partial_\beta\cD(\phi)))\\
        &=\gamma \cdot \partial_\beta \cD(\phi)\\
        &= b(\gamma, \phi)\\
        &= -a(\gamma, \phi).
    \end{align*}
    Therefore, if $\gamma$ represents an invariant class in $(H_1(Y)/\tors)^{\tau_*}$, we have
    \begin{align*}
        0 &= a(\gamma, \phi) - a(\tau(\gamma),\phi)\\
        & = a(\gamma, \phi) + a(\gamma, \phi)\\
        &= 2 a(\gamma, \phi).
    \end{align*}
    for any class $\phi \in \pi_2^R(\x,\y)$ Hence, we have 
    \begin{align*}
        2a(\gamma, {-}) = 0 \in H^1(\Omega(\Ta, M_R)).
    \end{align*}
    As $H^1(\Omega(\Ta, M_R))$ is torsion free, this forces $a(\gamma, {-}) = 0$, and therefore $A_\gamma \sim 0$ for any $\gamma \in (H_1(Y)/\tors)^{\tau_*}$. The conclusion then follows from the exact sequence \eqref{eq:HH1}.  
\end{proof}

\section{Proof of invariance}\label{sec:invariance}
Invariance of $\CFR^\circ(Y, \tau, \w)$ follows much as in the standard case: we define maps associated to changes in the almost complex structure, real isotopies, real handleslides, and real stabilizations, and show that these operations induce homotopy equivalences between complexes. Since the bounding chain $b_\w$ introduced in \Cref{eq:delp} is central, we have that for any morphism $f$,
\begin{align*}
    [\delp, f] = [\Del, f].
\end{align*}
For this reason, we are free to ignore $b_\w$, and we will do so throughout this section.

\subsection{Almost Complex Structures and Isotopies}
As $\CFR^\circ(Y, \tau, \s)$ is defined to be the usual Lagrangian Floer complex of the alpha torus with the fixed set, the standard proof of the independence of the choice of almost complex structure holds in our setting. 

\begin{prop}
\label{prop:invarianceJ}
    Let $(Y, \tau, \w, \s)$ be a real pointed 3-manifold with a $\frs$-strongly admissible real Heegaard diagram $(\Sigma, \bm \alpha, \bm \beta, \w, \tau)$. Then, the complex $\CFR^\circ_J(\Sigma, \bm \alpha, \bm \beta, \w, \s)$ is independent of the choice of almost complex structures $J \in \JR$, up to chain homotopy equivalence.
\end{prop}
\begin{proof}
    This follows from \Cref{lem:independence of acs}. 
\end{proof}

The change of almost complex structure maps also commute with the $H_1(Y)^{-\tau_*}/\tors$-action. 

\begin{lem}
    For $\gamma \in H_1(Y)^{-\tau_*}/\tors$,
    \begin{align*}
        A_\gamma^R \circ \Phi_{J_s} \simeq \Phi_{J_s} \circ A_\gamma^R
    \end{align*}
    where $J_s$ is a path in $\JR$.
\end{lem}
\begin{proof}
    Keeping the notation from the proof of \Cref{lem:independence of acs}, define 
    \begin{align*}
        H_{J_s}(\x) = \sum_{\y} \sum_{\phi \in \pi_2(\x,\y), \ind_R(\phi) = 0} a(\phi, \gamma)\# \cM_R^{J_s}(\phi) U^{n_\w(\x)}\y.
    \end{align*}
    If $\phi \in \pi_2(\x, \y)$ is a strip with $\ind_R(\phi) = 1$, the compactified space $\overline{\cM_R^{J_s}(\phi)}$ has ends which correspond to the terms in $A_\gamma^R \circ \Phi_{J_s}$, $\Phi_{J_s} \circ A_\gamma^R$, $\partial_{J_1}\circ H_{J_s}$, and $H_{J_s}\circ \partial_{J_0}$. Hence, the sum of these maps is trivial.
\end{proof}

Invariance under isotopies of the attaching curves which do not introduce new intersection points follows from the independence of the almost complex structure just as in \cite[Theorem 7.3]{os_holodisks}. Hence, it suffices to consider real isotopies of the attaching circles which introduce canceling pairs of intersection points.  Let $\Psi_s$ be an exact Hamiltonian isotopy of the alpha curves which is supported in the complement of the basepoints. We can assume $\Psi_s$ depends on a parameter $s \in \R$ but is constant for $s \leq 0$ and also for $s \geq 1$. Of course, $\Psi_s$ induces an isotopy of $\bT_\beta$ as well. As usual, for $\x \in \bT_\alpha \cap M^R$ and $\y \in \Psi_1(\bT_\alpha) \cap M^R$, we consider homotopy classes of disks $\pi_2^{\Psi_s}(\x,\y)$ with \emph{dynamic boundary conditions}, which are homotopy classes of maps $u: \R \times [0,1] \ra \Sym^m(\Sigma)$ satisfying  
\begin{align*}
u(s,0) \in \Psi_s(\bT_\alpha),\ \ 
    u(s,1) \in M^R\\
    \lim_{s\ra -\infty} u(s,t) = \x,  \ \ 
    \lim_{s\ra\infty} u(s,t) = \y.
\end{align*}
Let $\cM^{\Psi_s}(\phi)$ be the moduli space of holomorphic strips in the class $\phi \in \pi_2^{\Psi_s}(\x,\y)$. Define 
\begin{align*}
    \Gamma_{\Psi_s}: \CFR^\circ(\Sigma,\bm \alpha, \bm \beta, \w, \s) \ra \CFR^\circ(\Sigma,\bm \alpha', \bm \beta',\w, \s)
\end{align*}
by 
\begin{align}\label{eqn:continuation_map}
    \Gamma_{\Psi_s}(\x) = 
    \mathlarger{\sum}_{\y}
    \mathlarger{\sum}_{\substack{\phi \in \pi_2^{\Psi_s}(\x, \y)\\ \mu(\phi) = 0}} \# \widehat{\cM}^{\Psi_s}(\phi) U^{n_{\bm w}(\phi)} \y,
    \end{align}
where $\mu(\phi)$ is the expected dimension of $\widehat{\cM}^{\Psi_s}(\phi)$. 

\begin{prop}
\label{prop:invariance isotopy}
    Let $(\Sigma, \bm \alpha, \bm \beta, \w)$ be a real Heegaard diagram for $(Y, \tau)$. Let  $(\Sigma, \bm \alpha', \bm \beta', \w)$ be the real Heegaard diagram obtained by applying a real isotopy to $\alpha_i$ (and hence also to $\beta_i = \tau(\alpha_i)).$ Then, by composing maps $ \Gamma_{\Psi_s}$ we obtain a map
    \begin{align*}
    \Phi_{\alpha \ra \alpha'}^{\beta\ra \beta'}: \CFR^\circ(\Sigma,\bm \alpha, \bm \beta, \w, \s) \ra \CFR^\circ(\Sigma,\bm \alpha', \bm \beta',\w, \s)
\end{align*}
    which is a homotopy equivalence. Further, these maps commute with the $H_1(Y)^{-\tau_*}/\tors$-action.
\end{prop}
\begin{proof}
    We show that the maps $\Gamma_{\Phi_s}$ are homotopy equivalences. The spaces $\cM^{\Phi_s}(\phi)$ have Gromov compactifications. The 1-dimensional moduli spaces have no ends in which disks or spheres bubble off by \Cref{lem:Talpha_degenerations} and \Cref{lem:MR_degenerations}. It follows that $\Gamma_{\Psi_s}$ is a chain map. The map $\Gamma_{\Psi_{1-s}}$ is a chain homotopy inverse for $\Gamma_{\Psi_{s}}$ by a standard argument. The claim follows. 

    Commutativity with $A_\gamma^R$ follows by an argument similar to that in the case of almost complex structures. 
\end{proof}

\subsection{Handleslides}

Rather than proving handleslides invariance by way of the triangle counting argument of \cite{os_holodisks}, we appeal to the work of Perutz. 

\begin{prop}
\label{prop:invariance hslide}
    Let $(\Sigma, \bm \alpha, \bm \beta, \w)$ be a strongly $\frs$-admissible real Heegaard diagram for $(Y, \tau)$ and let $(\Sigma, \bm \alpha', \bm \beta', \w)$ be the diagram obtained by performing a real handleslide. Then, there is a canonical map 
    \begin{align*}
    \Gamma_{\Psi_t}: \CFR^\circ(\Sigma,\bm \alpha, \bm \beta, \w, \s) \ra \CFR^\circ(\Sigma,\bm \alpha', \bm \beta',\w, \s)
    \end{align*}
    which is a homotopy equivalence.
\end{prop}
\begin{proof}
    According to \cite[Theorem 1.2]{perutz}, the symplectic form $\omega$ may chosen so that $\bT_\alpha$ and $\bT_{\alpha'}$ are Hamiltonian isotopic, via $\Psi_t$. (Perutz stated the theorem for the case $m=g$, but the proof works for any $m$.) Therefore, there is a continuation map 
    \begin{align*}
        \Gamma_{\Psi^t}: \CF^\circ(\bT_\alpha, M^R, \frs) \ra \CF^\circ(\bT_\alpha', M^R, \frs).
    \end{align*}
    It remains to show that this map is well-defined, i.e. that the $\omega$-area of the disks involved are controlled by the Maslov index and the intersection number with $\w \times \Sym^{m-k}(\Sigma)$. Given a strip $\phi_0$ with dynamic alpha boundary, Perutz considers a sequence of curves starting at $\phi_0$ and limiting to the nodal curve shown in \Cref{fig:dynamic.pdf}, which consists of a triangle $T$ and a disk $D$. By the arguments of \cite[Section 9.2]{os_holodisks}, the strong admissibility of the diagram guarantees that there are finitely many index 0 holomorphic triangles with given intersection with $\w \times \Sym^{m-k}(\Sigma)$; similarly for $D$. It follows the maps are well defined. 
    
    In the real setting, the triangle-disk configuration doubles to a nodal curve in which a pair of disks $D_0$ and $D_1$ split off from a rectangle $R$ (again, reference \Cref{fig:dynamic.pdf}.)  Admissibility of the quadruple diagram $(\Sigma, \bm \alpha', \bm \alpha, \bm \beta, \bm \beta')$ follows much like \cite[Lemma 9.6]{os_holodisks}. Admissibility of a quadruple diagram depends not on a single $\SpinC$ structure, but on a 
    \begin{align*}
        \delta H_1(Y_{\alpha', \beta}) + \delta H_1(Y_{\alpha, \beta'})
    \end{align*}
    orbit of $\SpinC$ structures. But, since $\bm \alpha$ and $\bm \alpha'$ span the same subspace of $H_1(\Sigma)$, every $(\bm \alpha', \bm \beta)$-periodic domain can be written as a sum of $(\bm \alpha', \bm \alpha)$ and $(\bm \alpha, \bm \beta)$ periodic domains, and hence $\delta H_1(Y_{\alpha', \beta}) = 0$. The symmetric argument shows that $\delta H_1(Y_{\alpha, \beta'}) = 0$. Therefore, these $\SpinC$ orbits are trivial. Consequently, we can always assume the handleslide quadruple diagram $(\Sigma, \bm \alpha', \bm \alpha, \bm \beta, \bm \beta')$ is strongly admissible. It then follows by the argument of Perutz that the continuation maps are indeed well-defined. \end{proof}

    \begin{figure}
        \centering
        \includegraphics[width=0.9\linewidth]{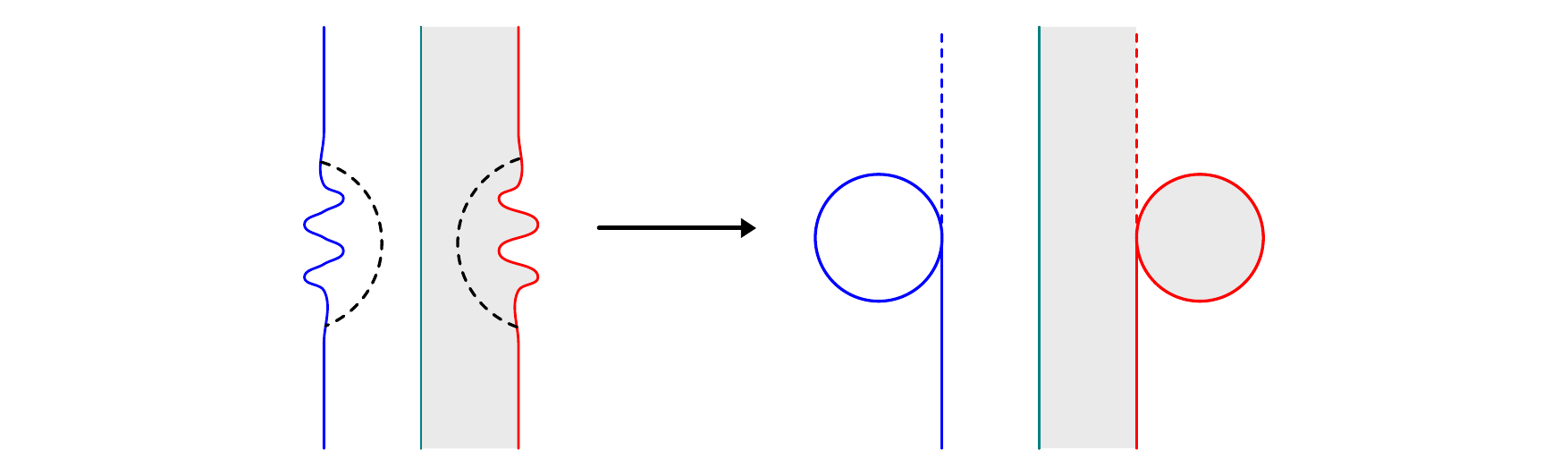}
        \caption{Left: A real invariant strip with dynamic boundary conditions. The dashed line indicates where the almost complex structure is pinched. Right:  A nodal invariant curve, consisting of a pair of monogons and a real invariant rectangle.}
        \label{fig:dynamic.pdf}
    \end{figure}

\subsection{Stabilizations}

At last, we turn to real stabilizations. 

\begin{prop}\label{prop:doublestabilization}
    Let $\cH$ be a strongly admissible real Heegaard diagram for $(Y, \tau, \s)$ and let $\cH'$ be the diagram obtained by a real stabilization. Then, there is a homotopy equivalence
    \begin{align*}
        \CFR^\circ(\cH, \s) \ra \CFR^\circ(\cH', \s). 
    \end{align*}
\end{prop}
\begin{proof}
The stabilization arguments from \cite{os_holodisks} extend to the real setting, as follows. 

First, we consider the case of a free stabilization. Let $p\in \Sigma \smallsetminus C$ and let $q = \tau(p)$. Let $E$ the standard genus 1 Heegaard diagram for $S^3$ with a single intersection point, $c$, and let $E'$ be the diagram obtained from $E$ by swapping the alpha and beta curve. Let $c'$ the point in $E'$ corresponding to $c$.

Let $J \in \JR$ be a path of symmetric almost complex structures on $\Sym^m(\Sigma)$ and let $J(T)$ be the path of almost complex structures obtained by inserting a pair of cylinders $S^1 \times [-T, T]$ connecting $\Sigma$ to $E$ and $\tau E$.
Fix $\x, \y \in \bT_\alpha \cap M^R$ as well as $\phi \in \pi_2^R(\x, \y)$ with $\ind_R(\phi) = 1$. We may assume that each $u \in \widehat{\cM}_R^J(\phi)$ of index one meets the connected sum points $p$ and $q$ transversely. By taking a product with the constant maps to $c$ and $c'$, we obtain a map into $E \times \Sym^m(\Sigma) \times \tau E  \sub \Sym^{m+1}(E \vee \Sigma \vee \tau E)$. By \cite[Lemma 10.3]{os_holodisks}, through each pair of points in $\Sym^2(E)$ there is a unique holomorphic sphere, $S$, in the class of the positive generator of $\pi_2(\Sym^2(E))$. By splicing in pairs of such spheres, we obtain a map
\begin{align*}
    \D \ra \Sym^2(E)\times\Sym^{m-1}(\Sigma)\times \Sym^2(\tau E) \sub \Sym^{m+1}(E \vee \Sigma \vee \tau E),
\end{align*}
as in \cite[Section 10.2]{os_holodisks}. For large $T$, this map becomes close to being $J(T)$-holomorphic, and an application of the inverse function theorem produces a nearby $J$-holomorphic map in $\cM^{J(T)}_R(S\# \phi\# \tau S)$. By \cite[Theorem 10.4]{os_holodisks}, for sufficiently large $T$, there is an identification $\cM_R^{J}(\phi) \cong \cM_R^{J_s(T)}(S\# \phi\# \tau S)$. Hence, it follows that 
\begin{align*}
    \CFR^\circ_{J_s}(\cH,\s) \simeq \CFR^\circ_{J_s(T)}(\cH',\s).
\end{align*}
The theorem then follows from the independence of almost complex structure. 

The case of a fixed set stabilization is proven in the same manner by again fixing a symmetric almost complex structure, inserting a long neck at a point $p$ contained in the fixed set of $\tau$, and splicing in spheres. 

It is clear that stabilization does not affect the quantity $a(\phi, \gamma)$. Hence, the stabilization maps respect the action of $A_\gamma^R$. 
\end{proof}

\begin{corollary}
\label{cor:invariance}
The chain homotopy equivalence class of $\CFR^\circ(\cH,\s)$, as a $\F[U]\otimes (H_1(Y)^{-\tau_*}/\tors)$-module, is a well-defined invariant of $(Y, \tau, \w, \s)$.
\end{corollary}

\begin{proof}
Put together Lemma~\ref{lem:isotopy of real admissible} and Propositions~\ref{prop:invarianceJ}, \ref{prop:invariance isotopy}, \ref{prop:invariance hslide}, and \ref{prop:doublestabilization}. 
\end{proof}

\begin{proof}[Proof of Theorem~\ref{thm:invariant}] This
follows immediately from  Corollary~\ref{cor:invariance}, in the particular case when $\w$ consists of a single basepoint (so we can take homology). Note that, on homology,  Lemma~\ref{lem:AA} implies that the action of $H_1(Y)^{-\tau_*}/\tors$ extends to one of its exterior algebra.
\end{proof}

\section{Examples}\label{sec:exs}

In this section, we carry out a few computations. 

\begin{example}
    In \Cref{fig:s3_diagrams}, there are three real Heegaard diagrams for $(S^3, \tau)$. There is a single generator in all three examples, and therefore, the differential must be trivial. Hence, $\CFR^-(S^3, \tau) \cong \F[U]$ in its unique real $\SpinC$-structure.
\end{example}

\begin{figure}[h]
\def\svgwidth{.8\linewidth}
\begingroup%
  \makeatletter%
  \providecommand\color[2][]{%
    \errmessage{(Inkscape) Color is used for the text in Inkscape, but the package 'color.sty' is not loaded}%
    \renewcommand\color[2][]{}%
  }%
  \providecommand\transparent[1]{%
    \errmessage{(Inkscape) Transparency is used (non-zero) for the text in Inkscape, but the package 'transparent.sty' is not loaded}%
    \renewcommand\transparent[1]{}%
  }%
  \providecommand\rotatebox[2]{#2}%
  \newcommand*\fsize{\dimexpr\f@size pt\relax}%
  \newcommand*\lineheight[1]{\fontsize{\fsize}{#1\fsize}\selectfont}%
  \ifx\svgwidth\undefined%
    \setlength{\unitlength}{640.62992126bp}%
    \ifx\svgscale\undefined%
      \relax%
    \else%
      \setlength{\unitlength}{\unitlength * \real{\svgscale}}%
    \fi%
  \else%
    \setlength{\unitlength}{\svgwidth}%
  \fi%
  \global\let\svgwidth\undefined%
  \global\let\svgscale\undefined%
  \makeatother%
  \begin{picture}(1,0.39823009)%
    \lineheight{1}%
    \setlength\tabcolsep{0pt}%
    \put(0,0){\includegraphics[width=\unitlength,page=1]{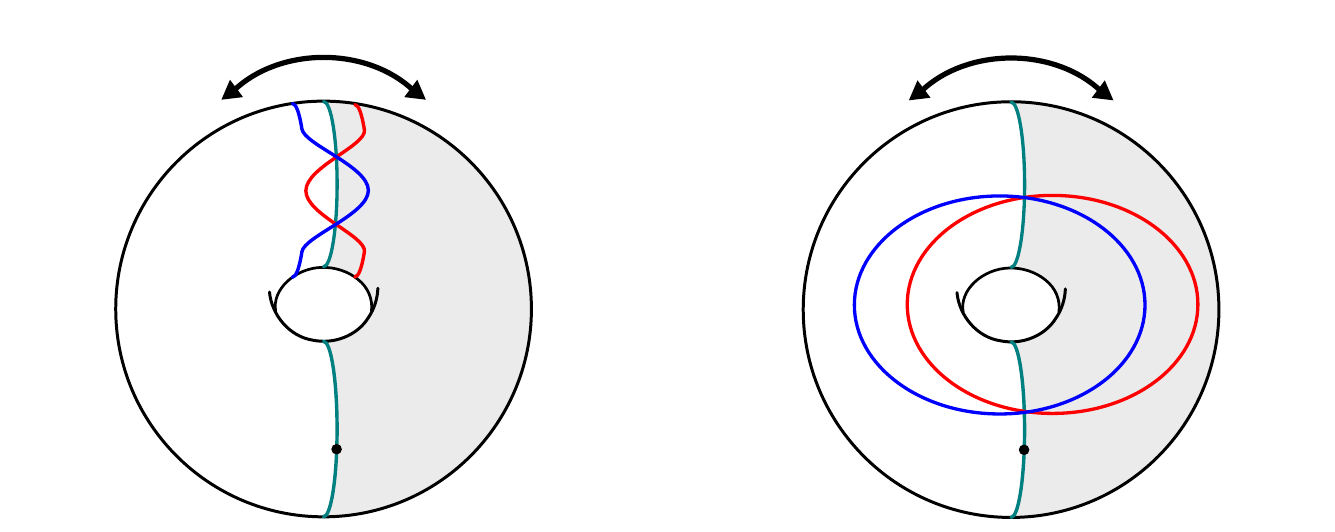}}%
    \put(0.23671165,0.36946679){\color[rgb]{0,0,0}\makebox(0,0)[lt]{\smash{\begin{tabular}[t]{l}{$\tau_1$}\end{tabular}}}}%
    \put(0.7518297,0.36946679){\color[rgb]{0,0,0}\makebox(0,0)[lt]{\smash{\begin{tabular}[t]{l}{$\tau_2$}\end{tabular}}}}%
  \end{picture}%
\endgroup%

\caption{Real Heegaard diagrams for two involutions on $S^1 \times S^2$.}
    \label{fig:ex_s1s2_unlink}
\end{figure}

\begin{example}
\label{ex:s1xs2 reflections}
    Consider the Heegaard diagram for $S^1\times S^2$ shown in the first frame of \Cref{fig:ex_s1s2_unlink}. There is a involution $\tau_1$ on $\Sigma$ given by reflection through the horizontal plane; the fixed set of this involution is drawn in green. The sphere $\pt\times S^2$ is obtained by capping off the single periodic domain, and is reflected under $\tau$. Similarly, the fiber $S^1 \times \pt$ is also mapped onto itself by a reflection by $\tau$. Therefore, this real Heegaard diagram represents $S^1 \times S^2$ equipped with the involution which is reflection on both factors; equivalently, this real $S^1 \times S^2$ is obtained by taking the branched double cover of the two component unlink. The kernel of $\Theta$ is all of $H^2(S^1 \times S^2)$, and therefore every $\SpinC$-structure on $S^1 \times S^2$ admits a (unique) real structure. 

    Both intersection points between $\alpha$ and $\beta$ lie on the fixed set, and therefore represent classes in $\CFR^\circ(S^1 \times S^2, \tau_1, \frs_0)$, where $\frs_0$ is the torsion $\SpinC$-structure. There are two real domains connecting $\theta^+$ to $\theta^-$, each with real index 1, and each having a single holomorphic representative. Hence, the differential on the Heegaard Floer complex is trivial, and we obtain 
    $$\HFR^-(S^1 \times S^2, \tau_1) \cong H_*(S^1) \otimes \F[U],$$ 
    in agreement with Li's calculation in the monopole setting \cite[Section 14.3]{li:HMR}.
    The other real $\SpinC$-structures can be realized by winding $\alpha$ and $\beta$ in opposite directions along a a curve isotopic to $S^1 \times \pt$. Just as in the classical case, 
    \begin{align*}
        \HFR^-(S^1 \times S^2;\frs_0 \pm nh) \cong \F[U]/(U^n-1),
    \end{align*}
    where $h$ is a generator for $H_2(S^1 \times S^2)$ and $n > 0$. Compare \cite[Section 14.3]{li:HMR}.

    With regard to the $\cH_1(S^1 \times S^2,\tau_1)$-action, the group $H_1(S^1 \times S^2)$ is generated by $\gamma := [S^1 \times \pt]$ on which $(\tau_1)_*$ acts by $-1$. There are no invariant classes and the group is torsion free, so $$H_1(S^1 \times S^2) = H_1(S^1 \times S^2)^{-(\tau_1)_*}  \cong \cH_1(S^1 \times S^2, \tau_1).$$ This curve is shown in the diagram in the left frame of \Cref{fig:action.pdf} and intersects a single bigon between the generators. Hence, $A_{[\gamma]}^R(\theta^+) = \theta^-$ and $A_{[\gamma]}^R(\theta^-) = 0$. 
\end{example}

\begin{example}\label{ex: s1xs2 rotation}
    In the second frame of \Cref{fig:ex_s1s2_unlink} there is another Heegaard diagram for $S^1 \times S^2$, and we consider the same involution on $\Sigma$, reflection through the vertical plane. However, this is not the same involution of $S^1 \times S^2$. This involution is the identity on the $S^1$ factor and rotation by $\pi$ on $S^2$. The fixed set is now two copies of the $S^1$-fiber. We call this involution $\tau_2$. In this case, $H^2(S^1 \times S^2)^{-\tau^*} = 0$, and therefore only the torsion $\SpinC$-structure admits a real structure. In fact, it admits two, as $H^1(S^1\times S^2)^{\tau_*}/(1 + \tau_*) = \Z/2\Z$. The two intersection points in this diagram realize different real $\SpinC$-structures, and the obstruction $\zeta$ is visible as the $\tau_*$-invariant $S^2$ obtained by capping off the single periodic domain in this figure. There can be no real invariant domains, and hence $$\HFR^-(S^1 \times S^2, \tau_2, \frs) \cong \F[U]$$ in each real $\SpinC$-structure $\s$. We also remark that if one winds symmetrically to obtain admissibility for those $\SpinC$-structures which do not admit real structures, no new invariant intersection points are created. As there are no invariant domains, the $\cH(S^1 \times S^2, \tau_2)$-action is trivial. Compare \cite[Section 14.4]{li:HMR}.
\end{example}

    The following example demonstrates why we prefer the action of $\cH_1(Y, \tau)$ to that of $H_1(Y)^{-\tau_*}/\tors$.

    \begin{figure}
        \centering
        \def\svgwidth{.8\linewidth}
\begingroup%
  \makeatletter%
  \providecommand\color[2][]{%
    \errmessage{(Inkscape) Color is used for the text in Inkscape, but the package 'color.sty' is not loaded}%
    \renewcommand\color[2][]{}%
  }%
  \providecommand\transparent[1]{%
    \errmessage{(Inkscape) Transparency is used (non-zero) for the text in Inkscape, but the package 'transparent.sty' is not loaded}%
    \renewcommand\transparent[1]{}%
  }%
  \providecommand\rotatebox[2]{#2}%
  \newcommand*\fsize{\dimexpr\f@size pt\relax}%
  \newcommand*\lineheight[1]{\fontsize{\fsize}{#1\fsize}\selectfont}%
  \ifx\svgwidth\undefined%
    \setlength{\unitlength}{841.88976378bp}%
    \ifx\svgscale\undefined%
      \relax%
    \else%
      \setlength{\unitlength}{\unitlength * \real{\svgscale}}%
    \fi%
  \else%
    \setlength{\unitlength}{\svgwidth}%
  \fi%
  \global\let\svgwidth\undefined%
  \global\let\svgscale\undefined%
  \makeatother%
  \begin{picture}(1,0.3030303)%
    \lineheight{1}%
    \setlength\tabcolsep{0pt}%
    \put(0,0){\includegraphics[width=\unitlength,page=1]{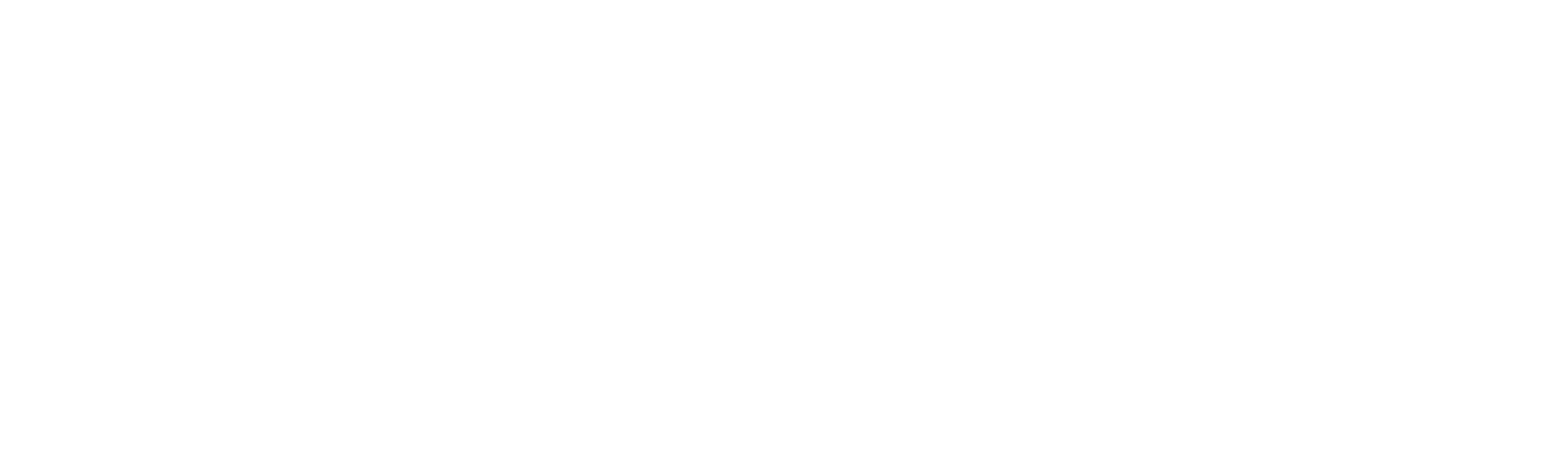}}%
    \put(0.19096515,0.12445655){\color[rgb]{0,0,0}\makebox(0,0)[t]{\smash{\begin{tabular}[t]{c}{\small$\gamma$}\end{tabular}}}}%
    \put(0,0){\includegraphics[width=\unitlength,page=2]{action.pdf}}%
    \put(0.25537448,0.28153169){\color[rgb]{0,0,0}\makebox(0,0)[lt]{\smash{\begin{tabular}[t]{l}{$\tau$}\end{tabular}}}}%
    \put(0,0){\includegraphics[width=\unitlength,page=3]{action.pdf}}%
    \put(0.58381519,0.15733562){\color[rgb]{0,0,0}\makebox(0,0)[t]{\smash{\begin{tabular}[t]{c}{\small$\gamma$}\end{tabular}}}}%
    \put(0,0){\includegraphics[width=\unitlength,page=4]{action.pdf}}%
    \put(0.69738872,0.24869024){\color[rgb]{0,0,0}\makebox(0,0)[lt]{\smash{\begin{tabular}[t]{l}{$\sigma$}\end{tabular}}}}%
    \put(0,0){\includegraphics[width=\unitlength,page=5]{action.pdf}}%
    \put(0.81022669,0.1584521){\color[rgb]{0,0,0}\makebox(0,0)[t]{\smash{\begin{tabular}[t]{c}{\small$\tau(\gamma)$}\end{tabular}}}}%
  \end{picture}%
\endgroup%

        \caption{Left: The action of $A_\gamma^R$ in Example~\ref{ex:s1xs2 reflections}. Right: The action of $A_\gamma^R$ in Example~\ref{ex:connsum}.}
        \label{fig:action.pdf}
    \end{figure}
    
\begin{example}
\label{ex:connsum}
     Consider $(Y \#Y, \sigma)$ with $Y = S^1 \times S^2$, where $\sigma$ swaps the two summands. A real Heegaard diagram for $(Y \# Y, \sigma)$ is shown on the right of \Cref{fig:action.pdf}. The differential vanishes and so $\HFR^\circ(Y\#Y, \sigma)$ is generated by the generators $\theta^+\times\sigma(\theta^+)$ and $\theta^-\times\sigma(\theta^-)$. The group $H_1(Y \# Y)$ is generated by $\gamma$ and $\sigma(\gamma)$ where $\gamma$ is the $S^1$-fiber of the first summand. In this case, $H_1(Y)^{-\sigma_*}$ is generated by $\gamma - \sigma(\gamma)$. The action of $A_{\gamma - \sigma(\gamma)}$ takes 
    $$
    \theta^+ \times \sigma(\theta^+)\mapsto 2(\theta^- \times \sigma(\theta^-)) \equiv 0 \mod 2
    $$ 
    since $a(\gamma - \sigma(\gamma),\phi) = \pm2$ for the domain shown in the figure. The action is trivial on $\theta^- \times \sigma(\theta^-)$. On the other hand, $\cH_1(Y, \sigma)$ is generated by $[\gamma] = -[\sigma(\gamma)]$, and therefore $$A_{[\gamma]}^R(\theta^+\times\sigma(\theta^+))  = \theta^-\times\sigma(\theta^-)= A_\gamma(\theta^+\times\sigma(\theta^+))$$ and $A_{[\gamma]}^R$ annihilates $\theta^- \times \sigma(\theta^-)$.
\end{example}

\subsection{Branched double covers}
Of course, a very natural class of examples of real 3-manifolds comes from the branched double covers of knots (and links) in $S^3$. Note that for branched double covers of knots in $S^3$, every $\SpinC$-structure admits a unique real structure. As in \Cref{sec:rhd}, we can construct real Heegaard diagrams for $\Sigma_2(L)$ from free spanning surfaces for $L$. For links in $S^3$ (or more generally, for null-homologous links in other 3-manifolds), we can always work with orientable Seifert surfaces; cf. Remark~\ref{rem:orientable quotients}. We let $\tau =\tau_L$ be the branching involution.

\begin{example}\label{ex:torus}
\cite{ni_homology_actions,zemke_graphcob},Consider the torus knot $T(2, 2n+1)$, whose double branched cover is the lens space $L(2n+1, 1)$. This knot admits a genus 1 real Heegaard diagram, obtained by thickening up the $(2n+1)$-twist M\"obius band $M_{2n+1}$ in $S^3$. The resulting handlebody is the standard solid torus in $S^3$, and its complement $U$ is also a solid torus. We choose the usual beta curve which bounds a disk in the complement of the standard solid torus. To obtain a Heegaard diagram for $\Sigma_2(T(2, 2n+1))$, we apply the involution on $\partial \nu(M_{2n+1})$ to the beta curve to obtain the alpha curve. It is easy to see that this process produces the usual Heegaard diagram for $L(2n+1, 1)$. Every $\SpinC$-structure admits a single real $\SpinC$-structure. Hence, 
\begin{align*}
    \HFR^-(\Sigma_2(T(2, 2n+1)), \tau, \frs) \cong \HF^-(\Sigma_2(T(2, 2n+1)), \frs) \cong \F[U],
\end{align*}
for each real $\SpinC$-structure $\frs$.
\end{example}

\begin{figure}
    \centering
    \includegraphics[width=0.5\linewidth]{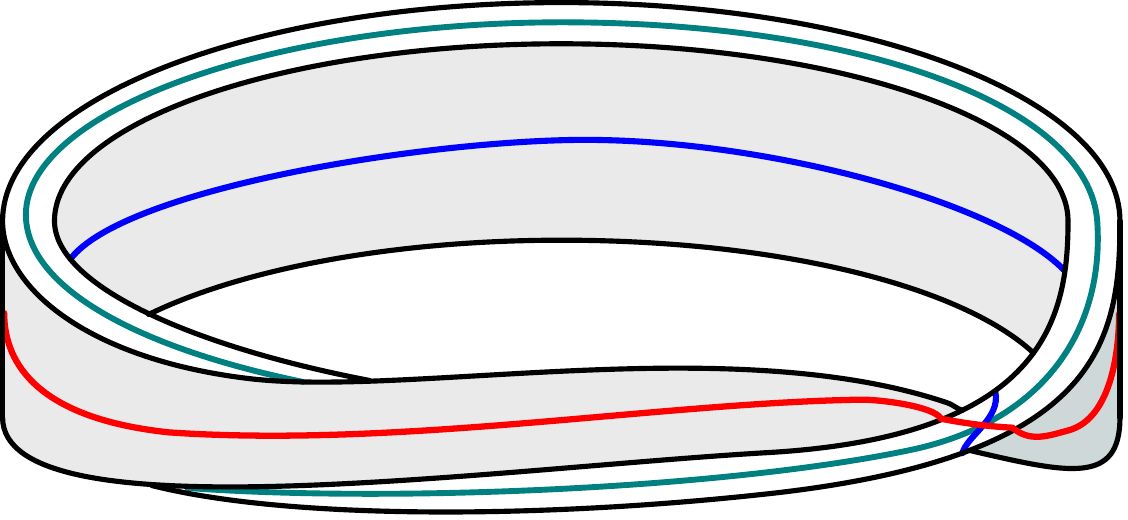}
    \caption{A real Heegaard diagram for $\Sigma_2(T_{2,1})\cong S^3$ built from a M\"obius band. The link $C$ is drawn in green. The diagrams for $\Sigma_2(T_{2,2n+1})\cong L(2n+1, 1)$ are similar; each twist of the M\"obius band introduces a new intersection point between $\alpha$ and $\beta$ on the fixed set.}
    \label{fig:mobius}
\end{figure}

\begin{example}
As a slightly more interesting example, consider the knot $K = 9_{46}$ shown in \Cref{fig:946}. By a computer calculation using the program \cite{bfh_python}, the branched double cover of $K$ is an L-space and has nine $\SpinC$-structures; each admits a single real structure. The knot $K$ has a genus 1, free Seifert surface $F$. A regular neighborhood of $F$ is isotopic to the standard genus 2 handlebody in $S^3$, and hence its complement is also a handlebody; to see this, simply twist the two central feet clockwise $3\pi$ radians, undoing the three twists shown in \Cref{fig:946}. Choose alpha curves which bound disks in the complement. By reversing this isotopy, we obtain the alpha curves shown on the right hand side of  \Cref{fig:946}. To obtain the beta curves, we apply $\tau$ to the alpha curves. 
\begin{figure}
    \center
    \includegraphics[width=0.8\linewidth]{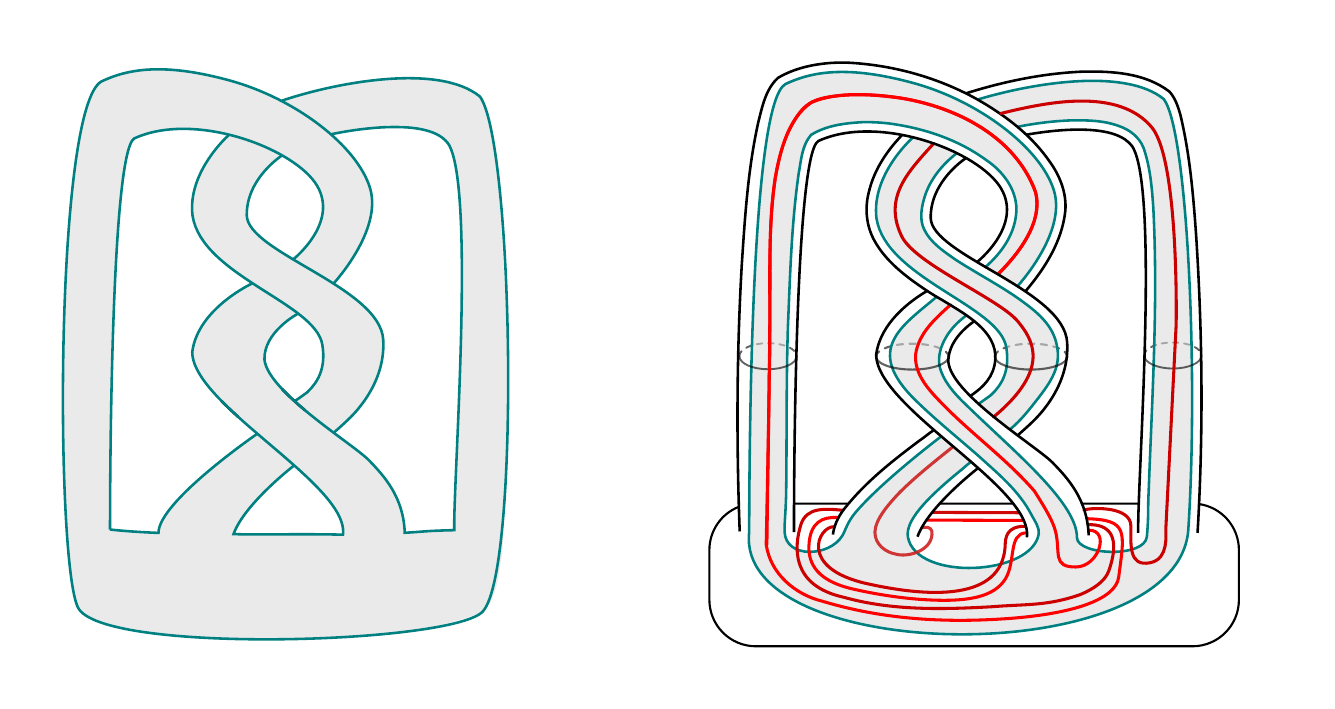}
    \caption{Building a real Heegaard diagram for the branched cover of the knot $9_{46}$ from a free Seifert surface. The two red curves bound disks in the complement of the thickened surface.}
    \label{fig:946}
\end{figure}
\begin{figure}[h]
\def\svgwidth{.8\linewidth}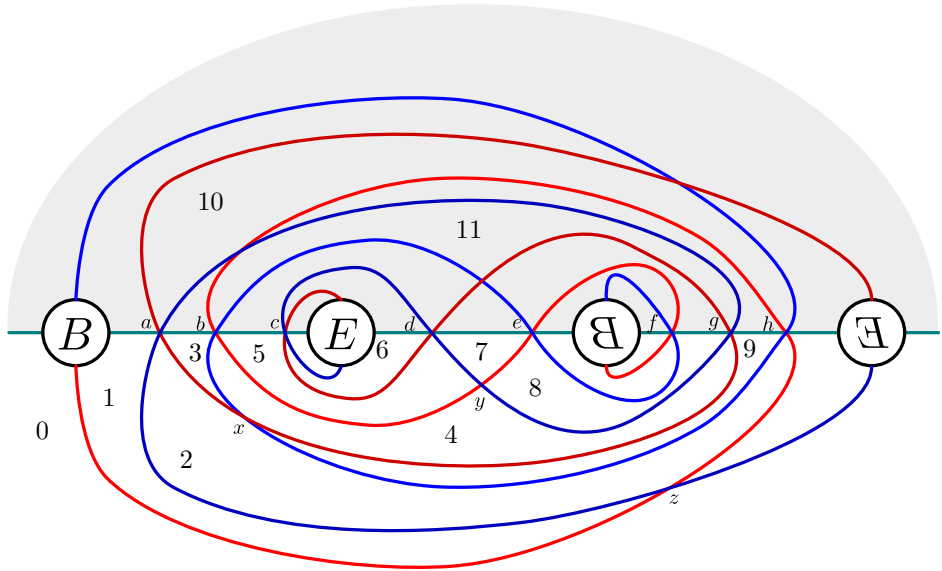
\caption{A real Heegaard diagram for the branched cover of the knot $9_{46}$.}
    \label{fig:946_planar}
\end{figure}

\Cref{fig:946_planar} is a planar representation of the resulting diagram; the intersection points between curves are labeled by letters and the regions are labeled by integers. We place the basepoint in the region $R_1$ and focus on the hat version; that is, we will only study holomorphic curves in domains that have multiplicity zero at $R_1$. 

The resulting complex has 18 generators:
\begin{align*}
ab, ae, af, ah, bc, bd, bg, ce, cf, ch,\\ de, df, dh, eg, fg, x\tau(x), y\tau(y), z\tau(z)   
\end{align*}
split according to real $\SpinC$ structures into
\begin{align*}
\{ab, \  gh,\ x\tau(x)\},\{ah,de, y\tau(y)\}, \{bg,\ cf, \ z\tau(z)\}, \\
\{af, \ bd, cd\}, \{ch, \ df, \ eg\}, \{bc\}, \{dh\},\{fg\}, \{ae\}
\end{align*} 

There are actually only a few domains of real index 1; most are bigons, rectangles, or annuli, and therefore have unique holomorphic representatives. A straightforward computation shows that there are nine real $\SpinC$-structures. There are no index 1 domains in or out of the generators 
\begin{align*}
    ae, ah, bc, dh, fg
\end{align*}
which therefore support the homology in five of the nine real $\SpinC$-structures. There are two rectangular index 1 domains into $x\tau(x)$, coming from $ab$ and $gh$. Hence, $ab + gh$ generates another summand. There is another rectangle from $y\tau(y)$ to $de$; this is the only index 1 domain interacting with these generators, so neither contribute to the homology. The last index 1 rectangle is $R_0$ and goes from $z\tau(z)$ to $cf$. There is another index 1 domain into $cf$:  an immersed annulus $A = R_5 + 2R_7 + R_8$ from $bg$ to $cf$. Whether this class admits a holomorphic representative or not, the homology is rank 1 in this real $\SpinC$-structure.

This leaves two summands for which we have not accounted. There are three index 1 annular domains, $R_0 + R_6$, $R_5 + R_7$, and $R_7 + R_8$, from $ch$ to $df$, $bd$ to $ce$, and $eg$ to $df$ respectively, which admit holomorphic representatives. Hence, $ch + eg$ generates the eighth summand. There is one last complicated domain from $af$ to $bd$ (the domain is $\Sigma + R_5+R_7-R_0-R_1$).  Whether or not this domain admits a holomorphic representative, the rank of homology in the last real $\SpinC$-structure is 1.

Therefore, we have that
\begin{align*}
    \widehat{\HFR}(\Sigma_2(K), \tau, \frs) \cong \F
\end{align*}
for each real $\SpinC$-structure $\frs$. 

While there are more complicated domains which do cover the basepoint, $\HFR^-(\Sigma_2(K), \tau, \frs)$ is determined algebraically by $\widehat{\HFR}(\Sigma_2(K), \tau, \frs)$; indeed, $\HFR^-(\Sigma_2(K), \tau, \frs) \cong \F[U]$ for each real $\SpinC$-structure. 
\end{example}

\begin{remark}
We conjecture that $\HFh(Y,\s)$ is  related to $\HFRh(Y, \tau, \s)$ by a localization spectral sequence, similar to those in \cite{seidel-smith}, \cite{hendricks-covers}, \cite{LM-covers}. If that were the case, then we would have  $\widehat{\HFR}(\Sigma_2(K),\tau, \frs) \cong \F$ whenever $\Sigma_2(K)$ is an L-space.
\end{remark}

\section{The Euler characteristic}\label{sec:euler}

When $L_0$ and $L_1$ are oriented Lagrangians in a symplectic manifold $M$, the Floer group $\CF(L_0, L_1)$ inherits a $\Z/2$-grading, and its Euler characteristic is given by the algebraic intersection of these Lagrangians in $M$ (see \Cref{subsec:grading}.) Recall from Section~\ref{subsec:grading} that in real Heegaard Floer theory, the Lagrangian $M^R$ is not always orientable. Nevertheless, we can ensure orientability if we just consider the hat theory and assume the following:
\begin{itemize}
\item the image of $C$ in the quotient $Y/\tau$ is null-homologous, and
\item every component of $C=Y^{\tau}$ contains at least one basepoint. (For simplicity, we will assume that it contains exactly one.)
\end{itemize}
In that case, we explain below how to compute the grading and the Euler characteristic explicitly.

\subsection{Signs of intersection points}
 Let $L$ be a null-homologous link in a 3-manifold $X$.  Let $Y$ be a branched double cover of $X$ along $L$ and write $\tau$ for the branching involution. By the construction in \Cref{sec: real HD}, the real manifold $(Y, \tau)$ can be represented by a real Heegaard diagram $(\cH, \tau, \w) = (\Sigma, \bm{\alpha},\bm{\beta},\w,\tau)$ whose quotient $\Sigma' = \Sigma/\tau$ is an orientable surface with boundary. We identify $L$ with its lift $C \subset Y$.
 

Since we assume that $\w$ contains exactly one basepoint on each component of $C$, there are $$m = 2h+2(l-1)$$ alpha curves, where $h$ is the genus of $\Sigma'$ and $l$ is the number of components of $C$. Further, as we are interested in the hat theory, we remove the divisor $\w \times \Sym^{2g-1}(\Sigma)$ from $M$; we denote its complement by $\hM$. 
  
\begin{lemma}
\label{lem:MRo}
Under the assumptions above, the fixed set $\hMR=M^R \cap \hM$ is orientable.  
\end{lemma}

\begin{proof}
The description of $M^R$ in Section~\ref{sec:Sym} gives a decomposition of $\hMR$ into strata of the form
\begin{equation}
\label{eq:stratum}
 \Sym^k(\Sigma') \times \Sym^{m-2k}(\hC),
 \end{equation}
where $\hC = C \setminus \w$ is a collection of intervals.  To orient $\hMR$, we need to orient these strata in a compatible way.

We start by orienting $\Sigma'$. Since $\tau$ is orientation reversing, the orientation on $\Sigma$ does not determine one on $\Sigma' = (\Sigma\smallsetminus C)/\tau$. Rather, we choose an arbitrary orientation on $\Sigma'$ and then orient $\hC$ as its boundary (using the ``outward normal first'' convention). Further, the orientation on $\Sigma'$ induces one on the symmetric products $\Sym^k(\Sigma')$.

Let us also pick an order of the components of $C$. Then, the symmetric products $\Sym^{m-2k}(\hC)$ get an induced orientation from that on $\hC$ and the order, as follows. Note that $\Sym^{m-2k}(\hC)$ is a manifold with corners, and it suffices to orient its interior. A point in its interior is a tuple
$$ c=(c_1, \dots, c_{m-2k})$$
where all $c_j$ are distinct. We require that the points  $c_1, \dots,  c_{m-2k}$ appear on $C$ in this order, as we go on the components of $C$ in the chosen order, and on each component we go from the basepoint $w$ to itself following the orientation of $\hC$. We choose disjoint intervals $I_j \subset \hC$ around the $c_j$; the orientation of $\hC$ restricts to orientations on each of these intervals. A neighborhood of $c$ in $\Sym^{m-2k}(\hC)$ is identified with the product $I_1 \times \dots \times I_{m-2k}$. We equip it with the product orientation multiplied by the sign $$(-1)^{h+l-k-1}.$$

We have oriented $\Sym^k(\Sigma')$ and $\Sym^{m-2k}(\hC).$ By taking their product orientation, we orient each stratum \eqref{eq:stratum}. We need to check that these orientations are compatible along boundaries, and hence glue to give an orientation on $\hMR$. It suffices to check pairwise compatibility along codimension one parts of the boundary. Consider gluing $\Sym^k(\Sigma') \times \Sym^{m-2k}(\hat{C})$ to $\Sym^{k-1}(\Sigma') \times \Sym^{m-2(k-1)}(\hat{C})$. Two points of $\hC$ coalesce to glue to a limit of points $(z, \tau(z))$, with $z \in \Sigma'$ approaching $\hC$. The local model is $\Sym^2([0,1])$, which we identify with the triangle
$$\{(x, y) \in [0,1]^2 \mid x \leq y\}.$$
The orientations on $\Sigma'$ and $\hC$ are related by the ``outward normal first'' convention, so near a point $(x, x)$ on the diagonal we have the clockwise orientation coming from $\Sigma'$. On $\Sym^2([0,1])$ we start with the natural counterclockwise orientation. However, on each stratum we need to multiply the orientation with a sign, and these signs are different for the two strata we are gluing: $(-1)^{h+l-k-1}$ versus $(-1)^{h+l-(k-1)-1}$. We deduce that the orientations are compatible. 
\end{proof}

\begin{remark}
Pulling back the orientation on $\Sigma'$ to its double cover $\Sigma \smallsetminus C$ splits the latter into two disjoint parts: $\Sigma^+$, where the orientation pulled back from $\Sigma'$ agrees with the one on $\Sigma$, and $\Sigma^-=\tau(\Sigma^+)$, where it disagrees. In particular, $\Sigma \smallsetminus C$ is disconnected.
\end{remark}

\begin{remark}
    We note that in the hat flavor of Heegaard Floer homology, one works with exact symplectic manifolds and exact Lagrangians. The same is true for hat real Heegaard Floer homology, as we now sketch. As in  \cite[Sections 4]{hendricks-covers}, one may choose an anti-invariant exhausting function $\psi: M \to \R$ such that $\omega = -dd^\C \psi$ is a symplectic form for $M$ so that $(M, \omega)$ is convex at infinity. Furthermore, by averaging if necessary, we may arrange that $R$ is anti-symplectic with respect to this form. Following Perutz \cite{perutz} as we did \Cref{sec:def}, we arrange that away from the diagonal, $\omega$ is induced by a symmetric area form on $\Sigma$. Following \cite[Proposition 4.2]{HLL_rankHF_BDC}, $-d^\C \psi$ can be modified to ensure that $\Ta$ is exact (preserving the symmetry). The symmetry forces $M^R$ to be exact.
\end{remark}

Let us also orient the other Lagrangian $\Ta$. For this, choose an ordering of the alpha curves $\alpha_1, \dots, \alpha_{m}$, and orient the curves $\alpha_i$ arbitrarily. This determines an orientation of their product. The involution $\tau$ then determines an ordering of the beta curves as well as an orientation, hence this also induces an orientation on $\Tb$.

Let $\x$ be an element in $\Ta \cap M^R$. Hence, $\x$ is contained in one of the strata of $\hMR$ of the form
$$\x=\{ z_1, \tau(z_1), \dots, z_k, \tau(z_k), c_1, \dots, c_{m-2k} \} \in \Sym^k(\Sigma') \times \Sym^{m-2k}(\hC),$$
where
$$ z_i \in \alpha_{r(i)} \cap \beta_{\sigma(r(i))} \cap \Sigma^+,\ \ \tau(z_i) \in \alpha_{\sigma(r(i))} \cap \beta_{r(i)} \cap \Sigma^-, \ \ i=1, \dots, k$$
and 
$$c_j \in \alpha_{s(j)} \cap \beta_{s(j)}, \ j=1, \dots, m-2k. $$
Here, $\sigma$ is an involution on the set $\{1, 2, \dots, m\}$ with fixed points $s(1), \dots, s(m-2k)$. The 2-cycles of $\sigma$ are $(r(i), \sigma(r(i))$, for $i=1, \dots, k$, for $r: \{1, \dots, k\} \to \{1, \dots, m-2k\}$ an injection. Note that, as a permutation, $\sigma$ has sign $(-1)^k$. We require that the points  $c_1, \dots,  c_{m-2k}$ appear on $\hC$ in order as specified in the proof of Lemma~\ref{lem:MRo}: we first arrange them according to the order of the components of $C$, and then on each component we follow the orientation of $\hC$.

For each $i=1, \dots, k$, the projections of the oriented curves  $\alpha_{r(i)}$ and $\beta_{\sigma(r(i))}$ to $\Sigma'$ intersect at $[z_i]$ with a sign $\epsilon(z_i) \in \{\pm 1\}$; here, we use here the orientation of $\Sigma'$. (Alternatively, we can think of $\epsilon(z_i)$ as the sign of intersection between $\alpha_{r(i)}$ and $\beta_{\sigma(r(i))}$ in $\Sigma^+$.) Similarly, for each $j=1, \dots, m-2k$, the oriented curves $\alpha_{s(j)}$ and $C$ intersect at $c_j$ with a sign $\epsilon(c_j)$, using the orientation on $\Sigma$.

Let also $\epsilon(r, \sigma, s)$ be the sign of the following permutation of $\{1, \dots, m\}$:
$$(r(1), \sigma(r(1)), r(2), \sigma(r(2)), \dots, r(k), \sigma(r(k)), s(1), \dots, s(m-2k)).$$

Combining all of this, we compute the sign of an intersection point of $\Ta\cap M^R$ as follows.
\begin{proposition}
The sign of $\x$, as an intersection point between $\Ta$ and $\hMR$, is given by the formula
\begin{equation}
    \label{eq:signs}
 \sgn(\x) :=    \epsilon(r, \sigma, s) \cdot \prod_{i=1}^k \epsilon(z_i) \cdot \prod_{j=1}^{m-2k} \epsilon(c_j). 
 \end{equation}
\end{proposition}

\begin{proof}
We need to compare the direct sum orientation on $T_{\x}\Ta \oplus T_{\x}\hMR$ with the orientation on $T_\x \Sym^m(\Sigma)$ coming from $\Sigma$.

The orientation on $T_{\x}\Ta$ differs from the one on the direct sum
\begin{equation}
    \label{eq:alphao}
 \bigoplus_{i=1}^k T_{\{z_i, \tau(z_i)\}} (\alpha_{r(i)}\times\alpha_{\sigma(r(i))}) \oplus  \bigoplus_{i=1}^{m-2k} T_{c_i} \alpha_{s(i)}
 \end{equation}
by a sign of $\epsilon(r, \sigma, s)$, which comes from reordering the alpha curves. 

The orientation on $T_{\x}\hMR$ differs from the one on the direct sum
\begin{equation}
    \label{eq:mor}
\bigoplus_{i=1}^k T_{\{z_i, \tau(z_i)\}} \Sym^1(\Sigma') \oplus \bigoplus_{i=1}^{m-2k}T_{c_i} C
\end{equation}
by a sign of $(-1)^{h+l-k-1}$, which was introduced in the proof of Lemma~\ref{lem:MRo}.

Putting these together, we find that the orientation on $T_{\x}\Ta \oplus T_{\x}\hMR$ differs from that on
\begin{equation}
    \label{eq:alphamoro}\bigoplus_{i=1}^k \Bigl( T_{\{z_i, \tau(z_i)\}} (\alpha_{r(i)}\times\alpha_{\sigma(r(i))}) \oplus T_{\{z_i, \tau(z_i)\}} \Sym^1(\Sigma') \Bigr) \oplus \bigoplus_{i=1}^{m-2k} \Bigl( T_{c_i} \alpha_{s(i)} \oplus T_{c_i} C\Bigr)
    \end{equation}
by a sign of $\epsilon(r, \sigma, s)$. Indeed, the other sign $(-1)^{h+l-k-1}$ 
cancels with the one coming from interleaving the $m-2k$ terms of the form  $T_{c_i} \alpha_{s(i)}$ with those of the form $T_{c_i} C$:
$$ (-1)^{(m-2k)(m-2k-1)/2} = (-1)^{h+l-k-1}.$$
The other re-orderings needed to go from the direct sum of \eqref{eq:alphao} and \eqref{eq:mor} to \eqref{eq:alphamoro} produce positive signs, because they involve two-dimensional vector spaces.

We are left to compare \eqref{eq:alphamoro} with the orientation on $T_\x \Sym^m(\Sigma)$, i.e., to the direct sum orientation on 
$$\bigoplus_{i=1}^k \Bigl( T_{\{z_i, \tau(z_i)\}}  \Sym^2(\Sigma) \Bigr) \oplus \bigoplus_{i=1}^{m-2k}  T_{c_i} \Sigma .
  $$

At a point $\{z_i, \tau(z_i)\}$, comparing the orientation on $T_{\{z_i, \tau(z_i)\}} (\alpha_{r(i)}\times\alpha_{\sigma(r(i))}) \times T_{\{z_i, \tau(z_i)\}} \Sym^1(\Sigma')$ with that on $T_{\{z_i, \tau(z_i)\}} \Sym^2(\Sigma)$ yields the sign $\epsilon(z_i) $. 
     Similarly, at a point $c_i \in \x \cap C$, the orientation on $T_{c_i} \alpha_{s(i)}\times T_{c_i} C$ differs from the one on $T_{c_i}\Sigma$ by the sign $\epsilon(c_j)$.  

     Overall, we get the expression \eqref{eq:signs}.
\end{proof}

\subsection{Two relative gradings}
In light of the previous section, the group $\HFRh(\cH, \tau, \w,\frs)$ can be equipped with a relative $\Z/2$-grading in two ways. On the one hand, we can compare signs of intersection points between $\Ta$ and $\hMR$:
\begin{align*}
    \widehat{\gr}(\x, \y) = \sgn(\x)\cdot \sgn(\y).
\end{align*}

On the other hand, if $\frs_\w^R(\x) = \frs_\w^R(\y) \in \RSpinC(Y, \tau)$ and $\delta(\s_\w(\x))$ is divisible by $4$ (cf. \Cref{subsec:grading}),  there is also a relative Maslov grading on $\CFR^-(\cH, \tau, \w, \frs)$, given by
\begin{align*}
    \mu^R(\x, \y) = \mu^R(\phi) - n_\w(\phi) \mod 2,
\end{align*}
for $\phi \in \pi_2^R(\x,\y)$. This descends to a relative grading on $\HFRh(\cH, \tau, \w, \frs)$. 

 First, we emphasize that it is not clear that $\widehat{\gr}$ is independent of the choice of diagram. To define $\widehat{\gr}$, we must choose a real Heegaard surface with orientable quotient. However, we might not be able to connect any two such real Heegaard splittings through a sequence of splittings with this property. In fact, we expect that there is a dependence on the homology class $[\Sigma'] \in H_2(Y/\tau, L)$.

Second, as noted in \Cref{sec:relative}, the natural map
\begin{equation}
     \RSpinCrel(Y, \tau, \w) \to \RSpinC(Y, \tau)
\end{equation}
need not be injective. Thus, not every pair of intersection points $\x$ and $\y$ which determine the same real $\SpinC$-structure can be connected by a real invariant domain that misses $\w$. The real Maslov index does not give an intrinsic way to compare intersection points in different relative real $\SpinC$-structures. 

\begin{remark}
    \label{rem:oneC}
When $C$ has a single connected component, these two relative gradings do agree, because the map 
\begin{equation}
     \RSpinCrel(Y, \tau, w) \to \RSpinC(Y, \tau)
\end{equation}
is an isomorphism (see \Cref{sec:relative}). In this case, it follows from the invariance of $\HFRh$ that $\grhat$ is independent of the choice of diagram. We can then make use of the fact that $\grhat$ is defined on all of $\HFRh(Y, \tau, \w)$, rather than on $\HFRh(Y, \tau, \w, \frs)$ for particular real $\SpinC$ structures.
\end{remark}
 
In general, we have two candidate Euler characteristics for the theory. For $\s \in \RSpinC(Y, \tau)$, we may define the Euler characteristic of the hat theory by
\begin{equation}
\label{eq:chiHFR}
\widehat{\chi}(\widehat{\HFR}(\cH, \tau, \w, \s)) := \sum_{\substack{\x \in \Ta \cap M^R \\ \s_\w^R(\x) =\s}} \sgn(\x).
\end{equation}
This is well-defined only up to a sign, because it depends on the chosen orientations of $\Ta$ and $\hMR$. (More precisely, it  depends on the resulting orientation on their product $\Ta \times \hMR$.) Alternatively, we may define
\begin{equation}
\label{eq:chiHFR}
\chi(\widehat{\HFR}(\cH, \tau, \w, \s)) := \sum_{\substack{\x \in \Ta \cap M^R \\ \s_\w^R(\x) =\s}} (-1)^{\mu^R(\x,\x_0)},
\end{equation}
for a fixed intersection point $\x_0$ realizing $\s$. This is, of course, also only well-defined up to sign. In general, these Euler characteristics need not agree, as the following example demonstrates.

\begin{example}\label{ex: different euler characteristics}
    Consider again the real Heegaard diagram for $(\cH, \tau) = (S^1 \times S^2, \tau)$, where $\tau = \id \times \mathrm{rot}$, from \Cref{ex: s1xs2 rotation}. The real manifold $(S^1 \times S^2, \tau)$ has two real $\SpinC$-structures, $\frs_1$ and $\frs_2$. Add a second basepoint together with a symmetric pair of alpha and beta curves as in \Cref{fig:ex_relative_spinC}. Then $\HFRh(\cH, \tau, \w)$ is four-dimensional, generated by $xy$, $xz$, $vy$, and $vz$. Here, $xy$ and $xz$ are in a real $\SpinC$-structure $\frs_1$, whereas $vy$ and $vz$ are in the other real $\SpinC$-structure $\frs_2$. All four are in different relative real $\SpinC$-structures. 

    There is a real invariant bigon of index one from $xz$ to $xy$, covering a single basepoint. Hence, 
    \begin{align*}
        \mu^R(xz, xy) = 0 \mod 2,
    \end{align*}
    and therefore $\chi(\widehat{\HFR}(\cH, \tau, \w, \s_1)) = 2$. Similarly, $\chi(\widehat{\HFR}(\cH, \tau, \w, \s_2)) = 2$. 

    Things are more interesting when computing $\widehat{\chi}$. The signs of $vy$ and $vz$ disagree (since $\ep(y)$ and $\ep(z)$ disagree), and therefore $\widehat{\chi}(\widehat{\HFR}(\cH, \tau, \w, \s_2)) = 0$. However, the signs of $xy$ and $xz$ agree. Indeed, while $\ep(y)$ and $\ep(z)$ disagree, the signs of the permutation $\ep(r, \sigma, s)$ for the two intersection points disagree as well: This is because on the component if $C$ containing $x, y,$ and $z$, the point $x$ separates $y$ and $z$ when one traverses the fixed point set starting at the basepoint.  Therefore, in this case we have  $\widehat{\chi}(\widehat{\HFR}(\cH, \tau, \w, \s_2)) = 2$.
\end{example}

\begin{figure}[h]
\def\svgwidth{.7\linewidth}
\begingroup%
  \makeatletter%
  \providecommand\color[2][]{%
    \errmessage{(Inkscape) Color is used for the text in Inkscape, but the package 'color.sty' is not loaded}%
    \renewcommand\color[2][]{}%
  }%
  \providecommand\transparent[1]{%
    \errmessage{(Inkscape) Transparency is used (non-zero) for the text in Inkscape, but the package 'transparent.sty' is not loaded}%
    \renewcommand\transparent[1]{}%
  }%
  \providecommand\rotatebox[2]{#2}%
  \newcommand*\fsize{\dimexpr\f@size pt\relax}%
  \newcommand*\lineheight[1]{\fontsize{\fsize}{#1\fsize}\selectfont}%
  \ifx\svgwidth\undefined%
    \setlength{\unitlength}{566.92913386bp}%
    \ifx\svgscale\undefined%
      \relax%
    \else%
      \setlength{\unitlength}{\unitlength * \real{\svgscale}}%
    \fi%
  \else%
    \setlength{\unitlength}{\svgwidth}%
  \fi%
  \global\let\svgwidth\undefined%
  \global\let\svgscale\undefined%
  \makeatother%
  \begin{picture}(1,0.55)%
    \lineheight{1}%
    \setlength\tabcolsep{0pt}%
    \put(0,0){\includegraphics[width=\unitlength,page=1]{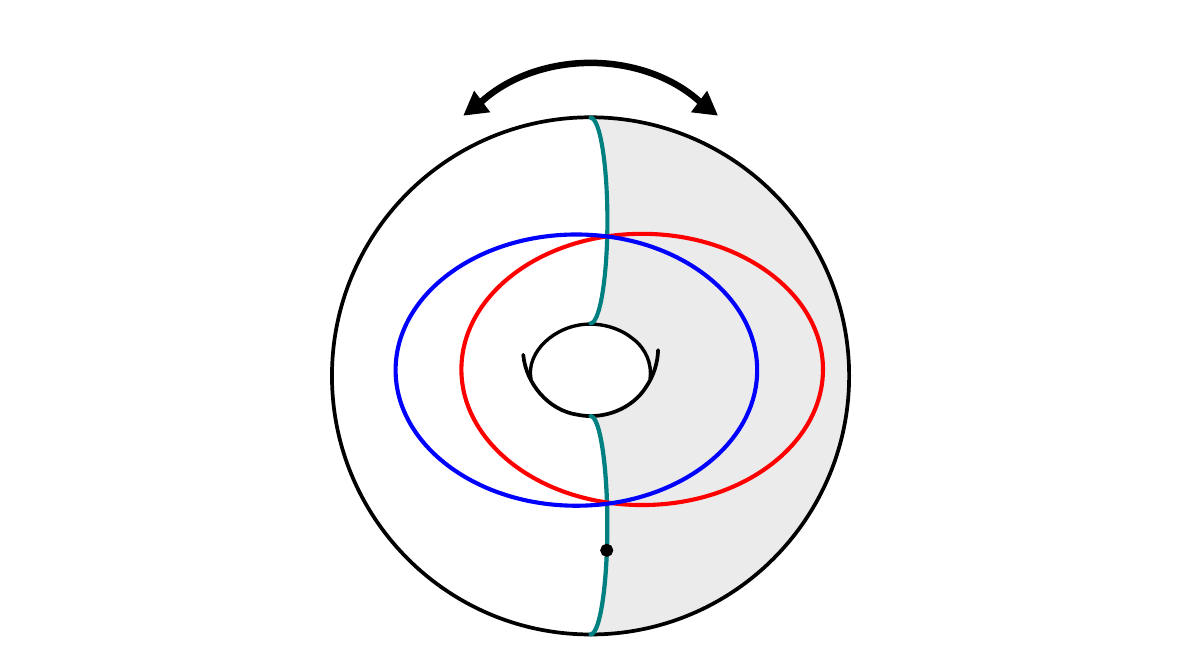}}%
    \put(0.49194553,0.51743837){\color[rgb]{0,0,0}\makebox(0,0)[lt]{\smash{\begin{tabular}[t]{l}{$\tau$}\end{tabular}}}}%
    \put(0,0){\includegraphics[width=\unitlength,page=2]{ex_relative_spinC.pdf}}%
  \end{picture}%
\endgroup%

\caption{A doubly pointed real Heegaard diagram for $(S^1 \times S^2, \tau)$.}
    \label{fig:ex_relative_spinC}
\end{figure}

As noted above, $\grhat$ has the advantage that is defined on all of $\HFRh(\cH, \tau, \w)$, rather than on $\HFRh(\cH, \tau, \w, \frs)$ for particular real $\SpinC$ structures. It follows that we can also sum over all intersection points, to obtain a ``total'' Euler characteristic. Again, this is well-defined up to a sign. 

\subsection{Double branched covers}
Let us specialize the above discussion to the case of double branched covers over knots $K \subset S^3$. Since in this case each $\SpinC$-structure on $\Sigma_2(K)$ has a unique real $\SpinC$ structure, we will not distinguish between the two concepts. Further, as noted in Remark~\ref{rem:oneC}, since we are working with knots rather than links, the two relative $\Z/2$-gradings on $\HFRh$ agree. Hence, the two notions of Euler characteristic ($\chi$ and $\chihat$) also agree.

Following the notation from the Introduction, for $\s \in \SpinC(\Sigma_2(K))$ we define
 $$ \chi_\s(K):= \chi(\widehat{\HFR}(\Sigma_2(K), \tau_K, \s)).$$

Because $\chi_\s(K)$ is a signed count of intersection points, and non-invariant intersection points come in pairs, the parity of $\chi_\s(K)$ is the same as that of $\chi(\widehat{\HF}(\Sigma_2(K),  \s))$, which is $1$ by \cite[Proposition 5.1]{os_properties_apps}. Thus, $\chi_\s(K)$ is an odd integer (so far, only defined up to sign).

 Let 
$$\chitot(K) := \sum_\s \chi_\s(K).$$
This is also odd, because the number of $\SpinC$ structures is  $\det(K)$, which is odd. Let us fix the sign of $\chitot(K)$ to be positive. This fixes an orientation on $\Ta \times \hMR$, and therefore fixes the signs of $\chi_\s(K) \in 2\Z+1$ for all $\s$.

We have written some computer code to compute $\chi_\s(K)$ using the formulas \eqref{eq:signs} and \eqref{eq:chiHFR}; see \cite{HFR_python}. Roughly, given a representation of $K$ as the closure of a braid, $B$, there is a canonical Seifert surface $\Sigma_B$. This surface is simple enough that the intersections between the alpha and beta curves on the associated real Heegaard diagram (as well as their signs of intersection) are determined by the braid $B$. It is then straightforward to compute the quantity \eqref{eq:signs} for each generator for the real Floer complex. 
Note that the genus of the resulting Heegaard splitting is twice that of the surface $\Sigma_B$, so $|\Ta \cap M^R|$ (and hence the computing time) grows rather quickly. 

We computed the Euler characteristic for all knots with up to $8$ crossings, as well as some with more crossings (see \cite{HFR_python} for a full list). A sample is shown in Table~\ref{tab:1}. 

\begin{table}[t]
    \centering

\renewcommand{\arraystretch}{1.3}
\setlength{\tabcolsep}{8pt}

\begin{tabular}{|c|c|c|c|}
\hline
Knot & $\det(K)$ & $\chitot(K)$ & $\chi_\frs(K)$ \\ 
\hline

\text{unknot} & $1$ & $1$ & [1] \\
$3_1$ & $3$ & $1$ & $[1, -1, 1]$ \\
$4_1$ & $5$ & $3$ & $[1, 1, -1, 1, 1]$ \\
$5_1$ & $5$ & $1$ & $[-1, 1, 1, 1, -1]$ \\
$5_2$ & $7$ & $3$ & $[1, 1, -1, 1, -1, 1, 1]$ \\
$6_1$ & $9$ & $5$ & $[1, 1, 1, -1, 1, -1, 1, 1, 1]$ \\
$6_2$ & $11$ & $1$ & $[1, 1, -1, 1, -1, -1, -1, 1, -1, 1, 1]$ \\
$6_3$ & $13$ & $3$ & $[1, -1, 1, 1, 1, -1, -1, -1, 1, 1, 1, -1, 1]$ \\
$7_1$ & $7$ & $1$ & $[-1, 1, 1, -1, 1, 1, -1]$ \\
$7_2$ & $11$ & $5$ & $[1, -1, 1, 1, 1, -1, 1, 1, 1, -1, 1]$ \\
$7_3$ & $13$ & $1$ & $[-1,1,1,1,-1,-1,1,-1,-1,1,1,1,-1]$ \\
$7_4$ & $15$ & $7$ & $[1,1,-1,1,-1,1,1,1,1,1,-1,1,-1,1,1]$ \\
$7_5$ & $17$ & $1$ & $[1,-1,1,1,1-1,-1,-1,1,-1,-1,-1,1,1,1,-1,1]$ \\
$7_6$ & $19$ & $5$ & $[-1, 1, -1, 1, 1, -1, 1, 1, 1, -1, 1, 1, 1, -1, 1, 1, -1, 1, -1]$ \\
$7_7$ & $21$ & $7$ & $[1, 1, -1, 1, 1, 1, 1, -1, 1, -1, -1, -1, 1, -1, 1, 1, 1, 1, -1, 1, 1]$ \\
$8_1$ & $13$ & $7$ & $[1, 1, 1, 1, -1, 1, -1, 1, -1, 1, 1, 1, 1]$ \\
$8_2$ & $17$ & $3$ & $[-1, -1, 1, -1, 1, 1, 1, 1, -1, 1, 1, 1, 1, -1, 1, -1, -1]$ \\
$8_3$ & $17$ & $9$ & $[1, 1, 1, -1, 1, -1, 1, 1, 1, 1, 1, -1, 1, -1, 1, 1, 1]$ \\
$8_4$ & $19$ & $19$ & $[1,1,-1,1,1,1,-1,1,-1,-1,-1,1,-1,-1,-1,1,-1,1,1,1,-1,1,1]$ \\
$8_5$ & $21$ & $1$ & $[1, -1, 1, 1, -1, -1, 1, -1, -1, 1, 1, 1, -1, -1, 1, -1, -1, 1, 1, -1, 1]$ \\
$8_6$ & $23$ & $3$ & $[1, 1,-1,1,1,1,-1, 1,-1,  -1,  -1,  1,-1, -1,  -1,1, -1,1,1,1, -1, 1,1]$ \\
$8_7$ & $23$ & $1$ & $[1, -1, 1, 1, 1, 1, -1, -1, -1, 1, -1, -1, -1, 1, -1, -1, -1, 1, 1, 1, 1, -1, 1]$ \\
$9_1$ & $9$ & $1$ & $[1, -1, -1, 1, 1, 1, -1, -1, 1]$ \\
$10_{152}$ & $11$ & $1$ & $[-1,-1,1, 1,-1,3,-1,1,1,-1,-1]$ \\
$12n_{121}$ & $1$ & $3$ & $[3]$ \\
$12n_{242}$ & $1$ & $3$ & $[3]$ \\
\hline
\end{tabular}

    \caption{Values of $\chitot$ and $\chi_\s$ for some small knots. The middle entry in the list of $\chi_\s$ corresponds to the spin structure.}
    \label{tab:1}
    
\end{table}

For a $\SpinC$-rational homology sphere $(Y, \frs)$, $\HFh(Y, \frs)$ always has Euler characteristic equal to one. This is notably false in the case of real Seiberg-Witten theory \cite{miyazawa,li:spectral}, and also for our theory. The smallest knot we found for which  $\chi_\s(K) \neq \pm 1$ for some $\s$ was $10_{152}$. We verified as well that for the pretzel knot $P(-2,3,7) = 12n_{242}$ we have $\chitot (P(-2,3,7)) = 3$, in agreement with the computation from Seiberg-Witten theory \cite{miyazawa}. A similar knot is $12n_{121}$, whose double branched cover is the Brieskorn sphere $\Sigma(2,3,11)$; just like $P(-2,3,7)$, this knot has determinant one and Heegaard Floer degree equal to $3$.

\begin{rem}
    In classical setting, $\chi(\HFh(Y, \frs))$ is independent of the $\SpinC$-structure $\frs$. This is because if $w$ and $w'$ are two basepoints which can be connected by a path which only intersects $\beta_i$ once, $\frs_w(\x)$ and $\frs_{w'}(\x)$ differ by the Poincar\'e dual of a curve $\beta_i^*$ which is geometrically dual to $\beta_i$. Hence, $\HFh(Y, \frs)$ and $\HFh(Y, \frs+\mathrm{PD(\beta_i^*}))$ can be calculated by the same intersection points. In our situation, we have less freedom, as we can only drag the basepoint along $C$; moreover, since $\alpha_i \cap C = \beta_i\cap C$, the basepoint necessarily crosses the curves in pairs, and so changes the $\SpinC$-structure by $\mathrm{PD(\alpha_i^* + \beta_i^*})$. This explains the dependence of $\chi_\s$ on $\s$.
\end{rem}

\begin{rem}
    It follows from \Cref{ex:torus} that $|\chi_\frs(T_{2,2n+1})| = 1$ for every real $\SpinC$-structure. However, the diagrams we used there have non-orientable quotients. Therefore, even though the real Floer complex is computable, it is not clear how to pin down the signs, nor how to compute $\chitot(T_{2,2n+1})$ from this diagram. 
\end{rem}

\bibliographystyle{amsalpha}
\bibliography{mathbib}
\end{document}